\documentclass[nonacm]{acmart}
\acmSubmissionID{}

\AtBeginDocument{%
  \providecommand\BibTeX{{%
    \normalfont B\kern-0.5em{\scshape i\kern-0.25em b}\kern-0.8em\TeX}}}

\usepackage[latin1]{inputenc}
\usepackage{amsmath, amsxtra,amsfonts, amssymb}
\usepackage{hyperref}
\usepackage{xspace}
\usepackage{geometry} 
\usepackage{framed}
\usepackage{todonotes}
\usepackage{enumerate}

\newcommand{\ignore}[1]{}
\newcommand{\N}{\mathbb N}
\newcommand{\Q}{\mathbb Q}
\newcommand{\Z}{\mathbb Z}
\newcommand{\bA}{{\mathfrak A}}
\newcommand{\bB}{{\mathfrak B}}
\newcommand{\bC}{{\mathfrak C}}

\newcommand{\bH}{{\mathfrak H}}

\newcommand{\bM}{{\mathfrak M}}
\newcommand{\bN}{{\mathfrak N}}

\DeclareMathOperator{\Can}{Can}
\DeclareMathOperator{\PC}{PC}
\DeclareMathOperator{\Cyc}{Cyc}

\DeclareMathOperator{\Imp}{Imp}
\DeclareMathOperator{\Aut}{Aut}

\DeclareMathOperator{\fPol}{fPol}
\DeclareMathOperator{\VCSP}{VCSP}
\DeclareMathOperator{\CSP}{CSP}

\DeclareMathOperator{\Pol}{Pol}
\DeclareMathOperator{\Sup}{Supp}
\DeclareMathOperator{\Opt}{Opt}
\DeclareMathOperator{\Id}{Id}
\DeclareMathOperator{\Feas}{Feas}

\DeclareMathOperator{\NAE}{NAE}
\DeclareMathOperator{\OIT}{OIT}
\DeclareMathOperator{\Age}{Age}

\newtheorem{theorem}{Theorem}[section]
\newtheorem{definition}[theorem]{Definition}
\newtheorem{proposition}[theorem]{Proposition}
\newtheorem{lemma}[theorem]{Lemma}
\newtheorem{corollary}[theorem]{Corollary}
\newtheorem{conjecture}[theorem]{Conjecture}
\newtheorem{expl}[theorem]{Example}

\newtheorem{remark}[theorem]{Remark}

\newtheorem{question}[theorem]{Question}
\newtheorem{claim}[theorem]{Claim}

\begin{document}

\title{The Complexity of Resilience Problems via Valued Constraint Satisfaction}

\author{Manuel Bodirsky}
\email{manuel.bodirsky@tu-dresden.de}
\orcid{0000-0001-8228-3611}
\author{\v{Z}aneta Semani\v{s}inov\'{a}}
\email{zaneta.semanisinova@tu-dresden.de}
\orcid{0000-0001-8111-0671}
\affiliation{%
  \institution{Institut f\"{u}r Algebra, TU Dresden}
  \streetaddress{Zellescher Weg 12-14}
  \city{Dresden}
  \country{Germany}
  \postcode{01069}
}

\author{Carsten Lutz}
\orcid{0000-0002-8791-6702}
\affiliation{%
  \institution{Institut f\"{u}r Informatik, Leipzig University}
  \streetaddress{Augustusplatz 10}
  \city{Leipzig}
  \country{Germany}
  \postcode{04109}}
\email{clu@informatik.uni-leipzig.de}

\renewcommand{\shortauthors}{Manuel Bodirsky, \v{Z}aneta Semani\v{s}inov\'{a} and Carsten Lutz}

\begin{abstract}
Valued constraint satisfaction problems (VCSPs) constitute a large class of computational optimization problems. It was shown recently that, over finite domains, every VCSP is in P or NP-complete, depending on the admitted cost functions. In this article, we study cost functions over countably infinite domains whose automorphisms form an oligomorphic permutation group. Our results include a hardness condition based on a generalization of pp-constructability as known from classical CSPs and a polynomial-time tractability condition based on the concept of fractional polymorphisms.  We then observe that the resilience problem for unions of conjunctive queries (UCQs) studied in database theory, under bag semantics, may be viewed as a special case of the VCSPs that we consider. We obtain a complexity dichotomy for the case of incidence-acyclic UCQs and exemplarily use our methods to determine the complexity of a conjunctive query that has been stated as an open problem in the literature. We conjecture that our hardness and tractability conditions match for resilience problems for UCQs. Further, we obtain a complete dichotomy for resilience problems for two-way regular path queries, under bag semantics.
\end{abstract}





\maketitle

\section{Introduction}
 In data management, the \emph{resilience} of a query $\mu$ in a 
 relational database $\bA$ is the minimum number of tuples that need to be removed from $\bA$ to achieve that $\mu$ is false in $\bA$. This associates every fixed query $\mu$ with a natural decision
 problem:
given a database $\bA$ and a candidate resilience  $n \in \N$, decide whether the
resilience of $\mu$ in $\bA$ is at most $n$.
 Significant efforts have been invested into classifying the  
complexity of such resilience problems depending on the query~$\mu$, concentrating on the case that $\mu$ is a conjunctive query~\cite{Resilience,NewResilience, LatestResilience}. Notably,
research has identified several classes
of conjunctive queries for which the 
resilience problem is in polynomial time and others for which it is NP-complete.
A general classification, however, has remained open. 

Resilience problems have been considered 
 on set databases~\cite{Resilience,NewResilience} and, more recently, also on bag databases~\cite{LatestResilience}.
The latter means that every fact in the 
database is associated with a multiplicity,
that is, a tuple of constants may have multiple occurrences in the same relation.
Bag databases are of importance because they represent SQL databases
 more faithfully than set databases~\cite{bag-semantics}. Note that if the resilience problem of a query $\mu$ can be solved in polynomial time on bag databases, then it can be solved in polynomial time also on set databases. 
Regarding the converse, Makhija and Gatterbauer \cite{LatestResilience} identify a conjunctive query for which the resilience problem on bag databases is NP-hard whereas the resilience problem on set databases is in P.

In this article, we present a surprising link between the resilience problem for (unions of) conjunctive queries under bag semantics
and \emph{valued constraint satisfaction problems (VCSPs)}, which constitute 
a large class of computational optimization problems. In a VCSP, we are given a finite set of variables, a finite sum of cost functions on these variables, and a threshold $u$, and the task is to 
find an assignment to the variables so that the sum of the costs is at most $u$.
The computational complexity of such problems has been studied depending on the admitted cost functions, which 
we may view as a \emph{valued structure}. 
A complete classification has been obtained for valued structures with a finite domain, showing that the corresponding VCSPs are in P or NP-hard~\cite{KozikOchremiak15,KolmogorovKR17,BulatovFVConjecture,ZhukFVConjecture,Zhuk20}. There are also some  results about VCSPs of valued structures with infinite domains~\cite{BodirskyMaminoViola-Journal,ViolaZivny}.

We show that the resilience problem for every union of connected conjunctive queries
can be formulated as a VCSP for a valued structure with an \emph{oligomorphic automorphism group}, i.e., a structure with a countable domain that, for every fixed $k$, has only finitely many orbits of $k$-tuples under the action of the automorphism group. This property is important for classical CSPs (which can be seen as VCSPs where all cost functions take values in 
$\{0,\infty\}$) since it enables the use
and extension of some tools from finite-domain CSPs (see, e.g.,~\cite{Book}). 
The complexity classification for unions of general, not necessarily connected, conjunctive queries can be reduced to the connected case. 
In the important special case that every conjunctive query in the union is incidence-acyclic (meaning that it has an acyclic incidence graph), we even obtain a VCSP for a valued structure with a finite domain and consequently obtain a P versus NP-complete dichotomy from the known dichotomy for such VCSPs. In addition, we consider the resilience problem for
two-way regular path queries (2RPQs), a fundamental query language for graph databases that 
is orthogonal in expressive power to unions of conjunctive queries \cite{DBLP:journals/sigmod/CalvaneseGLV03,DBLP:conf/icdt/MartensT18}. We observe that the resilience problem for any 2RPQ can also be
translated into a VCSP for a valued structure with a finite domain, thus obtaining a P versus NP-complete dichotomy 
in this case.

The above results actually 
hold in the more general setting where some relations or tuples may be declared to be \emph{exogenous}, meaning that they may \emph{not} be removed from the database to make the query false. This is useful when the data in the database stems from different sources; it was also considered in~\cite{Resilience,NewResilience, LatestResilience}.

As a main contribution of this article, the novel connection between resilience 
and VCSPs leads us to initiating the systematic study of VCSPs of countably infinite valued structures whose automorphisms form an oligomorphic permutation group. In particular, we develop a notion of \emph{expressive power} which is based on \emph{primitive positive definitions} and other complexity-preserving operators, inspired by the techniques known from VCSPs over finite domains. We use the expressive power to obtain polynomial-time reductions between VCSPs and use
them as the basis for formulating  a hardness condition for infinite domain
VCSPs.  We  conjecture that for VCSPs that stem from resilience problems, this hardness condition is not only
sufficient but also necessary, unless P = NP. More
precisely, we conjecture that this is the case when the automorphism group of the valued structure is identical
to that of a structure which is a reduct of a countable finitely
bounded homogeneous structure.

We also present an algebraic condition for (infinite-domain) valued structures which implies that the induced VCSP is in P, based on the concept of \emph{fractional polymorphisms} which generalize classical polymorphisms, a common tool for proving tractability of CSPs (see, e.g.,~\cite{CohenCooperJeavonsVCSP}). 
To prove membership in P, we use a reduction to finite-domain VCSPs which can be solved by a linear programming relaxation technique. We conjecture that the resulting algorithm solves all resilience problems that are in P. 
We demonstrate the utility of our  tractability condition by
applying it to a concrete conjunctive query for which the computational complexity of resilience has been stated as an open problem in the literature~\cite{NewResilience} (Section~\ref{sect:expl}). 
 
\paragraph{Related Work.}
The study of resilience problems was initiated in \cite{Resilience} under set semantics. The authors study the class of self-join-free conjunctive queries, i.e., queries in which each relation symbol occurs at most once, and obtain a P versus NP-complete dichotomy for this class. In a subsequent article~\cite{NewResilience}, several results are obtained for conjunctive queries with self-joins of a specific form, while the authors also state a few open problems of similar nature that cannot be handled by their methods. In the subsequent \cite{LatestResilience}, Gatterbauer and Makhija present a unified approach to resilience problems based on integer linear programming that works also for queries with self-joins, both under bag semantics and under set semantics. The new complexity results in \cite{LatestResilience} again concern self-join-free queries. Our approach is independent from self-joins and hence allows to study conjunctive queries that were difficult to treat before. In the very recent article \cite{amarilli2024resilienceregularpathqueries}, Amarilli et al.\ study the resilience of (one-way) regular path queries under set and bag semantics. They identify several 
tractable classes and several intractable classes, but a full classification remains open.

We stress that VCSPs of countable valued structures with an oligomorphic automorphism group greatly surpass resilience problems. For example, many problems in the recently very active area of graph separation problems \cite{KPW22,KKPW21} such as \emph{(weighted) directed feedback arc set problem} and \emph{directed symmetric multicut problem} can be formulated as VCSPs of appropriate countable valued structures with an oligomorphic automorphism group. Several of these problems such as the \emph{multicut problem} and the \emph{coupled min cut problem} can even be formulated as VCSPs over finite domains.
 
\paragraph{Outline.} The article is organized from the general to the specific, starting with VCSPs in full generality (Section~\ref{sect:prelims}), then focusing on valued structures with an oligomorphic automorphism group (Section~\ref{sect:oligo}), for which our notion of expressive power (Section~\ref{sect:expr}) leads to polynomial-time reductions. Our general hardness condition, which also builds upon the notion of expressive power, is presented in Section~\ref{sect:gen-hard}. 
To study the expressive power and to formulate general polynomial-time tractability results, we introduce the concept of \emph{fractional polymorphisms} in Section~\ref{sect:fpol} (they are probability distributions over operations on the valued structure). 
We take inspiration from the theory of VCSPs for finite-domain valued structures, but apply some non-trivial modifications that are specific to the infinite-domain setting (because the considered probability distributions are over uncountable sets) making sure that our definitions specialize to the standard ones over finite domains.
We then present a general polynomial-time tractability result (Theorem~\ref{thm:tract}) which is phrased in terms of fractional polymorphisms. Section~\ref{sect:resilience} applies the general theory to resilience problems. We illustrate the power of our approach by settling the computational complexity of a resilience problem for a concrete conjunctive query from the literature (Section~\ref{sect:expl}). 
Section~\ref{sect:concl} closes with open problems for future research.

Most results of the present article have been announced in a conference paper published at LICS'24~\cite{Resilience-VCSPs}. Compared to the conference version, we generalize the tractability result, include a section on two-way regular path queries,  add several examples of queries for which we determine the complexity of the resilience problem by our methods, and provide fully detailed proofs.

\section{Preliminaries}
\label{sect:prelims}
The set $\{0,1,2,\dots\}$ of natural numbers is denoted by ${\mathbb N}$, the set of rational numbers is denoted by $\mathbb Q$, the set of non-negative rational numbers by ${\mathbb Q}_{\geq 0}$ and the set of positive rational numbers by $\Q_{>0}$. We use analogous notation for the set of real numbers $\mathbb{R}$ and the set of integers $\Z$.
We also need an additional value $\infty$; all
we need to know about $\infty$ is that
\begin{itemize}
\item $a < \infty$ for every $a \in {\mathbb Q}$,
\item $a + \infty = \infty + a = \infty$ for all $a \in {\mathbb Q} \cup \{\infty\}$, and 
\item $0 \cdot \infty = 
\infty \cdot 0 = 0$
and $a \cdot \infty =
\infty \cdot a = \infty$ for $a > 0$. 
 \end{itemize}
For an operation  $f \colon A^{\ell} \to A$ 
and $a^1 = (a_1^1, \dots, a_k^1), \dots, a^\ell = (a_1^\ell, \dots, a_k^\ell) \in A^k$, we
use the following notation for applying $f$ componentwise: 
\[ f(a^1, \dots, a^\ell) := (f(a^1_1, a^2_1, \dots, a^\ell_1), \ldots, f(a^1_k, a^2_k, \dots, a^\ell_k)).\]

\subsection{Valued Structures}
Let $C$ be a set and let $k \in {\mathbb N}$. 
A \emph{valued relation of arity $k$ over $C$}
is 
a function $R \colon C^k \to {\mathbb Q} \cup \{\infty\}$.  
We write ${\mathscr R}_C^{(k)}$ for the set of all valued relations over $C$ of arity $k$, and define 
$${\mathscr R}_C := \bigcup_{k \in {\mathbb N}} {\mathscr R}_C^{(k)}.$$
A valued relation is called \emph{finite-valued} if it takes values only in $\Q$.
The \emph{valued equality relation} $R_=$  is the binary valued relation defined over $C$ by $R_=(x,y) = 0$ if $x=y$ and $R_=(x,y) = \infty$ otherwise. The \emph{empty relation} $R_{\emptyset}$ is the unary valued relation defined over $C$ by $R_{\emptyset}(x) = \infty$ for all $x \in C$. 

A valued relation $R \in {\mathscr R}_C^{(k)}$ that only takes values from $\{0,\infty\}$ will be identified with the `crisp' relation 
$\{a \in C^k \mid R(a) = 0\}.$
For $R \in {\mathscr R}_C^{(k)}$ the \emph{feasibility relation of $R$} is defined as
$$ \Feas(R) := \{a \in C^k \mid R(a) < \infty\}.$$

A \emph{(relational) signature} $\tau$ is a set of \emph{relation symbols}, each of them equipped with an arity from ${\mathbb N}$. A \emph{valued $\tau$-structure} $\Gamma$ consists of a set $C$, which is also called the \emph{domain} of $\Gamma$, and a valued relation $R^{\Gamma} \in {\mathscr R}_C^{(k)}$ for each relation symbol $R \in \tau$ of arity $k$. A \emph{$\tau$-structure} in the usual sense may then be identified with a valued $\tau$-structure where all valued relations only take values from $\{0,\infty\}$.
All valued and relational structures considered in this article have a countable domain.

\begin{expl}\label{expl:vs-mc} 
Let $\tau = \{<\}$ be a relational signature with a single binary relation symbol $<$. 
Let $\Gamma_{<}$ be the valued $\tau$-structure with domain $\{0,1\}$ and where
${<}(x,y) = 0$ if $x < y$, and
${<}(x,y) = 1$ otherwise. 
\end{expl}

If $\sigma \subseteq \tau$ are relational signatures, $\Gamma$ is a valued $\tau$-structure and $\Delta$ is a valued $\sigma$-structure such that $R^{\Delta} = R^{\Gamma}$ for every $R \in \sigma$, then we call $\Delta$ a \emph{reduct} of $\Gamma$ and $\Gamma$ an \emph{expansion} of $\Delta$. By the convention introduced above, this terminology applies also to relational structures.

An \emph{atomic $\tau$-expression} is an expression
of the form $R(x_1,\dots,x_k)$ for $R \in \tau \cup \{R_\emptyset, R_= \}$ and (not necessarily distinct) variable symbols $x_1,\dots,x_k$. 
A \emph{$\tau$-expression} is an expression $\phi$ 
of the form 
$\sum_{i \leq m} \phi_i$
where $m \in {\mathbb N}$ 
and $\phi_i$ for $i \in \{1,\dots,m\}$ is 
an atomic $\tau$-expression.
Note that the same atomic $\tau$-expression might appear several times in the sum. 
We write $\phi(x_1,\dots,x_n)$ for a $\tau$-expression where all the variables  
are from the set $\{x_1,\dots,x_n\}$. 
If $\Gamma$ is a valued $\tau$-structure, then a $\tau$-expression $\phi(x_1,\dots,x_n)$ defines over $\Gamma$ a member of ${\mathscr R}_C^{(n)}$, which we denote by $\phi^{\Gamma}$. If $\phi$ is the empty sum then $\phi^{\Gamma}$ is constant~$0$.

\subsection{Valued Constraint Satisfaction} 
\label{sect:vcsp}
In this section we assume that
 $\Gamma$ is a fixed valued $\tau$-structure for a 
 \emph{finite} signature $\tau$. 
 The valued relations of $\Gamma$ are also called \emph{cost functions}. 
The \emph{valued constraint satisfaction problem for $\Gamma$}, denoted by \emph{$\VCSP(\Gamma)$}, is the computational
problem to decide for a given $\tau$-expression $\phi(x_1,\dots,x_n)$ 
and a given $u \in {\mathbb Q}$
whether there exists
$a \in C^n$ such that $\phi^{\Gamma}(a) \leq u$. 
We refer to $\phi(x_1,\dots,x_n)$ as an \emph{instance} of $\VCSP(\Gamma)$,
and to $u$ as the \emph{threshold}. 
Tuples $a \in C^n$ such that $\phi^{\Gamma}(a) \leq u$ are called a \emph{solution for $(\phi,u)$}.
The \emph{cost}  of $\phi$ (with respect to $\Gamma$) is defined to be
$$\inf_{a \in C^n} \phi^{\Gamma}(a).$$
In some contexts, it will be beneficial to consider only a given $\tau$-expression $\phi$ to be the input of $\VCSP(\Gamma)$ (rather than $\phi$ and the threshold $u$) and a tuple $a \in C^n$ will then be called a \emph{solution for $\phi$} if the cost of $\phi$ equals $\phi^{\Gamma}(a)$.
Note that in general there might not be any solution, because the infimum from the definition of the cost might not be attained. 
If there exists a tuple $a \in C^n$ such that $\phi^{\Gamma}(a) < \infty$ then $\phi$ is called \emph{satisfiable}.

\begin{remark}\label{rem:def-vcsp}
The definition of a VCSP above is a direct generalization of the standard definition of a VCSP of a finite-domain valued structure. We remark that in other works on infinite-domain VCSPs~\cite{SchneiderViola, ViolaThesis}, the definition of a VCSP is changed so that it asks whether there exists $t \in C^n$ such that $\phi^\Gamma(t) < u$; this prevents problems arising from the infimum of $\phi^\Gamma$ not being realized. Since 
in valued structures considered in this paper this situation does not occur (see Lemma~\ref{lem:cost}), we prefer the definition with $\leq$, which is more standard in the setting where optimization problems are modeled as decision problems with a threshold.
\end{remark}

Note that our setting also captures classical CSPs, which can be viewed as the VCSPs for valued structures $\Gamma$ that only contain cost functions that take value $0$ or $\infty$. In this case, we will sometimes write $\CSP(\Gamma)$ for $\VCSP(\Gamma)$. Below we give two examples of known optimization problems that can be formulated as valued constraint satisfaction problems.

\begin{expl}\label{expl:mc}
The problem $\VCSP(\Gamma_<)$ for the valued structure $\Gamma_<$ from Example~\ref{expl:vs-mc} models the \emph{directed max-cut} problem: given a finite directed graph $(V,E)$ (we do allow loops and multiple edges), partition the vertices $V$ into two classes $A$ and $B$ such that the number of edges from $A$ to $B$ is maximal. Maximising the number of edges from $A$ to $B$ amounts to minimising the number $e$ of edges within $A$, within $B$, and from $B$ to $A$. So when we associate $A$ to the preimage of $0$
and $B$ to the preimage of $1$, computing the number $e$ corresponds to finding the evaluation map $s \colon V \rightarrow \{0,1\}$  that minimizes the value $\sum_{(x,y) \in E} {<}(s(x),s(y))$,
which can be formulated as an instance of 
$\VCSP(\Gamma_<)$. Conversely, every instance of $\VCSP(\Gamma_<)$ corresponds to a directed max-cut instance.   
It is known that $\VCSP(\Gamma_{<})$ is NP-complete (even if we do not allow loops and multiple edges in the input)~\cite{GareyJohnson}. We mention that this problem can be viewed as a resilience problem in database theory as explained in Section~\ref{sect:resilience}.
\end{expl}

\begin{expl}\label{expl:min-cut}
Consider the valued structure $\Gamma_{\geq}$ with domain $\{0,1\}$ and the binary valued relation $\geq$ defined by ${\geq}(x,y)=0$ if $x \geq y$ and ${\geq}(x,y)=1$ otherwise. Similarly to Example~\ref{expl:mc}, $\VCSP(\Gamma_{\geq})$ models the directed min-cut problem, i.e., given a finite directed graph $(V,E)$, partition the vertices $V$ into two classes $A$ and $B$ such that the number of edges from $A$ to $B$ is minimal. The min-cut problem is solvable in polynomial time; see, e.g., \cite{max-flow}.
\end{expl}

\section{Oligomorphicity}
\label{sect:oligo}
Many facts about VCSPs for valued structures with a finite domain can be generalized to a large class of valued structures over an infinite domain,   defined in terms of automorphisms. 
We define automorphisms of valued structures as follows.

\begin{definition} 
Let $k \in {\mathbb N}$, 
let $R \in {\mathscr R}^{(k)}_C$, and let $\alpha$ be a permutation of $C$. Then $\alpha$ \emph{preserves} $R$ if for all $a \in C^k$
we have $R(\alpha(a)) = R(a)$. 
If $\Gamma$ is a valued structure with domain $C$, then 
an \emph{automorphism} of $\Gamma$ is a permutation of $C$ that preserves all valued relations of $R$. 
\end{definition}

The set of all automorphisms of $\Gamma$ is denoted by $\Aut(\Gamma)$, and forms a group with respect to composition. For two valued structures $\Gamma$ and $\Delta$ we often write $\Aut(\Delta) \subseteq \Aut(\Gamma)$ to indicate that $\Delta$ and $\Gamma$ are on the same domain and all automorphisms of $\Delta$ are automorphisms of $\Gamma$ (typically, $\Delta$ will be a relational structure).
Let $k \in {\mathbb N}$ a let $G$ be a permutation group on a set $C$.
An \emph{orbit of $k$-tuples} of 
$G$ is a set of the form $\{ \alpha(a) \mid \alpha \in G \}$ for some $a \in C^k$ (as introduced in the preliminaries, $\alpha \in G$ acts on $C^k$ componentwise). 
A permutation group $G$ on a countable set is called \emph{oligomorphic} if for every $k \in {\mathbb N}$ there are finitely many orbits of $k$-tuples in $G$ \cite{Oligo}. 
Note that if $G$ and $H$ are permutation groups on the same set, $H \subseteq G$, and $H$ is oligomorphic, then $G$ is oligomorphic.
From now on, whenever we write that a structure has an oligomorphic automorphism group, we also imply that its domain is countable. 
Clearly, every valued structure with a finite domain has an oligomorphic automorphism group. A countable structure has an oligomorphic automorphism group if and only if it is \emph{$\omega$-categorical}, i.e., if all countable models of its first-order theory are isomorphic (see, e.g.~\cite[Theorem 6.3.1]{Hodges}). 

\begin{expl}\label{expl:MCC}
Let $\tau = \{E,N\}$ be a relational signature with two binary relation symbols $E$ and $N$. 
Let $\Gamma_{\text{LCC}}$ be the valued $\tau$-structure with domain ${\mathbb N}$ where
$E(x,y) = 0$ if $x = y$ and $E(x,y)=1$ otherwise, and where $N(x,y) = 0$ if $x \neq y$ and $N(x,y)=1$ otherwise.
Note that $\Aut(\Gamma_{\text{LCC}})$ is the full symmetric group on $\mathbb{N}$. This group is oligomorphic; for example, there are five orbits of triples represented by the tuples $(1,2,3)$, $(1,1,2)$, $(1,2,1)$, $(2,1,1)$ and $(1,1,1)$.

The problem of \emph{least correlation clustering with partial information}~\cite[Example 5]{ViolaThesis}
is equal to $\VCSP(\Gamma_{\text{LCC}})$.
It is a variant of the min-correlation clustering problem \cite{CorrelationClustering}
that does not require precisely one constraint between any two variables.
The problem is NP-complete in both settings \cite{GareyJohnson, ViolaThesis}.
\end{expl}

The following lemma shows that valued $\tau$-structures with an oligomorphic automorphism group always realize infima of $\tau$-expressions.
\begin{lemma}\label{lem:cost} 
Let $\Gamma$ be a valued structure with a countable domain $C$ and an oligomorphic automorphism group. Then 
for every instance $\phi(x_1,\dots,x_n)$ of $\VCSP(\Gamma)$ there exists $a \in C^n$ 
such that the cost of $\phi$ equals $\phi^{\Gamma}(a)$. 
\end{lemma}
\begin{proof}
The statement follows from the assumption that there are only finitely many orbits of $n$-tuples of $\Aut(\Gamma)$, because it implies that there are only finitely many possible values from ${\mathbb Q} \cup \{\infty\}$ for $\phi^{\Gamma}(a)$. 
\end{proof}

A first-order sentence is called \emph{universal} if it is of the form $\forall x_1,\dots,x_l.\; \psi$ where $\psi$ is quantifier-free. Every quantifier-free formula is equivalent to a formula in conjunctive normal form, so we generally assume that quantifier-free formulas are of this form.

Recall that a $\tau$-structure $\bA$ \emph{embeds} into a $\tau$ structure $\bB$ if there is an injective map from $A$ to $B$ that preserves all relations of $\bA$ and their complements; the corresponding map is called an \emph{embedding}. The \emph{age} of a $\tau$-structure is the class of all finite $\tau$-structures that embed into it.
A relational structure $\bB$ with a finite relational signature $\tau$ is called 
\begin{itemize}
    \item \emph{finitely bounded} if there exists a universal $\tau$-sentence $\phi$ such that a finite structure $\bA$
    is in the age of $\bB$ 
iff $\bA \models \phi$. 
    \item \emph{homogeneous} if every isomorphism between finite substructures of $\bB$ can be extended to an automorphism of~$\bB$.  
\end{itemize}
Note that every finite relational structure is finitely bounded, but it does not have to be homogeneous (and typically it is not).

Note that for every 
structure $\bB$ with a finite relational signature, for every $n$ there are only finitely many non-isomorphic substructures of $\bB$ of size $n$. Therefore, all countable homogeneous structures with a finite relational signature and all of their reducts have finitely many orbits of $k$-tuples for all $k \in \N$, and hence an oligomorphic automorphism group.

\begin{theorem}\label{thm:fb-NP}
Let $\Gamma$ be a countable valued structure
with finite signature such that there exists a finitely bounded 
homogeneous structure $\bB$ with $\Aut(\bB) \subseteq \Aut(\Gamma)$. Then $\VCSP(\Gamma)$ is in NP. 
\end{theorem}
\begin{proof}
Let $(\phi,u)$ be an instance of $\VCSP(\Gamma)$ with $n$ variables. Since $\Aut(\bB) \subseteq \Aut(\Gamma)$, every orbit of $n$-tuples of $\Aut(\Gamma)$
is determined by the substructure induced by $\bB$ on the elements of some tuple from the orbit. Note that two tuples $(a_1, \dots, a_n)$ and $(b_1, \dots, b_n)$ lie in the same orbit of $\Aut(\bB)$ if and only if the map that maps $a_i$ to $b_i$ for $i\in \{1,\dots,n\}$ is an isomorphism between the substructures induced by $\bB$ on $\{a_1, \dots, a_n\}$ and on $\{b_1, \dots, b_n\}$. Whether a given finite structure $\bA$ is in the age of a fixed finitely bounded structure $\bB$ can be decided in polynomial time: if $\phi$ is the universal $\tau$-sentence which describes the age of $\bB$,
it suffices 
to exhaustively check all possible instantiations of the variables of $\phi$ with elements of $A$ and verify whether $\phi$ is true in $\bA$ under the instantiation.
Hence, we may non-deterministically generate a structure $\bA$ with domain
$\{1,\dots,n\}$ from the age of $\bB$ 
and then verify in polynomial time whether the 
value $\phi^{\Gamma}(b_1, \dots, b_n)$ is at most $u$ for any tuple $(b_1,\dots,b_n) \in B^n$ 
such that $i \mapsto b_i$ is an embedding of $\bA$ into $\bB$.
\end{proof}

\section{Expressive Power} 
\label{sect:expr}
One of the fundamental concepts in the theory of constraint satisfaction is the concept of \emph{primitive positive definitions}, which is the fragment of first-order logic where only equality, existential quantification, and conjunction are allowed (in other words, negation, universal quantification, and disjunction are forbidden). 
The motivation for this concept is that relations with such a definition can be added to the structure without changing the complexity of the respective CSP. The natural generalization to \emph{valued} constraint satisfaction is the following notion of pp-expressibility; the acronym `pp' stands for primitive positive.


\begin{definition}
Let $A$ be a set and $R,R' \in {\mathscr R}_A$.
We say that $R'$ can be obtained from $R$ by 
\begin{itemize}
\item \emph{projecting} if $R'$ is of arity $k$, $R$ is of arity $k+n$, and for all $a \in A^k$ 
\[R'(a) = \inf_{b \in A^n} R(a,b).\]
\item \emph{non-negative scaling} 
 if there exists $a \in {\mathbb Q}_{\geq 0}$ such that $R' = a R$;
 \item  \emph{shifting} if there exists $a \in {\mathbb Q}$ such that $R' = R + a$. 
\end{itemize}
If $R$ is of arity $k$, then $\Opt(R)$ denotes the relation consisting of all minimal-value tuples of $R$, i.e.,
\[\Opt(R) := \{a \in \Feas(R) \mid R(a) \leq R(b) \text{ for every } b \in A^k\}.\]
\end{definition}

Note that $\inf_{b \in A^n} R(a,b)$ in item (1) might be irrational or $-\infty$. If this is the
case, then $\inf_{t \in A^n} R(s,t)$ does not express a valued relation, 
because valued relations must have weights from ${\mathbb Q} \cup \{\infty\}$.
However, if $R$ is preserved by all permutations of
an oligomorphic automorphism group, then Lemma~\ref{lem:cost}  guarantees
that the infimum is attained for some tuple $b$.


If $\mathscr{S} \subseteq \mathscr{R}_A$, then an \emph{atomic expression over $\mathscr{S}$} is an atomic $\tau$-expression for $\tau :=\mathscr{S}$.
We say that $\mathscr{S}$ is \emph{closed under forming sums of atomic expressions} if it contains all valued relations defined by sums of atomic expressions over $\mathscr{S}$.

\begin{definition}[valued relational clone]\label{def:wrelclone}
Let $C$ be a set. A \emph{valued relational clone (over $C$)} is 
a subset of 
${\mathscr R}_C$ 
that 
is closed under forming sums of atomic expressions,
projecting, shifting, non-negative scaling, $\Feas$, and $\Opt$; we will refer to expressions formed this way as \emph{pp-expressions}.
For a valued structure $\Gamma$ with the domain $C$,  
we write 
$\langle \Gamma \rangle$ for the smallest relational clone that contains the valued relations of $\Gamma$. 
If $R \in \langle \Gamma \rangle$, we say that $\Gamma$ \emph{pp-expresses} $R$. 
\end{definition}




The following example shows that neither the operator $\Opt$ nor the operator $\Feas$ is redundant in the definition above.

\begin{expl}\label{expl:opt-feas}
Consider the domain $C=\{0,1,2\}$ and the unary valued relation $R$ on $C$ defined by $R(0)=0$, $R(1)=1$ and $R(2)=\infty$. Then the relation $\Feas(R)$ cannot be obtained from $R$ by forming sums of atomic expressions, projecting, shifting, non-negative scaling and use of $\Opt$. Similarly, the relation $\Opt(R)$ cannot be obtained from $R$ by forming sums of atomic expressions, projecting, shifting, non-negative scaling and use of $\Feas$. 
\end{expl}

\begin{remark}\label{rem:aut}
Note that for every valued structure $\Gamma$ and $R\in \langle \Gamma \rangle$, every automorphism of $\Gamma$ is an automorphism of $R$. 
\end{remark}

The motivation for Definition~\ref{def:wrelclone} for valued CSPs stems from the following lemma, which shows that adding relations in $\langle\Gamma \rangle$ does not change the complexity of the $\VCSP$ up to polynomial-time reductions. For finite-domain valued structures this is proved  in~\cite{VCSP-Galois}, 
except for the operator $\Opt$, for which a proof can be found in~\cite[Theorem 5.13]{FullaZivny}. Parts of the proof have been generalized to infinite-domain valued structures without further assumptions; see, e.g. \cite{SchneiderViola} and \cite[Lemma 7.1.4]{ViolaThesis}. However, in these works the definition of VCSPs was changed to ask whether there is a solution of a cost strictly less than $u$, to circumvent problems about infima that are not realized; for structures with an oligomorphic automorphism group, this difference is unimportant as the valued relations attain only finitely many values.
Moreover,
in~\cite{SchneiderViola,ViolaThesis} the authors 
do not consider the operators $\Opt$ and $\Feas$. 
It is visible from Example \ref{expl:opt-feas} that neither the operator $\Opt$ nor the operator $\Feas$ can be simulated by the other ones already on finite domains, which is why they both appear in \cite{FullaZivny} ($\Feas$ was included implicitly by allowing to scale by $0$ and defining $0 \cdot \infty=\infty$). In this article we work with valued structures with an oligomorphic automorphism group so that the infima coming from projecting 
are attained and hence
we can adapt the proof from the finite-domain case.
We provide a full proof covering all operators from Definition~\ref{def:wrelclone} for the convenience of the reader.

\begin{lemma}\label{lem:expr-reduce}
Let $\Gamma$ be a valued structure with an oligomorphic 
automorphism group and a finite signature. 
Suppose that $\Delta$ is a valued structure with a finite signature over the same domain $C$ such that every cost function of $\Delta$ is from $\langle \Gamma \rangle$. 
Then there is a polynomial-time reduction from
$\VCSP(\Delta)$ to $\VCSP(\Gamma)$. 
\end{lemma}

\begin{proof}
Let $\tau$ be the signature of $\Gamma$. It suffices to prove the statement for expansions of $\Gamma$ to signatures $\tau \cup \{R\}$ that extend $\tau$ with a single relation $R$, $R^{\Delta} \in \langle \Gamma \rangle$. 

If $R^{\Delta} = R_\emptyset$,
 then an instance $\phi$ of $\VCSP(\Delta)$ with threshold $u \in \Q$ is unsatisfiable if and only if $\phi$ contains the symbol $R$ or if it does not contain $R$ and is unsatisfiable viewed as an instance of $\VCSP(\Gamma)$. 
In the former case, 
$(\phi,u)$ is satisfiable in $\Delta$ if and only if $(R_\emptyset, u)$ is satisfiable in $\Gamma$, 
so this provides a correct reduction. 
In the latter case, for every $a \in C^n$ we have that $\phi^{\Delta}(a) = \phi^{\Gamma}(a)$; 
this provides a polynomial-time reduction. 

Now suppose that $R^{\Delta} = R_=$. Let $\psi(x_{i_1},\dots,x_{i_k})$ be obtained from an instance $\phi(x_1,\dots,x_n)$ of $\VCSP(\Delta)$ by identifying all variables $x_i$ and $x_j$ such that $\phi$ contains the summand $R(x_i,x_j)$. 
Then $\phi$ is satisfiable if and only if the instance 
$\psi$ is satisfiable, and $\inf_{a \in C^n} \phi^{\Delta}(a) = \inf_{b \in C^k} \psi^{\Gamma}(b)$;
Again, this provides a polynomial-time reduction. 

Next, consider that for some $\tau$-expression $\delta(y_1,\dots,y_l,z_1,\dots,z_k)$ we have  
$$R^{\Delta}(y_1, \dots, y_l)= \inf_{a \in C^k} \delta^{\Gamma}(y_1, \dots, y_l, a_1, \dots, a_k).$$ 
Let $\phi(x_1,\dots,x_n)$ be an instance of $\VCSP(\Delta)$. 
We replace each summand $R(y_1,\dots,y_{l})$ in $\phi$ by $\delta(y_1,\dots,y_l,z_1,\dots,z_k)$ where $z_1,\dots,z_k$ are new variables (different for each summand). After doing this for all summands that involve $R$, let $\theta(x_1,\dots,x_n,w_1,\dots,w_t)$ be the resulting $\tau$-expression.  For any $a \in C^{n}$ 
we have that
$$\phi(a_1,\dots,a_n) = \inf_{b \in C^t} \theta(a_1,\dots,a_{n},b)$$
and hence $\inf_{a \in C^n} \phi = \inf_{c \in C^{n+t}} \theta$; here we used the assumption that $\Aut(\Gamma)$ is oligomorphic (see Lemma~\ref{lem:cost}). 
Since we replace each summand by a $\tau$-expression whose size is constant (since $\Gamma$ is fixed and finite) the $\tau$-expression $\theta$ can be computed in polynomial time, which shows the statement. 

Suppose that $R^{\Delta}=rS^{\Gamma}+s$ where $r \in \Q_{\geq 0}, s \in \Q$. Let $p\in \Z_{\geq 0}$ and $q \in \Z_{>0}$ be coprime integers such that $p/q = r$. Let $(\phi, u)$ be an instance of $\VCSP(\Delta)$ where $\phi(x_1, \dots, x_n)=\sum_{i=1}^{\ell} \phi_i + \sum_{j=1}^k \psi_j$, the summands $\phi_i$ contain only symbols from $\tau$, and each $\psi_j$ involves the symbol $R$. Let $\psi_j'$ be the $\tau$-expression obtained from $\psi_j$ by replacing $R$ with $S$.
For $i \in \{1, \dots, \ell\}$ replace $\phi_i$ with $q$ copies of itself and for $j \in \{1, \dots, k\}$, replace $\psi_j$ with $p$ copies of $\psi_j'$; let $\phi'(x_1, \dots, x_n)$ be the resulting $\tau$-expression. Define $u':= q(u-ks)$. Then for every $a \in C^n$ the following are equivalent:
\begin{align*}
 \phi(a_1, \dots, a_n) & = \sum_{i=1}^{\ell} \phi_i+  \sum_{j=1}^k \left(\frac{p}{q} \psi'_j + s \right)  \leq u  \\
  \phi'(a_1, \dots, a_n) & = q\sum_{i=1}^{\ell} \phi_i+ p \sum_{j=1}^k \psi'_j  \leq qu-qks = u'
\end{align*}
Since $(\phi', u')$ can be computed from $(\phi, u)$ in polynomial time, this provides the desired reduction.

Now suppose that $R^{\Delta}=\Feas(S^{\Gamma})$ for some $S\in \tau$. Let $(\phi,u)$ be an instance of $\VCSP(\Delta)$, i.e., $\phi(x_1, \dots, x_n)=\sum_{i=1}^{\ell} \phi_i + \sum_{j=1}^k \psi_j$ where $\psi_j$, $j \in \{1, \dots, k\}$ are all the atomic expressions in $\phi$ that involve $R$. If $R^\Delta=R_{\emptyset}$, then the statement follows from the reduction for $R_\emptyset$. Therefore, suppose that this not the case. Since $\tau$ is finite and $\Aut(\Gamma)$ is oligomorphic, we may assume without loss of generality that all valued relations attain only non-negative values; otherwise we shift the values, which by the previous case does not affect the complexity up to polynomial-time reductions. Let $w$ be the maximum finite weight assigned by $S$. Note that there are only finitely many values that the $\ell$ atoms $\phi_i$ may take and therefore only finitely many values that $\sum_{i=1}^{\ell} \phi_i$ may take. Let $v$ be the smallest of these values such that $v>u$ and let $d = v-u$; if $v$ does not exist, let $d=1$. To simplify the notation, set $t = \lceil (kw)/d \rceil + 1$. Let $\psi_j'$ be the atomic expression resulting from $\psi_j$ by replacing the symbol $R$ by the symbol $S$. Let $\phi'$ be the $\tau$-expression obtained from $\phi$ by replacing each atom $\phi_i$ with $t$ copies of it and replacing every atom $\psi_j$ by $\psi_j'$. Let $(\phi', tu+kw)$ be the resulting instance of $\VCSP(\Gamma)$; note that it can be computed in polynomial time.

We claim that for every $a\in C^n$, the following are equivalent:
\begin{align}
 \phi(a_1, \dots, a_n) & = \sum_{i=1}^{\ell} \phi_i+  \sum_{j=1}^k \psi_j  \leq u   \label{eq:feas1}\\
  \phi'(a_1, \dots, a_n) & = t  \cdot \sum_{i=1}^{\ell} \phi_i+ \sum_{j=1}^k \psi'_j  \leq tu+kw \label{eq:feas2}
\end{align}
If \eqref{eq:feas1} holds, then by the definition of $\Feas$ we must have $\psi_j=0$ for every $j \in \{1, \dots, k\}$. Thus $\sum_{i=1}^{\ell} \phi_i \leq u$ and $\sum_{j=1}^k \psi'_j \leq kw$, which implies \eqref{eq:feas2}. Conversely, if \eqref{eq:feas2} holds, then  $\psi_j'$ is finite for every $j \in\{1, \dots, k\}$ and hence $\psi_j=0$. Moreover, \eqref{eq:feas2} implies
\[\sum_{i=1}^{\ell} \phi_i \leq u + \frac{kw}{t}.\] Note that if $v$ exists, then $u + (kw)/t < v$. Therefore (regardless of the existence of $v$), this implies $\sum_{i=1}^{\ell} \phi_i \leq u$, which together with what we have observed previously shows \eqref{eq:feas1}.

Finally, we consider the case that $R^{\Delta} = \Opt(S^{\Gamma})$ for some relation symbol $S \in \tau$.  
Similarly to the previous case, we may assume without loss of generality that 
the minimum weight of all valued relations in $\Delta$ equals $0$; otherwise, we subtract the smallest weight 
assigned to a tuple by some 
valued relation in $\Delta$. 
We may also assume that $S^{\Gamma}$ takes finite positive values, because otherwise $\Opt(S^{\Gamma}) = S^{\Gamma}$ and the statement is trivial. 
Let $m$ be the smallest positive weight assigned by $S^{\Gamma}$ and let $M$ be the largest finite weight assigned by any valued relation 
of $\Gamma$ (again we use that $\tau$
is finite and that $\Aut(\Gamma)$ is oligomorphic). Let $(\phi, u)$, where $\phi(x_1,\dots,x_n)= \sum_{i=1}^k\phi_i$, be an instance  of $\VCSP(\Delta)$. For $i \in \{1,\dots,k\}$, if $\phi_i$ involves the symbol $R$, then replace it by $k \cdot \lceil M/m \rceil + 1$ copies and replace $R$ by $S$. Let $\phi'$ be the resulting $\tau$-expression.
We claim that $a \in C^n$ is a solution to the instance $(\phi',\min(kM,u))$ of $\VCSP(\Gamma)$ if and only if it is the solution to $(\phi,u)$.

If $a \in C^n$ is such that
$\phi(a) \leq u$ then for every $i \in \{1,\dots,k\}$ such that $\phi_i$ involves $R$ we have $\phi_i(a) = 0$. In particular, the minimal value attained by $S^{\Gamma}$ equals $0$ by our assumption, and hence $\phi'(a) = \phi(a) \leq u$
and hence $\phi'(a) \leq \min(kM,u)$.
Now suppose that 
$\phi(a) > u$. Then 
$\phi'(a)>u \geq \min(kM,u)$ or
there exists an $i \in \{1,\dots,k\}$ such that $\phi_i(a) = \infty$.
If $\phi_i$ does not involve the symbol $R$, then $\phi'(a) = \infty$ as well. If $\phi_i$  involves the symbol $R$, then $\phi'(a) \geq (k \cdot \lceil M/m \rceil + 1) m > kM$. In any case, $\phi'(a)>\min(kM, u)$.
Since $\phi'$ can be computed from $\phi$ in polynomial time, this concludes the proof.
\end{proof}

The next two examples illustrate the use 
of Lemma~\ref{lem:expr-reduce} for obtaining hardness results. 

\begin{expl}\label{expl:mc-hard}
    Recall the structure $\Gamma_<$ from Example~\ref{expl:vs-mc}. We have seen in Example~\ref{expl:mc} that $\VCSP(\Gamma_<)$ is the directed max-cut problem. Consider the classical relation $\NAE$ on $\{0,1\}$ defined by
    \[\NAE := \{0,1\}^3 \setminus \{(0,0,0), (1,1,1)\}.\]
    Note that 
    \[\NAE(x,y,z) = \Opt \big ({<}(x,y)+{<}(y,z)+{<}(z,x) \big).\]
    Since $\CSP(\{0,1\}, \NAE)$ is known to be NP-hard (see, e.g., \cite{Book}), this provides an alternative proof of the NP-hardness of the directed max-cut problem via Lemma~\ref{lem:expr-reduce}.
\end{expl}

\begin{expl}\label{expl:mcc-hard}
We revisit the countably infinite valued structure $\Gamma_{\text{LCC}}$ from Example~\ref{expl:MCC}.  Recall that 
$\VCSP(\Gamma_{\text{LCC}})$ is the 
least correlation clustering problem with partial information 
and that $\Aut(\Gamma_{\text{LCC}})$ is oligomorphic. 
Let $\Gamma_{\text{EC}}$ be the relational structure with the same domain as 
$\Gamma_{\text{LCC}}$ and the relation $R := \{(x,y,z) \mid (x=y \wedge y \neq z) \vee (x \neq y \wedge y = z)\}$ (attaining values $0$ and $\infty$). 
Note that 
$$R(x,y,z) = \Opt(N(x,z) + N(x,z) + E(x,y) + E(y,z)).$$
This provides an alternative proof of NP-hardness of the 
least correlation clustering with partial information via Lemma~\ref{lem:expr-reduce}, 
because $\CSP(\Gamma_{\text{EC}})$ is known to be NP-hard~\cite{ecsps}.
 
Note that we can replace $N(x,z) + N(x,z)$ in the definition of $R$ by $\Opt(N)(x,z)$ and that $\Opt(N)$ is equal to the classical relation $\neq$ (attaining values $0$ and $\infty$). This shows that even $\VCSP(\N; E, \neq)$ is NP-hard.
\end{expl}

\section{Hardness from pp-Constructions}
\label{sect:gen-hard}
A universal-algebraic theory of VCSPs for finite valued structures has been developed in~\cite{KozikOchremiak15}, following the classical approach to CSPs which is based on the concepts of cores, addition of constants, and primitive positive interpretations. Subsequently, 
an important conceptual insight has been made for classical CSPs which states that every structure that can be interpreted in the expansion of the core of the structure by constants can also be obtained by taking a pp-power if we then consider structures up to homomorphic equivalence~\cite{wonderland}. 
We are not aware of any published reference that adapts this perspective to the algebraic theory of VCSPs, so we  develop (parts of) this approach here. As in~\cite{wonderland}, we immediately step from valued structures with a finite domain to the more general case of valued structures with an oligomorphic automorphism group. The results in this section are adaptations of the known results for classical relational structures or finite-domain valued structures.

\begin{definition}[pp-power]
Let $\Gamma$ be a valued structure with domain $C$ and let $d \in {\mathbb N}$. Then a ($d$-th) \emph{pp-power}
of $\Gamma$ is a valued structure $\Delta$ with domain 
$C^d$ such that for every valued relation $R$ of $\Delta$ of arity $k$ there exists a valued relation $S$ of arity $kd$ in $\langle \Gamma \rangle$ such that 
$$R((a^1_1,\dots,a^1_d),\dots,(a^k_1,\dots,a^k_d)) = S(a^1_1,\dots,a^1_d,\dots,a^k_1,\dots,a^k_d).$$ 
\end{definition}

The name `pp-power' comes from `primitive positive power', since for relational structures pp-expressibility is captured by primitive positive formulas. The following proposition shows that the VCSP of a pp-power reduces to the VCSP of the original structure.

\begin{proposition}\label{prop:pp-hard}
Let $\Gamma$ and $\Delta$ be valued structures such that $\Aut(\Gamma)$ is oligomorphic and $\Delta$ is a pp-power of $\Gamma$. Then $\Aut(\Delta)$ is oligomorphic and 
there is a polynomial-time reduction from 
$\VCSP(\Delta)$ to $\VCSP(\Gamma)$. 
\end{proposition}
\begin{proof}
Let $d$ be the dimension of the pp-power
and let $\tau$ be the signature of $\Gamma$.
By Remark~\ref{rem:aut}, $\Aut(\Gamma) \subseteq \Aut(\Delta)$ and thus $\Aut(\Delta)$ is oligomorphic.
By Lemma~\ref{lem:expr-reduce}, we may suppose that for every valued relation $R$ of arity $k$ of $\Delta$ the valued relation $S \in \langle \Gamma \rangle$ of arity $dk$ from the definition of a pp-power equals $S^{\Gamma}$ for some $S \in \tau$. 
Let $(\phi,u)$ be an instance of $\VCSP(\Delta)$. 
For each variable $x$ of $\phi$ we introduce $d$ new variables $x_1,\dots,x_d$. 
For each summand $R(y^1,\dots,y^k)$ of $\phi$ we introduce a summand 
$S(y^1_1,\dots,y^1_d,\dots,y^k_1,\dots,y^k_d)$;
let $\psi$ be the resulting $\tau$-expression. 
It is now straightforward to verify that $(\phi,u)$ has a solution with respect to $\Delta$ if and only if $(\psi,u)$ has a solution with respect to $\Gamma$. 
\end{proof}

Heading towards the definition of pp-constructions, we introduce the important notion of fractional maps. To establish this notion and its basic properties, we need to work with notions from topology and measure theory, which we carefully introduce below at the required level of generality. If $C$ and $D$ are sets,
then we equip the space $C^D$ of functions from $D$ to $C$ with the topology of pointwise convergence, i.e., the product topology on the space $C^D$ where $C$ is a discrete topological space. 
In this topology, a basis of open sets is given by
$${\mathscr S}_{a,b} := \{f \in C^D \mid f(a)=b\}$$
for $a \in D^k$ and $b \in C^k$ for some $k \in {\mathbb N}$, and $f$ is applied componentwise.
For any topological space $T$, we denote by $\mathcal{B}(T)$
the Borel $\sigma$-algebra on $T$, i.e., the smallest subset of the powerset ${\mathcal P}(T)$ which contains all open sets and is closed under countable intersection and complement. We write $[0,1]$ for the set $\{x \in {\mathbb R} \mid 0 \leq x \leq 1\}$.

\begin{definition}[fractional map] \label{def:frac-map}
Let $C$ and $D$ be sets. A \emph{fractional map} from $D$ to $C$ is a probability distribution 
$$(C^D, \mathcal{B}(C^D),\omega \colon \mathcal{B}(C^D) \to [0,1]),$$
that is, $\omega(C^D) = 1$ and $\omega$ is countably additive: if $A_1,A_2,\dots \in \mathcal{B}(C^D)$ are disjoint, then $$\omega(\bigcup_{i \in {\mathbb N}} A_i) = \sum_{i \in {\mathbb N}} \omega(A_i).$$
\end{definition}
Note that every probability distribution is a \emph{measure}, i.e., it is a countably additive map $\omega$ from a $\sigma$-algebra to non-negative real numbers such that $\omega(\emptyset)=0$. If $f \in C^D$, we often write $\omega(f)$ instead of $\omega(\{f\})$. Note that $\{f\} \in \mathcal{B}(C^D)$ 
for every $f$. 
The set $[0,1]$ carries the topology inherited from  the standard topology on ${\mathbb R}$. We also view ${\mathbb R} \cup \{\infty\}$ as a topological space with a basis of open sets given by
all open intervals
$(a,b)$ for $a,b \in {\mathbb R}$, $a<b$ and additionally all sets of the form $\{x \in {\mathbb R} \mid x > a\} \cup \{\infty\}$.

A function $f \colon S \to T$ between two topological spaces is called \emph{(Borel-) measurable} if $f^{-1}(T') \in 
\mathcal{B}(S)$ for every $T' \in \mathcal{B}(T)$.
A \emph{(real-valued) random variable} is a measurable function $X \colon T \to {\mathbb R} \cup \{\infty\}$.
If $X$ is a real-valued random variable, then the \emph{expected value of $X$ (with respect to a probability distribution $\omega$)} is denoted by $E_\omega[X]$ and is defined 
via the Lebesgue integral 
$$ E_\omega[X] := \int_T X d \omega.$$

Recall that the Lebesgue integral $\int_T X d \omega$ need not exist, in which case $E_\omega[X]$ is undefined; otherwise, the integral equals a real number, $\infty$, or $-\infty$. 
The definition and some properties of the Lebesgue integral, specialized to our setting, 
can be found in Appendix~\ref{sect:lebesgue}.
Also recall that the expected value is 
\begin{itemize}
    \item \emph{linear}, i.e., for every $a,b \in \mathbb{R}$ and random variables $X$, $Y$ such that $E_\omega[X]$ and $E_\omega[Y]$ exist and $aE_\omega[X]+bE_\omega[Y]$ is defined we have $$E_\omega[aX+bY]=aE_\omega[X]+bE_\omega[Y];$$ 
\item \emph{monotone}, i.e., if $X,Y$ are random variables such that $E_\omega[X]$ and $E_\omega[Y]$ exist and $X(f) \leq Y(f)$ for all $f \in T$, then $E_\omega[X] \leq E_\omega[Y]$.
\end{itemize}

Let $C$ and $D$ be sets. In the rest of the article, we will work exclusively on a topological space $C^D$ of maps $f \colon D \rightarrow C$ and the special case where $D = C^\ell$ 
for some $\ell \in \N$.
Note that if $C$ and $D$ are infinite, then these spaces are uncountable and hence there are probability distributions $\omega$ such that $\omega(A)=0$ for every $1$-element set $A$. Therefore, in these cases, $E_\omega[X]$ for a random variable $X$ might not be expressible as a sum.

\begin{definition}[fractional homomorphism]\label{def:frac-hom}
Let $\Gamma$ and $\Delta$ be valued $\tau$-structures with domains $C$ and $D$, respectively. A \emph{fractional homomorphism} from
$\Delta$ to $\Gamma$ is a 
fractional map from $D$ to $C$
 such that for every $R \in \tau$ of arity $k$ 
and every tuple $a \in D^k$,
the expected value $E_\omega[X]$
of the random variable $X \colon C^D \rightarrow \mathbb{R}\cup\{\infty\}$ given  by
\[f \mapsto R^{\Gamma}(f(a))\] 
exists and satisfies 
$$E_\omega[X]
\leq R^{\Delta}(a).$$
\end{definition}
Note that fractional homomorphisms between valued structures are a natural generalization of homomorphisms between relational structures: if $\Delta$ and $\Gamma$ are $0$-$\infty$-valued and $\omega(\{f\}) = 1$ for some $f \colon D \to C$, then the definition requires that $R^\Gamma(f(a)) \leq R^\Delta(a)$ for every $a$, equivalently, whenever $R^\Delta(a)=0$, then $R^\Gamma(f(a))=0$, which is equivalent to $f$ being a homomorphism.

The following lemma shows that if $\Aut(\Gamma)$ is oligomorphic, then the expected value from Definition~\ref{def:frac-hom} always exists. 
\begin{lemma}\label{lem:E-exists}
Let $C$ and $D$ be sets, $a \in D^k$, and $R \in {\mathscr R}_C^{(k)}$.
If $X \colon C^D \rightarrow \mathbb{R}\cup\{\infty\}$ is the random variable given by
\[f \mapsto R(f(a))\]
and $\Aut(C; R)$ is oligomorphic, then $E_\omega[X]$ exists and $E_\omega[X] > -\infty$.
\end{lemma}
\begin{proof}
It is enough to show that $\int_{C^D} X^- d \omega \neq \infty$. Since $\Aut(C;R)$ is oligomorphic, there are only finitely many orbits of $k$-tuples in $\Aut(C;R)$. Let $O_1, \dots, O_m$ be all orbits of $k$-tuples of $\Aut(C; R)$ on which $R$ is negative. For every $i \in \{1, \dots, m\}$, let $b_i \in O_i$.
Then we obtain (see \eqref{eq:exp-sum} in Appendix~\ref{sect:lebesgue} for a detailed derivation of the first equality)
\begin{align*}
   \int_{C^D} X^- d \omega & =  \sum_{b \in C^k,R(b) < 0} -R(b) \omega({\mathscr S}_{a,b}) \\
   & = - \sum_{i=1}^m R(b_i) \sum_{b \in O_i} \omega({\mathscr S}_{a,b}) \\
   &= - \sum_{i=1}^m R(b_i) \; \omega \left( \bigcup_{b \in O_i}{\mathscr S}_{a,b} \right) \\
   & \leq -\sum_{i=1}^m R(b_i) < \infty.\qedhere 
\end{align*} 
\end{proof}

Let $T_1$ and $T_2$ be topological spaces. 
Then the product ${\mathcal B}(T_1) \otimes {\mathcal B}(T_2)$ is the  
smallest $\sigma$-algebra such that for $i\in\{1,2\}$, the projection $\pi_i \colon T_1 \times T_2 \to T_i$ given by
$\pi_i(U_1 \times U_2) := U_i$ is measurable.
For $i \in \{1,2\}$, let $\omega_i$ be 
a probability distribution on $T_i$. Then the probability distribution $\omega_1 \otimes \omega_2$ on $\mathcal{B}(T_1) \otimes \mathcal{B}(T_2)$ is the unique probability distribution that satisfies for all $U_1 \in \mathcal{B}(T_1)$ and $U_2 \in \mathcal{B}(T_2)$ that 
\[\omega_1 \otimes \omega_2(U_1 \times U_2) = \omega_1(U_1) \cdot \omega_2(U_2); \]
the existence and uniqueness of this measure
is folklore, see, e.g.,~\cite[Theorem 14.14]{ProbTheory-Klenke}.
A topological space $S$ is called \emph{second-countable} if it has a countable basis of open sets; note that if $C$ and $D$ are countable, then  the product space $C^D$ is second-countable. We provide the following statement with a proof for the convenience of the reader.

\begin{lemma}[{see, e.g.,~\cite[Volume II, Lemma 6.4.2]{BogachevBook} or~\cite[Lemma D.1.1]{GheysensThesis}}]\label{lem:borel-prod}
Let $S$ and $T$ be second-countable topological spaces. Then ${\mathcal B}(S) \otimes {\mathcal B}(T) = {\mathcal B}(S \times T)$.
\end{lemma}
\begin{proof}
The inclusion $\subseteq$ holds without the assumption of second-countability. To see this, note that ${\mathcal B}(S) \otimes {\mathcal B}(T)$ is the coarsest $\sigma$-algebra such that the projections $\pi_i$ are measurable whereas ${\mathcal B}(S \times T)$ is the Borel $\sigma$-algebra of the coarsest topology such that the projections are continuous. It is easy to see that if a map is continous with respect to a topology, then it is measurable over the Borel $\sigma$-algebra generated by this topology, hence we must have ${\mathcal B}(S) \otimes {\mathcal B}(T) \subseteq {\mathcal B}(S \times T)$.

To see the reverse inclusion, note that the product $S \times T$ is also second-countable with a countable basis $\mathscr{B}$ of open sets. Each element of $\mathscr B$ trivially belongs to ${\mathcal B}(S \times T)$: this is because the projections $\pi_i$ are measurable with respect to $\mathcal{B}(S) \otimes {\mathcal B}(T)$ and hence for every $S' \times T' \in \mathscr B$ the sets $S' \times T$ and $S \times T'$ are measurable and thus their intersection $S'\times T'$ is. Since $S \times T$ has a countable basis, every open set is a countable union of open sets and hence $\mathscr B$ generates $\mathcal{B}(S \times T)$. Therefore, we get ${\mathcal B}(S \times T) \subseteq {\mathcal B}(S) \otimes {\mathcal B}(T)$.
\end{proof} 

Also the following is considered to be known; again, we provide a proof for the convenience of the reader. 
\begin{lemma}\label{lem:comp}
Let $C_1, C_2, C_3$ be countable sets. The composition map $\circ \colon C_3^{C_2} \times C_2^{C_1} \to C_3^{C_1}$ is (Borel-)measurable.
\end{lemma}

\begin{proof}
Since we are working with Borel $\sigma$-algebras, 
it is enough to prove that $\circ$ is continuous. 
Therefore, it is enough to verify that preimages of basic open sets in $C_3^{C_1}$ are open. Let $a\in C_1^k$ and $c \in C_3^k$ for some $k \in \N$. Then
\[\circ^{-1} (\mathscr{S}_{a,c}) = \bigcup_{b \in C_2^k} (\mathscr{S}_{b,c} \times \mathscr{S}_{a,b}),\]
which is an open set in $C_3^{C_2} \times C_2^{C_1}$.
It follows that $\circ$ is continous and therefore measurable.
\end{proof}

\begin{definition}[composition of fractional maps]
Let $C_1, C_2, C_3$ be countable sets. Let $\omega_1$ be a fractional map from $C_1$ to $C_2$ and $\omega_2$ be a fractional map from $C_2$ to $C_3$. Then $\omega_2 \circ \omega_1 \colon \mathcal{B}(C_3^{C_1}) \to [0,1]$ is defined by 
\[\omega_2\circ \omega_1 (S) := \omega_2 \otimes \omega_1 (\{(g, f) \mid f \in C_2^{C_1}, g \in C_3^{C_2}, g \circ f \in S\}).\]
\end{definition}

Note that the set $\{(g, f) \mid f \in C_2^{C_1}, g \in C_3^{C_2}, g \circ f \in S\}$ is the
preimage of $S \in \mathcal{B}(C_3^{C_1})$ under $\circ$. The map $\circ \colon C_3^{C_2} \times C_2^{C_1} \to C_3^{C_1}$ is measurable by Lemma~\ref{lem:comp}, and therefore the preimage of $S$ under $\circ$ lies in  ${\mathcal B}(C_3^{C_2} \times C_2^{C_1})$. Since $C_3^{C_2}$ and $C_2^{C_1}$ are second-countable topological spaces, ${\mathcal B}(C_3^{C_2}) \otimes {\mathcal B}(C_2^{C_1}) = {\mathcal B}(C_3^{C_2} \times C_2^{C_1})$ by Lemma~\ref{lem:borel-prod}. Therefore, the set $\{(g, f) \mid f \in C_2^{C_1}, g \in C_3^{C_2}, g \circ f \in S\}$ lies in $ {\mathcal B}(C_3^{C_2}) \otimes {\mathcal B}(C_2^{C_1})$, and hence $\omega_2 \circ \omega_1$ is well-defined.
We proceed by proving several natural properties for the operation $\circ$. 

\begin{proposition}\label{prop:frac-comp}
Let $C_1, C_2, C_3$ be countable sets. Let $\omega_1$ be a fractional map from $C_1$ to $C_2$ and let $\omega_2$ be a fractional map from $C_2$ to $C_3$. Then the following statements hold.
\begin{enumerate}
    \item $\omega_2 \circ \omega_1$ is a fractional map from $C_1$ to $C_3$.
    \item For every $k \in \N$, $a \in C_1^k$, and $c \in C_3^k$, 
    \[\omega_2 \circ \omega_1 (\mathscr{S}_{a,c}) = \sum_{b \in C_2^k} \omega_1({\mathscr S}_{a,b}) \omega_2({\mathscr S}_{b,c}).\]
    \item The above property uniquely determines the fractional map $\omega_2 \circ \omega_1$.
    \item The operation $\circ$ of composition of fractional maps is associative.
    \item If $\Gamma_1, \Gamma_2$, and $\Gamma_3$ are valued $\tau$-structures, $\omega_1$ is a fractional homomorphism from $\Gamma_1$ to $\Gamma_2$, and $\omega_2$ is a fractional homomorphism from $\Gamma_2$ to $\Gamma_3$, then $\omega_2 \circ \omega_1$ is a fractional homomorphism from $\Gamma_1$ to $\Gamma_3$.
\end{enumerate} 
\end{proposition}

\begin{proof}
To prove item (1), note that 
$\omega_2 \circ \omega_1$ is the pushforward of $\omega_2 \otimes \omega_1$ by the measurable map $\circ \colon  C_3^{C_2} \times C_2^{C_1} \to C_3^{C_1}$; see, e.g.,~\cite[Exercise 1.4.38]{Tao2011Measure} for the definition.
This construction is known to yield a probability distribution (see again~\cite[Exercise 1.4.38]{Tao2011Measure}), in other words, a fractional map.

To show item (2), let $k \in \N$, $a \in C_1^k$, and $c \in C_3^k$. Then 
\begin{align*}
    \omega_2 \circ \omega_1 (\mathscr{S}_{a,c}) &=
    \omega_2 \otimes \omega_1(\{(g, f) \mid f \in C_2^{C_1}, g \in C_3^{C_2}, g \circ f \in \mathscr{S}_{a,c}\})
    \\ &=\omega_2 \otimes \omega_1(\{(g,f) \mid f \in C_2^{C_1}, g \in C_3^{C_2}, \exists b \in C_2^k: f(a)=b \wedge g(b) = c\}) \\
    &= \omega_2 \otimes \omega_1\left( \bigcup_{b \in C_2^k} \{(g, f) \mid f \in C_2^{C_1}, g \in C_3^{C_2}, f(a)=b \wedge g(b) = c\}  \right)  \\
    &= \sum_{b \in C_2^k} \omega_2 \otimes \omega_1\left( \mathscr{S}_{b,c} \times \mathscr{S}_{a,b} \right) \\
    &=\sum_{b \in C_2^k} \omega_2(\mathscr{S}_{b,c}) \cdot \omega_1(\mathscr{S}_{a,b}).
\end{align*}

Item (3) is a consequence of Dynkin's $\pi$-$\lambda$ theorem (see, e.g., \cite[Lemma 4.11]{aliprantis06}), because basic open sets form a $\pi$-system (a non-empty system closed under intersections). 

To show item (4), let $\omega_3$ be a fractional map from $C_3$ to a countable set $C_4$. Let $S \in \mathcal{B}(C_4^{C_1})$. We need to show $\omega_3 \circ (\omega_2 \circ \omega_1) (S) = (\omega_3 \circ \omega_2) \circ \omega_1(S)$. Note that it is enough to verify this property for basic open sets $S$, since then, again by Dynkin's $\pi$-$\lambda$ theorem, $\omega_3 \circ (\omega_2 \circ \omega_1) = (\omega_3 \circ \omega_2) \circ \omega_1$. Let $k \in \N$, $a \in C_1^k$ and $d \in C_4^k$. Then  
\begin{align*}
\omega_3 \circ(\omega_2 \circ \omega_1) (\mathscr{S}_{a,d}) &= \omega_3 \otimes (\omega_2 \circ \omega_1)\left(\left\{(h, h') \mid h \in C_4^{C_3}, h' \in C_3^{C_1}, h \circ h'(a) = d \right\}\right) \\
&= \omega_3 \otimes (\omega_2 \circ \omega_1) \left(\bigcup_{c \in C_3^k} \mathscr{S}_{c,d} \times \mathscr{S}_{a,c} \right) \\
&= \sum_{c \in C_3^k} \omega_3(\mathscr{S}_{c,d}) \cdot (\omega_2 \circ \omega_1)(\mathscr{S}_{a,c}) \\
&= \sum_{c \in C_3^k} \omega_3(\mathscr{S}_{c,d}) \cdot (\omega_2 \otimes \omega_1)\left( \left\{(g,f) \mid g \in C_3^{C_2}, f \in C_2^{C_1}, g \circ f(a)=c \right\} \right) \\
&= \sum_{c \in C_3^k} \omega_3(\mathscr{S}_{c,d}) \cdot (\omega_2 \otimes \omega_1) \left(\bigcup_{b \in C_2^k} \mathscr{S}_{b,c} \times \mathscr{S}_{a,b} \right) \\
&= \sum_{c \in C_3^k} \sum_{b \in C_2^k} \omega_3(\mathscr{S}_{c,d})  \cdot \omega_2(\mathscr{S}_{b,c}) \cdot \omega_1(\mathscr{S}_{a,b}) \\
&= \sum_{b \in C_2^k} \omega_1(\mathscr{S}_{a,b}) (\omega_3 \otimes \omega_2) \left(\bigcup_{c \in C_3^k} \mathscr{S}_{c,d} \times \mathscr{S}_{b,c}\right) \\
&= \dots = (\omega_3 \circ \omega_2) \circ \omega_1 (\mathscr{S}_{a,d})
\end{align*}
Note that the order of the summations can be exchanged by the discrete version of the Fubini-Tonelli theorem~\cite[Theorem 14.19]{ProbTheory-Klenke}, because the summands are non-negative. 

To prove item (5), let $R \in \tau$ be of arity $k$ and $a \in C_1^k$. Then by item (3) we have the following:
\begin{align*}
    E_{\omega_2 \circ \omega_1}[h \mapsto R(h(a))] &= \sum_{c \in C_3^k} \omega_2 \circ \omega_1(\mathscr{S}_{a,c}) \cdot R(c) \\
    &= \sum_{c \in C_3^k} R(c) \sum_{b\in C_2^k} \omega_2(\mathscr{S}_{b,c}) \cdot \omega_1(\mathscr{S}_{a,b}) \\
    &= \sum_{b\in C_2^k}\omega_1(\mathscr{S}_{a,b})  \sum_{c \in C_3^k} R(c)  \omega_2(\mathscr{S}_{b,c}) \\
    &= \sum_{b\in C_2^k}\omega_1(\mathscr{S}_{a,b}) E_{\omega_2}[g \mapsto R(g(b))] \\
    &\leq \sum_{b\in C_2^k}\omega_1(\mathscr{S}_{a,b}) \cdot R(b) \\
    &= E_{\omega_1}[f \mapsto R(f(a))] \\
    &\leq R(a),
\end{align*}
where the two inequalities follow from $\omega_1$ and $\omega_2$ being fractional homomorphisms.
The order of the summations can again be exchanged by the Fubini-Tonelli theorem.
It follows that $\omega_2 \circ \omega_1$ is a fractional homomorphism.
\end{proof}

The following property of fractional homomorphisms was shown for valued structures over finite domains in~\cite[Proposition 8.4]{Butti}.

\begin{proposition}
    \label{prop:frac-hom}
Let $\Gamma$ and $\Delta$ be valued $\tau$-structures 
with domains $C$ and $D$ and with a fractional homomorphism $\omega$ from $\Delta$ to $\Gamma$. Then the cost of every $\VCSP$ instance $\phi$ with respect to $\Gamma$ is at most the cost of $\phi$ with respect to $\Delta$. 
\end{proposition}

\begin{proof}
Let 
\[\phi(x_1,\dots,x_n) = \sum_{i=1}^m R_i(x_{j_1^i}, \dots, x_{j_{k_i}^i})\]
be a $\tau$-expression, where $j_1^i, \dots, j_{k_i}^i \in \{1, \dots, n\}$ for every $i \in \{1, \dots m\}$. To simplify the notation in the proof, if $v=(v_1, \dots, v_t)$ is a $t$-tuple of elements of some domain and $i_1,\dots, i_s \in \{1, \dots, t\}$, we will write $v_{i_1, \dots, i_s}$ for the tuple $(v_{i_1}, \dots, v_{i_s})$.

Let $\varepsilon > 0$. From the definition of infimum, there exists $a^*\in D^n$ such that 
\begin{equation}\label{eq:inf-tup}
    \phi^{\Delta}(a^*) \leq \inf_{a \in D^n} \phi^{\Delta}(a)+ \varepsilon/2
\end{equation} and $f^* \in C^D$ such that
\begin{equation}\label{eq:inf-op}
    \phi^{\Gamma}(f^*(a^*)) \leq \inf_{f\in C^D} \phi^{\Gamma} (f(a^*)) + \varepsilon/2.
\end{equation}

For every $i \in \{1, \dots, m\}$, $E_{\omega}[f \mapsto R_i^{\Gamma}(f(a^*)_{j_1^i, \dots, j_{k_i}^i})]$ exists  by the definition of a fractional homomorphism. Suppose first that $\sum_{i=1}^m E_{\omega}[f \mapsto R_i^{\Gamma}(f(a^*)_{j_1^i, \dots, j_{k_i}^i})]$ is defined.
Then by the monotonicity and linearity of $E_\omega$ and since $\omega$ is a fractional homomorphism we obtain 
\begin{align*}
    \inf_{b \in C^n} \phi^{\Gamma}(b) &\leq \phi^{\Gamma} (f^*(a^*)) \\
    & \leq \inf_{f \in C^D} \phi^{\Gamma} (f(a^*)) + \varepsilon/2  &&\text{(by \eqref{eq:inf-op}) }\\
    & \leq E_{\omega}[f \mapsto \phi^{\Gamma}(f(a^*))] + \varepsilon/2 \\
    &= \sum_{i=1}^m E_{\omega}[f \mapsto R_i^{\Gamma}(f(a^*)_{j_1^i, \dots, j_{k_i}^i})] + \varepsilon/2 \\
    & \leq \sum_{i=1}^m R_i^{\Delta}(a^*_{j_1^i, \dots, j_{k_i}^i}) + \varepsilon/2 \\
    &= \phi^{\Delta}(a^*) + \varepsilon/2  \\
    & \leq \inf_{a \in D^n} \phi^{\Delta}(a) + \varepsilon
    &&\text{(by \eqref{eq:inf-tup}).}
\end{align*}
Since $\varepsilon>0$ was chosen arbitrarily, it follows that the cost of $\phi$ with respect to $\Gamma$ is at most the cost of $\phi$ with respect to $\Delta$.

Suppose now that $\sum_{i=1}^m E_{\omega}[f \mapsto R_i^{\Gamma}(f(a^*)_{j_1^i, \dots, j_{k_i}^i})]$ is not defined. Then there exists $i \in \{1,\dots, m\}$ such that
\[E_\omega[f \mapsto R_i^\Gamma(f(a^*)_{j_1^i, \dots, j_{k_i}^i})] = \infty.\]
By the definition of a fractional homomorphism, this implies that $R_i^\Delta(a^*_{j_1^i, \dots, j_{k_i}^i})= \infty$ and hence $\sum_{i=1}^m R_i^{\Delta}(a^*_{j_1^i, \dots, j_{k_i}^i})=\infty$. Therefore, we obtain as above that \[\inf_{b \in C^n} \phi^{\Gamma}(b) \leq\inf_{a \in D^n} \phi^{\Delta}(a),\] which is what we wanted to prove. 
\end{proof}

\begin{remark}\label{rem:frac-hom-equiv-fin} 
For finite domains, the converse  of Proposition~\ref{prop:frac-hom} is true as well~\cite[Proposition 8.4]{Butti}. 
\end{remark}

We say that two relational $\tau$-structures $\bA$ and $\bB$ are \emph{homomorphically equivalent} if there is a homomorphism from $\bA$ to $\bB$ and from $\bB$ to $\bA$. 
We say that two valued $\tau$-structures  $\Gamma$ and $\Delta$ are \emph{fractionally homomorphically equivalent} if there exists a
fractional homomorphism from $\Gamma$ to $\Delta$ and from $\Delta$ to $\Gamma$. Clearly, fractional homomorphic equivalence is indeed an equivalence relation on valued structures of the same signature. 

\begin{corollary}\label{cor:hom-hard}
Let $\Gamma$ and $\Delta$ be valued $\tau$-structures with oligomorphic automorphism groups that are fractionally homomorphically equivalent.
Then $\VCSP(\Gamma)$ and $\VCSP(\Delta)$ are polynomial-time equivalent.  
\end{corollary}
\begin{proof}
In fact, the two problems $\VCSP(\Gamma)$ and $\VCSP(\Delta)$ coincide. By Proposition~\ref{prop:frac-hom}, for every instance $\phi$, the values of $\phi$ with respect to $\Gamma$ and $\Delta$ are equal. By Lemma~\ref{lem:cost} , the cost is attained in both structures and hence every instance $\phi$ with a threshold $u$ has a solution with respect to $\Gamma$ if and only if it has a solution with respect to $\Delta$.
\end{proof}

\begin{remark}
If $\Gamma$ and $\Delta$ are classical relational $\tau$-structures that are homomorphically equivalent in the classical sense,
then they are fractionally homomorphically equivalent when we view them 
as valued structures: if $h_1$ is the homomorphism from $\Gamma$ to $\Delta$ and $h_2$ is the homomorphism from $\Delta$ to $\Gamma$, 
then this is witnessed by the fractional homomorphisms $\omega_1$ and $\omega_2$ such that $\omega_1(h_1) = \omega_2(h_2) = 1$. 
\end{remark}

For fractionally homomorphically equivalent valued structures with an oligomorphic automorphism group, we can prove the following result of similar flavor as Proposition~\ref{prop:frac-hom}.

\begin{proposition}
\label{prop:frac-hom-expr}
Let $\Gamma$ and $\Delta$ be valued $\tau$-structures with oligomorphic automorphism groups 
that are fractionally homomorphically equivalent.
Let $\omega$ be a fractional homomorphism from $\Delta$ to $\Gamma$.
Let $R \in \langle \Delta \rangle$ be of arity $k$ and $R' \in \langle \Gamma \rangle$ be the valued relation obtained when the pp-expression for $R$ in $\Delta$ is interpreted over $\Gamma$. Then for every $a \in D^k$,
\begin{equation}
    E_{\omega}[f \mapsto R'(f(a))] \leq R(a).
    \label{eq:frac-hom}
\end{equation}
\end{proposition}
\begin{proof}
Recall that by Lemma~\ref{lem:E-exists}, the expected values from \eqref{eq:frac-hom} exist.
By the definition of a fractional homomorphism, \eqref{eq:frac-hom} holds for every pair $(R,R')=(S^\Delta, S^\Gamma)$ where $S \in \tau$. Clearly, the same is true for $R =R_{\emptyset}$. To see that \eqref{eq:frac-hom} holds for $R = R_=$, let $a \in D^2$. Note that either 
$a_1 = a_2$ in which case
$f(a_1)=f(a_2)$ for every $f \in C^D$, and hence both sides of \eqref{eq:frac-hom} are equal to $0$,
or $a_1 \neq a_2$, in which case $R(a) = \infty$ and \eqref{eq:frac-hom} is again satisfied.

We will show that every valued relation $R$ obtained from a valued relation satisfying \eqref{eq:frac-hom} by an application of an operator from Definition~\ref{def:wrelclone} satisfies \eqref{eq:frac-hom};
the general statement then follows by induction. This is clear for valued relations $R$ obtained by non-negative scaling and addition of constants, since these operations preserve \eqref{eq:frac-hom} by the linearity of expectation. The assumption that $\Gamma$ and $\Delta$ are fractionally homomorphically equivalent (rather than the existence of $\omega$) is needed only for the operator $\Opt$. 

Let $\phi(x_1,\dots,x_k, y_1, \dots, y_n)$ be a  $\tau$-expression. Let $R$ be the $k$-ary valued relation defined by $R(x)= \inf_{y\in D^n} \phi^{\Delta}(x,y)$  for every $x\in D^k$ . Since $\phi$ is a $\tau$-expression, there are $R_i\in \tau$ such that 
\[\phi(x_1,\dots,x_k, y_1, \dots, y_n)=\sum_{i=1}^m R_i(x_{p_1^i}, \dots, x_{p_{k_i}^i},y_{q_1^i}, \dots, y_{q_{n_i}^i})\]
for some $k_i, n_i \in \N$, $p_1^i, \dots, p_{k_i}^i\in \{1, \dots, k\}$ and $q_1^i, \dots, q_{n_i}^i\in \{1, \dots, n\}$.
In this proof, if $v=(v_1, \dots, v_t)$ is a tuple and $i_1,\dots, i_s \in \{1, \dots, t\}$, we will write $v_{i_1, \dots, i_s}$ for the tuple $(v_{i_1}, \dots, v_{i_s})$ for short. 

Let $a \in D^k$. By the oligomorphicity of $\Aut(\Delta)$, there is $b \in D^n$ such that $R(a) = \phi^{\Delta}(a, b)$. Moreover, for every $f \in C^D$, 
\[ R'(f(a)) \leq \phi^{\Gamma}(f(a), f(b)).\]
By the linearity and monotonicity of expectation, we obtain
\begin{align*} 
E_\omega[f \mapsto R'(f(a))] 
&\leq E_\omega[f \mapsto \phi^{\Gamma}(f(a), f(b))]  \\
&= E_\omega[f \mapsto \sum_{i=1}^{m} R_i^{\Gamma}((f(a))_{p_1^i, \dots, p_{k_i}^i}, (f(b))_{q_1^i, \dots, q_{n_i}^i})]\\
&= \sum_{i=1}^{m} E_\omega[f \mapsto R_i^{\Gamma}((f(a))_{p_1^i, \dots, p_{k_i}^i}, (f(b))_{q_1^i, \dots, q_{n_i}^i})]. 
\end{align*}
Since $\omega$ is a fractional homomorphism, the last row of the inequality above is at most
\begin{align*}
\sum_{i=1}^{m}
R_i^{\Delta}(a_{p_1^i, \dots, p_{k_i}^i}, b_{q_1^i, \dots, q_{n_i}^i})
 = \phi^{\Delta}(a, b) = R(a).
\end{align*}
It follows that \eqref{eq:frac-hom} holds for $R$.

Next, we prove the statement for $R = \Feas(S^{\Delta})$ for some $S \in \tau$ of arity $k$. Let $a \in D^k$. 
If $R(a) = \infty$, then \eqref{eq:frac-hom} is trivially true. So suppose that $R(a) = 0$, i.e., $S^{\Delta}(a) < \infty$. Since $\omega$ is a fractional homomorphism, we have
\begin{align}
E_\omega[f \mapsto S^{\Gamma}(f(a))]
\leq S^{\Delta}(a) 
\end{align}
and hence the expected value on the left-hand side is finite as well.
By \eqref{eq:exp-sum} in Appendix~\ref{sect:lebesgue},
\begin{align}
E_\omega[f \mapsto S^{\Gamma}(f(a))] =  \sum_{b \in C^k} S^{\Gamma}(b) \omega({\mathscr S}_{a,b}),
\end{align}
which implies that $S^{\Gamma}(b)$ is finite unless $\omega({\mathscr S}_{a,b})=0$, and hence $R'(b)=0$. Consequently (again by \eqref{eq:exp-sum}),
\begin{equation*}
E_\omega[f \mapsto R'(f(a))]
= \sum_{b \in C^k} R'(b) \omega({\mathscr S}_{a,b}) = 0 = R(a). 
\end{equation*}
It follows that \eqref{eq:frac-hom} holds for $R$.

Finally, suppose that $R = \Opt(S^{\Delta})$. Let $a \in D^k$; note that we may again assume that $R(a)=0$ as we did in the previous case. This means that $S^{\Delta}(a) \leq S^{\Delta}(a')$ for every $a' \in D^k$. Let $c \in C^k$ be such that $S^{\Gamma}(c)$ is minimal; such a $c$ exists by the oligomorphicity of $\Aut(\Gamma)$.
By the monotonicity of expected value and since $\omega$ is a fractional homomorphism, we have
\begin{align}
S^{\Gamma}(c) \leq E_\omega[f \mapsto S^{\Gamma}(f(a))]=
\sum_{b \in C^k} S^{\Gamma}(b) \omega({\mathscr S}_{a,b}) \leq S^{\Delta}(a).
\label{eq:frac-hom-opt}
\end{align}
If $\omega'$ is a fractional homomorphism from $\Gamma$ to $\Delta$ (which exists by assumption) and since $a \in \Opt(S^\Delta)$, we also have
\begin{align*}
S^{\Delta}(a) \leq E_{\omega'}[g \mapsto S^{\Delta}(g(c))] \leq S^{\Gamma}(c).
\end{align*}
This implies that $S^\Delta(a) = S^{\Gamma}(c)$.
Since $\omega$ is a probability distribution, we obtain from \eqref{eq:frac-hom-opt} that $S^{\Gamma}(b) = S^{\Gamma}(c)$ unless $\omega({\mathscr S}_{a,b}) = 0$, and hence $R'(b) = \Opt(S^\Gamma)(b) = 0$. Therefore,  
\begin{align*}
E_\omega[f \mapsto R'(f(a))]
= \sum_{b \in C^k} R'(b) \omega({\mathscr S}_{a,b}) = 0 = R(a).
\end{align*}  
This concludes the proof.
\end{proof}

We now have all ingredients needed to introduce the notion of a pp-construction.

\begin{definition}[pp-construction]
Let $\Gamma, \Delta$ be valued structures. Then  $\Delta$ has a \emph{pp-construction} in $\Gamma$ if $\Delta$ 
is fractionally homomorphically equivalent to a structure $\Delta'$ which is a pp-power of $\Gamma$. 
\end{definition}

Combining Proposition~\ref{prop:pp-hard} and Corollary~\ref{cor:hom-hard} gives the 
following.
\begin{corollary}\label{cor:pp-constr-red}
Let $\Gamma$ and $\Delta$ be valued structures 
with finite signatures and oligomorphic automorphism groups such that $\Delta$ has a pp-construction in $\Gamma$. Then there is a  polynomial-time reduction from 
$\VCSP(\Delta)$ to $\VCSP(\Gamma)$.
\end{corollary}

Note that the hardness proofs in Examples~\ref{expl:mc-hard} and~\ref{expl:mcc-hard} are special cases of Corollary~\ref{cor:pp-constr-red}. A more involved example uses the
relational structure with a hard CSP introduced below and is shown in Example~\ref{expl:triad}.

Let $\OIT$ be the following relation 
$$ \OIT = \{(0,0,1),(0,1,0),(1,0,0) \}.$$
It is well-known (see, e.g.,~\cite{Book}) that $\CSP(\{0,1\};\OIT)$ is NP-complete.
Combining Corollary~\ref{cor:pp-constr-red} with
the NP-hardness of $\CSP(\{0,1\};\OIT)$ we obtain:
\begin{corollary}\label{cor:OIT}
Let $\Gamma$ be a valued structure with a finite signature and oligomorphic automorphism group such that 
$(\{0,1\};\OIT)$ 
has a pp-construction in $\Gamma$. Then $\VCSP(\Gamma)$ is NP-hard. 
\end{corollary}

Since Corollary~\ref{cor:OIT} is our typical tool to prove hardness, we are interested in pp-constructing relational structures in valued structures. If $\Gamma$ is a valued structure on a domain $C$, then we call the relational structure on the domain $C$ whose relations are all classical relations (i.e., with values in $\{0,\infty\}$) in $\langle \Gamma \rangle$ the \emph{underlying crisp structure of $\Gamma$}. The following proposition first appeared in \cite{tvcsp} and shows that for hardness proofs by Corollary~\ref{cor:OIT}, it is equivalent to check pp-constructability in the underlying crisp structure.

\begin{proposition}[{\cite[Proposition 3.2]{tvcsp}}]
\label{prop:pp-constr-rel}
Let $\Gamma$ be a~valued structure and let $\bB$ be a~relational structure, both with countable domains. Then $\Gamma$ pp-constructs $\bB$ if and only if the underlying crisp structure of $\Gamma$ pp-constructs $\bB$.
\end{proposition}

We finish this section by the following useful lemma.

\begin{lemma}\label{lem:pp-trans}
The relation of pp-constructability on the class of countable valued structures is transitive. 
\end{lemma}
\begin{proof}
Clearly, a pp-power of a pp-power is again a pp-power, and fractional homomorphic equivalence is transitive by Proposition~\ref{prop:frac-comp}.
We are therefore left to prove that if $\Gamma$ and $\Delta$ are valued structures such that $\Delta$ is a $d$-dimensional pp-power of $\Gamma$,
and if $\Gamma'$ is fractionally homomorphically equivalent to $\Gamma$ via fractional homomorphisms $\omega_1 \colon \Gamma \to \Gamma'$ and $\omega_2 \colon \Gamma' \to \Gamma$, then $\Delta$ also has a  pp-construction in $\Gamma'$.  

Let $C$ and $C'$ be the domains of $\Gamma$ and $\Gamma'$, respectively. Take the pp-expressions that define the valued relations of $\Delta$ over $\Gamma$, and interpret them over $\Gamma'$ instead of $\Gamma$;
let $\Delta'$ be the resulting valued structure. Note that $\Delta'$ is a $d$-dimensional pp-power of $\Gamma'$.
For a map $f \colon \Gamma \to \Gamma'$, let $\tilde f \colon \Delta \to \Delta'$ be given by 
 $(x_1,\dots,x_{d}) \mapsto (f(x_1),\dots,f(x_{d}))$. 
Then for all $S \in B((C')^{C})$
we define 
$$\tilde \omega_1 (\{\tilde f \mid f \in S\}) := \omega_1(S)$$
and 
\[\tilde \omega_1(\tilde S):=\tilde \omega_1 (\tilde{S} \cap \{\tilde f \mid f \in (C')^C\})\]
for all $\tilde S \in B \big(((C')^d)^{C^d} \big)$.
We argue that $\tilde \omega_1$ is a fractional homomorphism from $\Delta$ to $\Delta'$. To see this, let $R$ be a valued relation of $\Delta$ and $R'$ be the corresponding valued relation of $\Delta'$. If we view $R$ as an element of $\langle \Gamma \rangle$, then Proposition~\ref{prop:frac-hom-expr} applied to $\omega_1$ implies precisely that $\tilde \omega_1$ is a fractional homomorphism from $\Delta$ to $\Delta'$. 
Analogously we obtain from $\omega_2$ a fractional homomorphism $\tilde \omega_2$ from $\Delta'$ to $\Delta$. Therefore, $\Delta$ is fractionally homomorphically equivalent to $\Delta'$, which is a pp-power of $\Gamma'$. In other words, $\Delta$ has a pp-construction in $\Gamma'$.
\end{proof}

\section{Fractional Polymorphisms}

\label{sect:fpol}
In this section we introduce  \emph{fractional polymorphisms} of valued structures; they are an important tool for formulating tractability results and complexity classifications of VCSPs.
For valued structures with a finite domain, our definition specializes to the established notion of a fractional polymorphism which has been used 
to study the complexity of VCSPs for valued structures over finite domains (see, e.g.~\cite{ThapperZivny13}). 
Our approach is different from the one of Schneider and Viola~\cite{ViolaThesis,SchneiderViola} and Viola and \v{Z}ivn\'{y} \cite{ViolaZivny} in that we work with arbitrary probability spaces instead of distributions with finite support or countable additivity property.  
As we will see in Section~\ref{sect:tract}, fractional polymorphisms can be used to give sufficient conditions for tractability of $\VCSP$s of certain valued structures with oligomorphic automorphism groups. This justifies our more general notion of a fractional polymorphism, as it might provide a tractability proof for more problems.

Let ${\mathscr O}_C^{(\ell)}$ be the set of all operations $f \colon C^{\ell} \to C$ on a set $C$ of arity $\ell$.
We equip ${\mathscr O}_C^{(\ell)}$ with the topology of pointwise convergence, where $C$ is taken to be discrete. That is, the basic open sets are of the form
\begin{align}
{\mathscr S}_{a^1,\dots,a^{\ell}, b} := \{ f \in {\mathscr O}^{(\ell)}_C \mid f(a^1,\dots,a^{\ell}) = b \}
\label{eq:S}
\end{align}
where $a^1,\dots,a^{\ell}, b\in C^{k}$, for some $k \in {\mathbb N}$, and $f$ is applied componentwise. 
Let $${\mathscr O}_C := \bigcup_{\ell \in {\mathbb N}} {\mathscr O}_C^{(\ell)}.$$

\begin{definition}[fractional operation]
Let $\ell \in {\mathbb N}$. 
A \emph{fractional operation on $C$ of arity $\ell$} is a probability distribution 
$$\big({\mathscr O}_C^{(\ell)},B({\mathscr O}_C^{(\ell)}), \omega \colon B({\mathscr O}_C^{(\ell)}) \to [0,1] \big).$$ 
The set of all fractional operations on $C$ of arity $\ell$ is denoted by ${\mathscr F}^{(\ell)}_C$, and ${\mathscr F}_C := \bigcup_{\ell \in {\mathbb N}} {\mathscr F}^{(\ell)}_C$.
\end{definition}

If the reference to $C$ is clear, we occasionally omit the subscript~$C$. 
We often use $\omega$ for both the entire fractional operation and for the map $\omega \colon B({\mathscr O}_C^{(\ell)}) \to [0,1]$. Note that fractional operations are simply fractional maps where $D=C^\ell$ (see Definition~\ref{def:frac-map}).

\begin{definition}\label{def:pres}
A fractional operation $\omega \in {\mathscr F}_C^{(\ell)}$ \emph{improves}
a $k$-ary valued relation $R \in {\mathscr R}^{(k)}_C$ if for all $a^1,\dots,a^{\ell} \in C^k$ 
$$E := E_\omega[f \mapsto R(f(a^1,\dots,a^{\ell}))]$$ 
exists and 
\begin{align}
E \leq
\frac{1}{\ell} \sum_{j = 1}^{\ell} R(a^j). \label{eq:fpol}
\end{align}
\end{definition}

Note that~\eqref{eq:fpol} has the interpretation that 
the expected value of $R(f(a^1,\dots,a^\ell))$ is at most the average of the values $R(a^1),\dots$, $R(a^\ell)$. 
Also note that if $R$ is a classical relation improved by a fractional operation $\omega$
and $\omega(f) > 0$ for $f \in {\mathscr O}^{(\ell)}$, then $f$ must preserve $R$ in the usual sense.
It follows from Lemma~\ref{lem:E-exists} that if $\Aut(C; R)$ is oligomorphic, then $E_\omega[f \mapsto R(f(a^1, \dots, a^{\ell}))]$ always exists and is greater than $-\infty$.

\begin{definition}[fractional polymorphism]
If $\omega$ improves every valued relation in $\Gamma$, 
then $\omega$ is called a \emph{fractional polymorphism 
of $\Gamma$}; the set of all fractional polymorphisms of $\Gamma$ is denoted by $\fPol(\Gamma)$.
\end{definition}

\begin{remark}
   Our notion of fractional polymorphism coincides with the previously used notions of fractional polymorphisms with finite support \cite{SchneiderViola, ViolaThesis} or the countable additivity property \cite{ViolaZivny}, since in this case the expected value on the left-hand side of \eqref{eq:fpol} is equal to the weighted arithmetic mean.
\end{remark}

 \begin{remark}\label{rem:frac-pol-frac-hom}
A fractional polymorphism of arity $\ell$ of a valued structure $\Gamma$ might also be viewed as a fractional homomorphism from a specific $\ell$-th pp-power of $\Gamma$ to $\Gamma$, which we denote by $\Gamma^{\ell}$: 
if $C$ is the domain and $\tau$ the signature of $\Gamma$, 
then the domain of $\Gamma^{\ell}$ is $C^{\ell}$,
and for every $R \in \tau$ of arity $k$ we have
$$R^{\Gamma^{\ell}}((a^1_1,\dots,a^1_{\ell}),\dots,(a^k_1,\dots,a^k_{\ell})) 
:= \frac{1}{\ell} \sum_{i=1}^\ell R^{\Gamma}(a^1_i,\dots,a^{k}_i).$$
\end{remark} 

\begin{expl}\label{expl:id}
Let $\pi^\ell_i \in {\mathscr O}^{(\ell)}_C$ be the $i$-th projection of arity $\ell$, which is given by $\pi^\ell_i(x_1,\dots,x_\ell) = x_i$.
The fractional operation $\Id_\ell$ of arity $\ell$ such that $\Id_\ell(\pi^{\ell}_{i}) = \frac{1}{\ell}$ 
for every $i \in \{1,\dots,\ell\}$ is a fractional polymorphism of every valued structure with domain $C$. 
\end{expl}

\begin{expl}\label{expl:aut}
Let $\Gamma$ be a valued structure and $\alpha\in \Aut(\Gamma)$. The fractional operation $\omega \in \mathscr{F}_C^{(1)}$ defined by $\omega(\alpha)=1$ is a fractional polymorphism of $\Gamma$.
\end{expl}

Let ${\mathscr C} \subseteq {\mathscr F}_C$. 
We write ${\mathscr C}^{(\ell)}$ for 
${\mathscr C} \cap {\mathscr F}^{(\ell)}_C$
and $\Imp({\mathscr C})$ for the set of valued relations that are improved by every fractional operation in 
${\mathscr C}$.

\begin{lemma}
\label{lem:aut}
Let $R \in {\mathscr R}^{(k)}_C$ and let $\Gamma$ be a valued structure with domain $C$ and an automorphism $\alpha \in \Aut(\Gamma)$ which does not preserve~$R$. Then 
$R \notin \Imp(\fPol(\Gamma)^{(1)})$.
\end{lemma}
\begin{proof}
Since $\alpha$ does not preserve $R$, there exists $a \in C^k$ such that $R(a) \neq R(\alpha(a))$. If $R(\alpha(a)) > R(a)$, then 
let $\omega \in {\mathscr F}_C^{(1)}$ be the fractional operation 
defined by $\omega(\alpha) = 1$. 
Then $\omega$ improves every valued relation of $\Gamma$ and does not improve $R$. If $R(\alpha(a)) < R(a)$, then the fractional polymorphism $\omega$ of $\Gamma$ given by $\omega(\alpha^{-1}) = 1$ does not improve $R$. 
\end{proof}

Parts of the arguments in the proof of the following lemma can be found in the proof of~\cite[Lemma 7.2.1]{ViolaThesis}; note that the author works with a more restrictive notion of fractional operation, so we cannot reuse her result. However, the arguments can be generalized to our notion of fractional polymorphism for all countable valued structures.

\begin{lemma}\label{lem:easy}
For every 
valued $\tau$-structure $\Gamma$ 
over a countable domain $C$
we have 
$\langle \Gamma \rangle \subseteq \Imp(\fPol(\Gamma))$.
\end{lemma}

\begin{proof}
Let $\omega \in \fPol(\Gamma)^{(\ell)}$. 
By definition, $\omega$ improves every valued relation $R$ of $\Gamma$. 
It is clear that $\omega$ also preserves $\phi_{\emptyset}$. To see that $\omega$ preserves $R_=$, let $a^1,\dots,a^{\ell} \in C^2$. Note that either 
$a^i_1 = a^i_2$ for every $i \in \{1,\dots,\ell\}$, in which case
$f(a^1_1,\dots,a^{\ell}_1)=f(a^1_2,\dots,a^{\ell}_2)$ for every $f \in {\mathscr O}_C^{(\ell)}$, and hence 
$$ E_\omega[f \mapsto R_=(f(a^1,\dots,a^{\ell}))]
= 0 = \frac{1}{\ell} \sum_{j = 1}^{\ell} R_=(a^j),
$$
or $a^i_1 \neq a^i_2$ for some $i \in \{1,\dots,\ell\}$, in which case $\frac{1}{\ell} \sum_{j = 1}^{\ell} R_=(a^j) = \infty$ and the inequality 
in~(\ref{eq:fpol}) is again satisfied. 

The statement is also clear for valued relations
obtained from valued relations in $\Gamma$
by non-negative scaling and addition of constants, since these operations preserve the inequality in~(\ref{eq:fpol}) by the linearity of expectation. 

Let $\phi(x_1,\dots,x_k, y_1, \dots, y_n)$ be a  $\tau$-expression. We need to show that 
the fractional operation $\omega$ improves the $k$-ary valued relation $R$ defined for every $a\in C^k$ by $R(a)= \inf_{b\in C^n} \phi^{\Gamma}(a,b)$. Since $\phi$ is a $\tau$-expression, there are $R_i\in \tau$ such that 
\[\phi(x_1,\dots,x_k, y_1, \dots, y_n)=\sum_{i=1}^m R_i(x_{p_1^i}, \dots, x_{p_{k_i}^i},y_{q_1^i}, \dots, y_{q_{n_i}^i})\]
for some $k_i, n_i \in \N$, $p_1^i, \dots, p_{k_i}^i\in \{1, \dots, k\}$ and $q_1^i, \dots, q_{n_i}^i\in \{1, \dots, n\}$.

In this paragraph, if $v=(v_1, \dots, v_t) \in C^t$ and $i_1,\dots, i_s \in \{1, \dots, t\}$, we will write $v_{i_1, \dots, i_s}$ for the tuple $(v_{i_1}, \dots, v_{i_s})$ for short. Let $a^1, \dots, a^{\ell} \in C^k$. Let $\varepsilon > 0$ be a rational number. From the definition of an infimum, for every $j \in \{1, \dots, \ell\}$, there is $b^j\in C^n$ such that
\[ R(a^j) \leq \phi(a^j, b^j) < R(a^j)+ \varepsilon.\] Moreover, for every $f \in {\mathscr O}_C^{(\ell)}$, 
\[ R(f(a^1, \dots, a^{\ell})) \leq \phi(f(a^1, \dots, a^{\ell}), f(b^1, \dots, b^{\ell})).\]
By linearity and monotonicity of expectation, we obtain
\begin{align*} 
E_\omega[f \mapsto R(f(a^1,\dots,a^{\ell}))] 
&\leq E_\omega[f \mapsto \phi(f(a^1,\dots,a^{\ell}), f(b^1, \dots, b^{\ell}))]  \\
&= E_\omega[f \mapsto \sum_{i=1}^{m} R_i((f(a^1,\dots,a^{\ell}))_{p_1^i, \dots, p_{k_i}^i}, (f(b^1, \dots, b^{\ell}))_{q_1^i, \dots, q_{n_i}^i})]\\
&= \sum_{i=1}^{m} E_\omega[f \mapsto R_i((f(a^1,\dots,a^{\ell}))_{p_1^i, \dots, p_{k_i}^i}, (f(b^1, \dots, b^{\ell}))_{q_1^i, \dots, q_{n_i}^i})].  
\end{align*}
Since $\omega$ improves $R_i$ for every $i\in \{1, \dots, m\}$, the last row of the inequality above is at most
\begin{align*}
\sum_{i=1}^{m} \frac{1}{\ell} \sum_{j=1}^{\ell} 
R_i(a^j_{p_1^i, \dots, p_{k_i}^i}, b^j_{q_1^i, \dots, q_{n_i}^i})
&= \frac{1}{\ell} \sum_{j=1}^{\ell} \sum_{i=1}^{m}
R_i(a^j_{p_1^i, \dots, p_{k_i}^i}, b^j_{q_1^i, \dots, q_{n_i}^i}) \\
&= \frac{1}{\ell} \sum_{j=1}^{\ell} \phi(a^j, b^j) \\
&< \frac{1}{\ell} \sum_{j=1}^{\ell} R(a^j) + \varepsilon. 
\end{align*}
Since $\varepsilon$ was arbitrary, it follows that $\omega$ improves $R$.

 Finally, we prove that $\Imp(\fPol(\Gamma))$ is closed under $\Feas$ and $\Opt$. 
 Let $R \in \tau$ be of arity $k$ and define $S= \Feas(R)$ and $T = \Opt(R)$. We aim to show that $S, T \in \Imp(\fPol(\Gamma))$. Let $s^1,\dots,s^{\ell} \in C^k$.
 If $S(s^i) = \infty$ for some $i \in \{1,\dots,\ell\}$,
 then $\frac{1}{\ell} \sum_{j = 1}^{\ell} S(s^j) = \infty$ 
 and hence $\omega$ satisfies~\eqref{eq:fpol} (with $R$ replaced by $S$) for the tuples $s^1,\dots,s^{\ell}$.
So suppose that $S(s^i) = 0$ for all $i \in \{1,\dots,\ell\}$, i.e., $R(s^i)$ is finite for all $i$. Since $\omega$ improves $R$ it holds that
\begin{align}\label{eq:imp-R}
E_\omega[f \mapsto R(f(s^1,\dots,s^{\ell}))]
\leq \frac{1}{\ell} \sum_{j = 1}^{\ell} R(s^j) 
\end{align}
and hence the expected value on the left-hand side is finite as well.
By \eqref{eq:exp-sum} in Appendix~\ref{sect:lebesgue},
\begin{align}\label{eq:E-expr}
E_\omega[f \mapsto R(f(s^1,\dots,s^{\ell}))] =  \sum_{t \in C^k} R(t) \omega({\mathscr S}_{s^1,\dots,s^{\ell}, t}),
\end{align}
which implies that $R(t)$ is finite and $S(t)=0$ unless $\omega({\mathscr S}_{s^1,\dots,s^{\ell}, t})=0$. Consequently (again by \eqref{eq:exp-sum}),
\begin{equation*}
E_\omega[f \mapsto S(f(s^1,\dots,s^{\ell}))]
= \sum_{t \in C^k} S(t) \omega({\mathscr S}_{s^1,\dots,s^{\ell}, t}) = 0 = \frac{1}{\ell} \sum_{j = 1}^{\ell} S(s^j). 
\end{equation*}
It follows that $\omega$ improves $S$.

Moving to the valued relation $T$, we may again assume without loss of generality that $T(s^i)=0$ for every $i \in \{1, \dots, \ell\}$ as we did for $S$. This means that $c := R(s^1) = \cdots = R(s^\ell) \leq R(b)$ for every $b \in C^k$. 
Therefore, the right-hand side in \eqref{eq:imp-R} is equal to $c$ and by combining it with \eqref{eq:E-expr} we get 
\begin{align*}
    \sum_{t \in C^k} R(t) \omega({\mathscr S}_{s^1,\dots,s^{\ell}, t}) \leq c.
\end{align*}
Together with the assumption that $R(t) \geq c$ for all $t \in C^k$ and $\omega$ being a probability distribution we obtain that $R(t) = c$ and $T(t) = 0$ unless $\omega({\mathscr S}_{s^1,\dots,s^{\ell}, t}) = 0$, and hence  
\begin{align*}
E_\omega[f \mapsto T(f(s^1,\dots,s^{\ell}))]
= \sum_{t \in C^k} T(t) \omega({\mathscr S}_{s^1,\dots,s^{\ell}, t}) = 0 = \frac{1}{\ell} \sum_{j = 1}^{\ell} T(s^j).
\end{align*}  
This concludes the proof that $\omega$ improves $T$.
\end{proof}

The following example shows an application of Lemma~\ref{lem:easy}.

\begin{expl}
Let $<$ be the binary relation on $\{0,1\}$ and $\Gamma_<$ the valued structure from Example~\ref{expl:vs-mc}. By definition, $\Opt(<) \in \langle \Gamma_< \rangle$. Denote the minimum operation on $\{0,1\}$ by $\min$ and let $\omega$ be a binary fractional operation defined by $\omega(\min)=1$. Note that $\omega \in \fPol(\{0,1\}; \Opt(<))$. However, 
\[< \left(\min \left( \begin{pmatrix}0\\1\end{pmatrix} , \begin{pmatrix}0\\0\end{pmatrix} \right) \right)= \, {<}(0, 0) = 1,\]
while $(1/2) \cdot {<}(0,1) + (1/2) \cdot {<}(0,0)= 1/2$. This shows that $\omega$ does not improve $<$ and hence $< \; \not \in \langle (\{0,1\}; \Opt(<)) \rangle$ by Lemma~\ref{lem:easy}. 
\end{expl}

\section{Tractability via Canonical Fractional Polymorphisms} \label{sect:tract}

In this section we make use of a tractability result for finite-domain VCSPs of
Kolmogorov, Krokhin, and Rol\'{i}nek~\cite{KolmogorovKR17},  building on earlier work of Kolmogorov, Thapper, and 
\v{Z}ivn\'y~\cite{KolmogorovThapperZivny15,ThapperZivny13}. To exploit this result, the key ingredient is a polynomial-time reduction for VCSPs of valued structures with an oligomorphic automorphism group satisfying certain assumptions to VCSPs of finite-domain structures. This reduction is inspired by a similar reduction in the classical relational setting~\cite{BodMot-Unary}. We then generalize and adapt several statements for finite-domain valued structures to be able to use the complexity classification for finite-domain VCSPs.

\begin{definition}
An operation $f \colon C^{\ell} \to C$ for $\ell \geq 2$ is called 
\emph{cyclic}
if $$f(x_1,\dots,x_{\ell}) = f(x_2,\dots,x_{\ell},x_1)$$ for all $x_1,\dots,x_{\ell} \in C$.
Let $\Cyc_C^{(\ell)} \subseteq {\mathscr O}_C^{(\ell)}$ be the set of all operations on $C$ of arity $\ell$ that are cyclic.
\end{definition} 

If $G$ is a permutation group on a set $C$, then $\overline G$ denotes the closure of $G$ in the space
of functions from $C \to C$ with respect to the topology of pointwise convergence. Note that 
$\overline G$ might contain some operations that are not surjective, but if $G = \Aut(\bB)$ for some structure $\bB$, then all operations in $\overline G$ are still embeddings of $\bB$ into $\bB$ that preserve all first-order formulas. 

\begin{definition}
Let $G$ be a permutation group on the set $C$. 
An operation $f \colon C^{\ell} \to C$ is called \emph{pseudo cyclic with respect to $G$} if 
there are $e_1,e_2 \in \overline G$ such that for all $ x_1,\dots,x_{\ell} \in C$
$$e_1(f(x_1,\dots,x_{\ell})) = e_2(f(x_2,\dots,x_{\ell},x_1)).$$
Let $\PC_G^{(\ell)} \subseteq {\mathscr O}_C^{(\ell)}$ be the set of all operations on $C$ of arity $\ell$ that are pseudo cyclic with respect to~$G$. 
\end{definition} 

Note that $\PC_G^{(\ell)} \in B({\mathscr O}^{(\ell)}_C)$. Indeed, the complement of the set $\PC_G^{(\ell)}$ in ${\mathscr O}^{(\ell)}_C$ can be written as a countable union of sets of the form ${\mathscr S}_{a^1,\dots,a^{\ell},b}$ where for all 
$f \in {\mathscr O}_C^{(\ell)}$ 
the tuples $f(a^1,\dots,a^{\ell})$ and $f(a^2,\dots,a^{\ell},a^1)$ lie in different orbits with respect to $G$.

\begin{definition}
Let $G$ be a permutation group with domain $C$. 
An operation $f \colon C^{\ell} \to C$ for $\ell \geq 2$ is called \emph{canonical with respect to $G$} if for all $k \in {\mathbb N}$ and $a^1,\dots,a^{\ell} \in C^k$ the orbit of the $k$-tuple
$f(a^1,\dots,a^{\ell})$ only depends on the orbits of $a^1,\dots,a^{\ell}$ with respect to~$G$. Let $\Can_G^{(\ell)} \subseteq {\mathscr O}_C^{(\ell)}$ be the set of all operations on $C$ of arity $\ell$ that are canonical with respect to $G$. 
\end{definition} 

\begin{remark}\label{rem:can-act}
Note that if 
 $h$ is an operation over $C$ of arity $\ell$ which is canonical with respect to $G$,  then 
$h$ induces for every $k \in {\mathbb N}$ an operation $h^*$ of arity $\ell$ on the orbits of $k$-tuples
of $G$.
Note that if $h$ is pseudo cyclic with respect to $G$, then $h^*$ is cyclic.
\end{remark}

Note that $\Can_G^{(\ell)} \in B({\mathscr O}^{(\ell)}_C)$, since the complement is a countable union of sets of the form ${\mathscr S}_{a^1,\dots,a^{\ell},b} \cap {\mathscr S}_{c^1,\dots,c^{\ell},d}$ where for all $i \in \{1,\dots,\ell\}$ the tuples $a^i$ and $c^i$ lie in the same orbit with respect to~$G$, but $b$ and $d$ do not.

\begin{definition}
A fractional operation $\omega$ is called \emph{pseudo cyclic with respect to $G$} if for every 
$A \in B({\mathscr O}_C^{(\ell)})$ we have $\omega(A) = \omega(A \cap \PC^{(\ell)}_{G})$.
Canonicity with respect to $G$ and cyclicity for fractional operations
are defined analogously.
\end{definition}

We refer to Section~\ref{sect:inf-dual} 
for examples of concrete fractional polymorphisms of valued structures $\Gamma$ that are cyclic and canonical with respect to $\Aut(\Gamma)$. 
We may omit the specification 
`with respect to $G$' when $G$ is clear from the context.

We  prove below that 
canonical pseudo cyclic fractional polymorphisms imply polynomial-time tractability of the corresponding VCSP, by reducing to a tractable VCSP over a finite domain. Motivated by Theorem~\ref{thm:fb-NP} and the infinite-domain tractability conjecture~\cite[Conjecture 1.2]{BPP-projective-homomorphisms}, we state these results
for valued structures related to finitely bounded homogeneous structures.

\begin{definition}[$\Gamma^*_{\bB,m}$]\label{def:typetemplate}
Let $\Gamma$ be a valued structure with signature $\tau$ such that $\Aut(\Gamma)$ contains the automorphism group of a finitely bounded homogeneous structure $\bB$ with a finite relational signature.
Let $m$ be at least as large as the maximal arity of the relations of $\Gamma$.
Let $\Gamma^*_{\bB,m}$ be the following valued structure.
\begin{itemize}
\item The domain of $\Gamma^*_{\bB,m}$ is the set
of orbits of $m$-tuples of $\Aut(\bB)$. 
\item For every $R \in \tau$ of arity $k \leq m$ 
the signature of $\Gamma_{\bB,m}^*$ contains a unary relation symbol $R^*$, which denotes 
in $\Gamma_{\bB,m}^*$ the unary valued relation that returns on the orbit of an $m$-tuple $t=(t_1, \dots, t_m)$ the value of  
$R^{\Gamma}(t_1,\dots,t_k)$ (this is well-defined as the value is the same for all representatives $t$ of the orbit). 
\item For every $p \in \{1,\dots,m\}$ 
and $i,j \colon \{1,\dots,p\} \to \{1,\dots,m\}$ there exists a binary relation $C_{i,j}$ which returns $0$ for two orbits of $m$-tuples $O_1$ and $O_2$ if 
for every $s \in O_1$ and $t \in O_2$ we have that 
$(s_{i(1)},\dots,s_{i(p)})$
and $(t_{j(1)},\dots,t_{j(p)})$ 
 lie in the same orbit of $p$-tuples of $\Aut(\bB)$, and returns $\infty$ otherwise.
 \end{itemize}
\end{definition}

Note that $\Aut(\bB)$ 
has finitely many orbits of $k$-tuples for every $k \in \N$ and therefore $\Gamma^*_{\bB,m}$ has a finite domain.
The following reduction is inspired by a known reduction for CSPs from~\cite{BodMot-Unary}. 

\begin{theorem}\label{thm:red-fin}
Let $\Gamma$ be a valued structure with a finite signature such that $\Aut(\Gamma)$ contains the automorphism group of a finitely bounded
homogeneous structure $\bB$.
Let $r$ be the maximal arity of the relations of $\bB$ and the valued relations in $\Gamma$,
let $v$ be the maximal number of variables that appear in a single conjunct
of the universal sentence $\psi$ that describes the age 
of $\bB$,
and let $m \geq \max(r+1,v,3)$. 
Then there is a polynomial-time reduction from $\VCSP(\Gamma)$ to $\VCSP(\Gamma^*_{\bB,m})$.
\end{theorem}

\begin{proof}
Let $\tau,\tau^*,\sigma$ be the signatures of $\Gamma$, $\Gamma^*_{\bB,m}$, and $\bB$, respectively.  Let $\phi$ be an instance of 
$\VCSP(\Gamma)$ with threshold $u$ and let $V$ be the variables of $\phi$.  
Create a variable $y({\bar x})$ for every 
$\bar x =(x_1, \dots, x_m) \in V^m$. For every summand $R(x_1,\dots,x_k)$ of $\phi$ and
we create a summand $R^*(y(x_1, \dots, x_k, \dots, x_k))$; this makes sense since $m\geq r$. 
For every $\bar{x}, \bar{x}' \in V^m$, $p \in \{1,\dots,m\}$, and 
$i,j \colon \{1,\dots,p\} \to \{1,\dots,m\}$, add the summand 
$C_{i,j}(y(\bar x),y(\bar x'))$
if $(x_{i(1)},\dots, x_{i(p)}) = (x'_{j(1)},\dots, x'_{j(p)})$; we will refer to these as \emph{compatibility constraints}.  
Let $\phi^*$ be the resulting $\tau^*$-expression. 
Clearly, $\phi^*$ can be computed from $\phi$ in polynomial time.

Suppose first that $(\phi,u)$ has a solution;
it will be notationally convenient to view the solution as a function $f$ from the variables of $\phi$ to the elements of $\Gamma$ (rather than a tuple). We claim that the map $f^*$ which maps $y(\bar x)$ to the orbit of $f(\bar x)$ in $\Aut(\bB)$ is a solution for $(\phi^*,u)$. And indeed, 
each of the summands involving a symbol $C_{i,j}$ evaluates to $0$,
and 
$(\phi^*)^{\Gamma^*_{\bB,m}}$ equals 
$\phi^{\Gamma}$.

Now suppose that $(\phi^*,u)$ has a solution $f^*$. 
To construct a solution $f$ to $(\phi,u)$, we first define an equivalence relation $\sim$
on $V$. For $x_1,x_2 \in V$, define $x_1 \sim x_2$
if a (equivalently: every) tuple $t$ in $f^*(y(x_1,x_2,\dots,x_2))$
satisfies $t_1=t_2$.
Clearly, $\sim$ is reflexive and symmetric.
To verify that $\sim$ is transitive, suppose that 
 $x_1 \sim x_2$ and $x_2 \sim x_3$. 
 In the following we use that $m \geq 3$. 
Let $i$ be the identity map on $\{1,2\}$, 
let $j \colon \{1,2\} \to \{2,3\}$ be given by $x \mapsto x+1$, 
and let $j' \colon \{1,2\} \to \{1,3\}$ be given by $j'(1) = 1$ and $j'(2) = 3$. Then $\phi^*$ contains the conjuncts 
\begin{align*}
    & C_{i, i}(y(x_1, x_2, x_2, \dots, x_2), y(x_1,x_2,x_3,\dots,x_3)),\\ & C_{i, j}(y(x_2,x_3,x_3,\dots,x_3),y(x_1,x_2,x_3,\dots,x_3)), \\ 
    & C_{i, j'}(y(x_1,x_3,x_3,\dots,x_3),y(x_1,x_2,x_3,\dots,x_3)). 
\end{align*}
Let $t$ be a tuple from $f^*(y(x_1,x_2,x_3,\dots,x_3))$. 
Then it follows from the conjuncts with the relation symbols $C_{i,i}$ and $C_{i,j}$ that $t_1=t_2$ and $t_2=t_3$, and therefore $t_1=t_3$. Thus we obtain from the conjunct with $C_{i, j'}$ that $x_1 \sim x_3$. 

\paragraph{Claim 0.} For all equivalence classes $[x_1]_\sim,\dots,[x_m]_\sim$, tuple $t \in f^*(y(x_1,\dots,x_m))$, $S \in \sigma$ of arity $k$, and a map $j \colon \{1, \dots, k\} \rightarrow \{1, \dots, m\}$,
whether 
$\bB \models S(t_{j(1)},\dots,t_{j(k)})$ does not depend on the choice of the representatives $x_{1},\dots,x_{m}$. 
It suffices to show this statement if we choose another representative $x'_i$ for $[x_i]_\sim$ for some $i \in \{1, \dots, m\}$, because the general case then follows by induction. 

Suppose that for every $t \in f^*(y(x_1,\dots,x_m))$ we have $\bB \models S(t_{j(1)},\dots,t_{j(k)})$ and we have to show that for every tuple $t' \in f^*(y(x_1,\dots,x_{i-1},x_i',x_{i+1},\dots,x_m))$ 
holds $\bB \models S(t'_{j(1)},\dots,t'_{j(k)})$. 
If $i \notin \{j(1),\dots,j(k)\}$, 
then $\phi^*$ contains $$C_{j,j}(y(x_1,\dots,x_m),y(x_1,\dots,x_{i-1},x_i',x_{i+1},\dots,x_m))$$
and hence 
$\bB \models S(t'_{j(1)},\dots,t'_{j(k)})$. 
Suppose $i \in \{j(1),\dots,j(k)\}$; for the sake of notation we suppose that $i=j(1)$. 
By the definition of $\sim$, and since $x_{j(1)} \sim x'_{j(1)}$, every 
 $t'' \in f^*(y(x_{j(1)},x_{j(1)}',\dots,x_{j(1)}'))$ satisfies $t''_1 = t''_2$. 
Let $\tilde t$ be a tuple from 
$$f^*(y(x_{j(1)},\dots,x_{j(k)},x'_{j(1)},\dots,x'_{j(1)})).$$  
(Here we use that $m \geq r+1$.)
\begin{itemize}
    \item $\bB \models S(\tilde t_{1},\dots,\tilde t_{k})$, because we have a compatibility constraint in $\phi^*$ between variables
$y(x_1,\dots,x_m)$ and $y(x_{j(1)}$,
$\dots,x_{j(k)},x'_{j(1)},\dots,x'_{j(1)})$; 
\item $\tilde t_1 = \tilde t_{k+1}$ because of $x_{j(1)}\sim x_{j(1)}'$  and a compatibility constraint between 
$y(x_{j(1)},\dots,x_{j(k)},x'_{j(1)},\dots,x'_{j(1)})$ and $y(x_{j(1)},x_{j(1)}',\dots,x_{j(1)}')$ 
in $\phi^*$; 
\item hence, $\bB \models S(\tilde t_{k+1},\tilde t_2,\dots,\tilde t_k)$; 
\item $\bB \models S(t'_{j(1)},t'_{j(2)},\dots,t'_{j(k)})$ 
due to a compatibility constraint between $y(x_{j(1)},\dots,x_{j(k)}$,
$x'_{j(1)},\dots,x'_{j(1)})$ and $y(x_1,\dots,x_{i-1},x'_i,x_{i+1},\dots,x_m)$
in $\phi^*$:
namely, consider the map $j' \colon \{1, \dots, k\} \to \{1, \dots, m \}$ that coincides with the identity map except that $j'(1) := k+1$, then $
\phi^*$ contains 
\begin{align*}
    C_{j,j'} \big (& y(x_1,\dots,x_{i-1},x'_{i},x_{i+1},\dots,x_m), y(x_{j(1)},\dots,x_{j(k)},x'_{j(1)},\dots,x'_{j(1)}) \big ).
\end{align*}
\end{itemize}

This concludes the proof of Claim 0. 

\medskip 

Now we can define a structure $\bC$ in the signature $\sigma$
on the equivalence classes of $\sim$. 
If $S \in \sigma$ has arity $k$, $j_1,\dots,j_k \in \{1,\dots,m\}$, and 
$[x_1]_\sim,\dots,[x_{m}]_\sim$ are equivalence classes of $\sim$ such that the tuples $t$ in $f^*(y(x_1,\dots,x_m))$ satisfy $S^{\bB}(t_{j_1},\dots,t_{j_k})$ for some representatives $x_1, \dots x_m$ (equivalently, for all representatives, by Claim 0), then add $([x_{j_1}]_\sim,\dots,[x_{j_k}]_\sim)$ to $S^{\bC}$.
No other tuples are contained in the relations of $\bC$.

\paragraph{Claim 1.} 
If $[x_1]_\sim,\dots,[x_m]_\sim$ are equivalence classes of $\sim$, and 
$t \in f^*(y(x_1,\dots,x_m))$, then $[x_i]_{\sim} \mapsto t_i$,
for $i \in \{1,\dots,m\}$, is an isomorphism
between a substructure of $\bC$ and a substructure of $\bB$ for any choice of representatives $x_1, \dots, x_m$. 
First note that $[x_i]_{\sim} = [x_j]_{\sim}$ if and only if $t_i = t_j$, so the map is well-defined and bijective. 
Let $S \in \sigma$ be of arity $k$ and 
$j \colon \{1, \dots, k\} \to \{1, \dots, m \}$. 
If $\bB \models S(t_{j(1)},\dots,t_{j(k)})$, 
then 
$\bC \models S([x_{j(1)}]_\sim,\dots,[x_{j(k)}]_\sim)$ by the definition of $\bC$. 
Conversely, suppose that 
$\bC \models S([x_{j(1)}]_\sim,\dots[x_{j(k)}]_\sim)$. 
By Claim 0 and the definition of $\bC$, there is $t' \in f^*(y(x_1,\dots,x_m))$ such that $\bB \models S(t'_{j(1)},\dots,t'_{j(k)})$. Since $f^*(y(x_1,\dots,x_m))$ is an orbit of $\Aut(\bB)$, we have $\bB \models S(t_{j(1)},\dots,t_{j(k)})$ as well.

\paragraph{Claim 2.} $\bC$ embeds into $\bB$. 
It suffices to verify that $\bC$ satisfies each conjunct of the universal sentence $\psi$. 
Let $\psi'(x_1,\dots,x_q)$ be such a conjunct, and let $[c_1]_\sim,\dots,[c_q]_\sim$ be elements of $\bC$. Consider the orbit $f^*(y(c_1,\dots,c_q,\dots,c_q))$ of $\Aut(\Gamma)$; this makes sense since $m \geq v$. Let $t \in f^*(y(c_1,\dots,c_q,\dots,c_q))$.
Since $t_1, \dots, t_q$ are elements of $\bB$, the tuple $(t_1,\dots,t_q)$ satisfies $\psi'$.
Claim 1 then implies that $([c_1]_\sim,\dots,[c_q]_\sim)$ satisfies $\psi'$. 

\medskip

Let $e$ be an embedding of $\bC$ to $\bB$. For every $x \in V$, define $f(x)=e([x]_\sim)$. Note that for every summand $R(x_1, \dots, x_k)$ in $\phi$ and $t \in f^*(y(x_1, \dots, x_k, \dots, x_k))$, we have
\begin{align*}
R^*(f^*(y(x_1, \dots, x_k, \dots, x_k)))  = R(t_1, \dots, t_k) = R(e([x_1]_\sim), \dots, e([x_k]_\sim)) = R(f(x_1), \dots, f(x_k)),
\end{align*}
where the middle equality follows from $t_i \mapsto e([x_i]_\sim)$ being a partial isomorphism of $\bB$ by Claim 1 and 2, which by the homogeneity of $\bB$ extends to an automorphism of $\bB$ and therefore also an automorphism of $\Gamma$.
Since $f^*$ is a solution to $(\phi^*,u)$, it follows from the construction of $\phi^*$ that $f$ is a solution to $(\phi,u)$.
\end{proof}

Let $G$ be a permutation group that contains the automorphism group of 
a finitely bounded homogeneous structure $\bB$ 
of maximal arity 
at most $m$. 
A fractional operation $\omega$ over the domain $C$ of  
$\Gamma$ of arity $\ell$ which is canonical with respect to $G$ 
 induces a fractional operation $\omega^*$ on the orbits of $m$-tuples 
 of  $G$, given by
 $$\omega^*(A) := \omega \big ( \{f \in \Can^{(\ell)}_G \mid f^* \in A\} \big ),
 $$
for every subset $A$ of the set of operations of arity $\ell$ on the 
set of orbits of $m$-tuples of $G$ 
(all such subsets are measurable). Note that 
$\{f \in \Can^{(\ell)}_G \mid f^* \in A\}$ is a measurable subset of $\mathscr O_C^{(\ell)}$.
Also note that if $\omega$ is pseudo cyclic, then $\omega^*$ is cyclic.
Statements about the fractional polymorphisms of $\Gamma^*_{\bB,m}$ lift back to statements about the fractional polymorphisms of $\Gamma$ via the following useful lemma. 

\begin{lemma}\label{lem:lift}
Let $\Gamma$ be a valued structure with a finite signature 
such that $\Aut(\Gamma)$ contains the automorphism group $G$ of a finitely bounded homogeneous structure $\bB$
and let $m$ be as in Theorem~\ref{thm:red-fin}. Let 
$\chi \in \fPol(\Gamma^*_{\bB,m})$.
Then there exists $\omega \in \fPol(\Gamma)$ which is canonical with respect to $G$ such that $\omega^* = \chi$.  
\end{lemma}

\begin{proof}
Let $C$ be the domain of $\Gamma$, let $D$ be the domain of $\Gamma^*_{\bB,m}$, and let $\ell$ be the arity of $\chi$.
Suppose that 
$\chi(f) > 0$ for some operation $f$. 
Then there exists a function $g \colon C^{\ell}\to C$ which is canonical with respect to $G$ such that $g^* = f$ by Lemma 4.9 in~\cite{BodMot-Unary} (also see Lemma 10.5.12 in~\cite{Book}). 
For every such $f$, choose $g$ such that $g^*=f$ and  define $\omega(g) := \chi(f)$ 
and $\omega(h) := 0$ for all other $h \in {\mathscr O}_C^{(\ell)}$. Since the domain of $\Gamma^*_{\bB,m}$ is finite, this correctly defines a fractional operation $\omega$ of the same arity $\ell$ as $\chi$. It also improves every valued relation $R$ of $\Gamma$: if $R$ has arity $k$, and $a^1,\dots,a^{\ell} \in C^k$, then 
\begin{align*}
    E_\omega[g \mapsto R(g(a^1,\dots,a^{\ell}))] 
    &= \sum_{g \in {\mathscr O}_C^{(\ell)}} \omega(g) R(g(a^1,\dots,a^{\ell})) \\
    &= \sum_{f \in {\mathscr O}^{(\ell)}_D}  \chi(f) R^*(f(a^1,\dots,a^{\ell})_1,\dots,f(a^1,\dots,a^{\ell})_k,\dots,f(a^1,\dots,a^{\ell})_k) \\
       & \leq  \; \frac{1}{\ell} \sum_{j = 1}^{\ell} R^*(a^j_1,\dots,a^j_k,\dots,a^j_k) \\
    &=  \frac{1}{\ell} \sum_{j = 1}^{\ell} R(a^j_1,\dots,a^j_k).  
    \qedhere
\end{align*}
\end{proof}

We finish this section with a lemma which is useful for proving pseudo cyclicity.

\begin{lemma}\label{lem:pc}
Let $G$ be the automorphism group 
of a homogeneous structure $\bB$ with a relational signature of maximal arity at most $m$.
    If $\omega \in {\mathscr F}_C^{(\ell)}$ 
is canonical with respect to $G$ such that $\omega^*$ (defined on the orbits of $m$-tuples of $G$) is cyclic, then $\omega$ is pseudo cyclic with respect to $G$. 
\end{lemma}

\begin{proof}
We use the fact that 
if $f$ is canonical with respect to $G$ such that $f^*$ (defined on the orbits of $m$-tuples) is cyclic, then
$f$ is pseudo cyclic
(see the proof of Proposition 6.6 in~\cite{BPP-projective-homomorphisms}; also see Lemma 10.1.5
in~\cite{Book}). 
Let $C$ be the domain of $\Gamma$
and let $a^1,\dots,a^\ell,b \in C^m$.
It suffices to show that  
$\omega(S_{a^1,\dots,a^{\ell},b} \cap \PC^{(\ell)}_G) = \omega(S_{a^1,\dots,a^{\ell},b}).$
Indeed, 
\begin{align*}
\omega(S_{a^1,\dots,a^{\ell},b}) &= 
\omega(S_{a^1,\dots,a^{\ell},b} \cap \Can_G^{(\ell)}) && \text{(canonicity of $\omega$)}\\
&= \omega^* \big (\{f^* \mid f \in S_{a^1,\dots,a^{\ell},b} \cap \Can_G^{(\ell)} \} \big )  && \text{(definition of $\omega^*$)}\\
&= \omega^* \big (\{f^* \mid f \in S_{a^1,\dots,a^{\ell},b} \cap \Can_G^{(\ell)}\} \cap \Cyc^{(\ell)}_C \big)  && \text{(by assumption)} \\
&= \omega^* \big (\{f^* \mid f \in S_{a^1,\dots,a^{\ell},b} \cap \Can_G^{(\ell)} \cap \PC^{(\ell)}_G\} \big ) && \text{(fact mentioned above and Remark~\ref{rem:can-act})} \\
&= \omega(S_{a^1,\dots,a^{\ell},b} \cap \Can_G^{(\ell)} \cap \PC^{(\ell)}_G) \\
&= \omega(S_{a^1,\dots,a^{\ell},b}  \cap \PC^{(\ell)}_G). &&  \qedhere 
\end{align*}
\end{proof}

\subsection{Fractional Polymorphims on Finite Domains}
For studying canonical operations, we can use known results about operations on finite domains.

\begin{definition}
Let $\omega$ be a fractional operation of arity $\ell$ on a finite domain $C$. Then the \emph{support of $\omega$} is the set
$$\Sup(\omega) = \{f \in {\mathscr O}^{(\ell)}_C \mid \omega(f) > 0\}.$$
If ${\mathscr F}$ is a set of fractional operations, then 
$$\Sup(\mathscr F) := \bigcup_{\omega \in {\mathscr F}} \Sup(\omega).$$
\end{definition}

Note that a fractional operation $\omega$ on a finite domain is determined by the values $\omega(f)$, $f \in \Sup(\omega)$, in contrast to fractional operations on infinite domains. 
Moreover, a fractional polymorphism $\omega$ of a valued structure with a finite domain is cyclic if and only if all operations in its support are cyclic, in accordance to the definitions from~\cite{KozikOchremiak15}.
An operation $f \colon C^4 \to C$ is called \emph{Siggers} if $f(a,r,e,a) = f(r,a,r,e)$ for all $a,r,e \in C$. 

\begin{lemma}\label{lem:h1-transfer}
Let $\Gamma$ and $\Delta$ be valued structures 
that are fractionally homomorphically equivalent. 
\begin{itemize}
    \item If $\Gamma$ has a cyclic fractional polymorphism, then $\Delta$ has a cyclic fractional polymorphism of the same arity.
    \item Suppose that the domains of $\Gamma$ and $\Delta$ are finite. If the set $\Sup(\fPol(\Gamma))$ contains a cyclic operation, then the set $\Sup(\fPol(\Delta))$ contains a cyclic operation of the same arity. 
\end{itemize}
\end{lemma}
\begin{proof}
Let $C$ be the domain of $\Gamma$ and let $D$ be the domain of $\Delta$. Let $\chi_1$ be a fractional homomorphism from $\Gamma$ to $\Delta$, and let $\chi_2$ be a fractional homomorphism from $\Delta$ to $\Gamma$.
Define $\chi_2'$ as the fractional homomorphism 
from $\Delta^{\ell}$ to $\Gamma^{\ell}$
as follows. 
If $f \colon D \to C$, then $f'$ denotes the map from $D^{\ell}$ to $C^{\ell}$ given by 
$(x_1,\dots,x_{\ell}) \mapsto (f(x_1),\dots,f(x_{\ell}))$. 
For a measurable set $S \subseteq (C^\ell)^{(D^\ell)}$, we define  $\chi_2'(S) := \chi_2(\{f \in C^D \mid f' \in S\})$; note that this defines a fractional operation. Since $\chi_2$ is a fractional homomorphism from $\Delta$ to $\Gamma$, $\chi_2'$ is a fractional homomorphism $\Delta^\ell$ to $\Gamma^\ell$. 

Suppose that $\omega$ is a fractional polymorphism of $\Gamma$ of arity $\ell$. 
Then $\omega' := \chi_1 \circ \omega \circ \chi_2'$ is a fractional homomorphism 
from $\Delta^{\ell}$ to $\Delta$ (see Proposition~\ref{prop:frac-comp}), 
and hence a fractional polymorphism of $\Delta$ (see Remark~\ref{rem:frac-pol-frac-hom}).
Note that if $\omega$ is cyclic, then $\omega'$ is cyclic; this shows that first statement of the lemma. 

Next, suppose that $C$ and $D$ are finite and  that there exists $\omega \in \fPol^{(\ell)}(\Gamma)$  such that $\Sup(\omega)$ contains a cyclic operation $g$ of arity $\ell$.  
Since the domain $C$ of $\Gamma$ is finite, there exists a function $f_1 \colon C \to D$ such that $\chi_1(f_1) > 0$ and a function $f_2 \colon D \to C$ such that $\chi_2(f_2)>0$. 
Note that $f_1 \circ g \circ f_2' \colon D^{\ell}\to D$ is cyclic since $g$ is cyclic, and that $\omega'(f_1 \circ g \circ f_2') > 0$. 
\end{proof}

The following definition is taken from~\cite{KozikOchremiak15}.
\begin{definition}[core]\label{def:core}
A valued structure $\Gamma$ over a finite domain is called a \emph{core} if all operations in  
$\Sup(\fPol(\Gamma))^{(1)}$ are injective. 
\end{definition}

We have been unable to find an explicit reference for the following proposition, but it should be considered to be known; 
we also present a proof as a guide to the literature. 

\begin{proposition}\label{prop:core}
Let $\Gamma$ be a valued structure with a finite domain. Then there exists a core valued structure $\Delta$ over a finite domain  which is fractionally homomorphically equivalent to $\Gamma$.
\end{proposition}

\begin{proof}
Let $C$ be the domain of $\Gamma$. 
If $\Gamma$ itself is a core then there is nothing to be shown, so we may assume that there exists a non-injective $f \in \Sup(\fPol^{(1)}(\Gamma))$. Since $C$ is finite, we have that $
D := f(C) \neq C$; let $\Delta$ be the valued structure with domain $D$ and the same signature as $\Gamma$ whose valued relations are obtained from the corresponding valued relations of $\Gamma$ by restriction to $D$. 
It then follows from Lemma 15 in~\cite{KozikOchremiak15-arxiv} 
in combination with Remark~\ref{rem:frac-hom-equiv-fin}
that $\Gamma$ and $\Delta$ are fractionally homomorphically equivalent. After applying this process finitely many times, we obtain a core valued structure that is fractionally homomorphically equivalent to $\Gamma$. 
\end{proof}

The following lemma is a variation of Proposition~17 from \cite{KozikOchremiak15}, which is phrased there only for valued structures $\Gamma$ that are cores 
and for idempotent cyclic operations.

\begin{lemma}\label{lem:cyclic-supp}
Let $\Gamma$ be a valued structure over a finite domain. Then $\Gamma$ has a cyclic fractional polymorphism if and only if $\Sup(\fPol(\Gamma))$ contains a cyclic operation. 
\end{lemma}
\begin{proof}
The forward implication is trivial. For the reverse implication, let $\Delta$ be a core valued structure over a finite domain that is homomorphically equivalent to $\Gamma$, which exists by Proposition~\ref{prop:core}. By Lemma~\ref{lem:h1-transfer}, $\Sup(\fPol(\Delta))$ contains a cyclic operation. Then $\Sup(\fPol(\Delta))$ contains even an idempotent cyclic operation: If $c \in \Sup(\fPol(\Delta))$ is cyclic, then the operation $c_0 \colon x \mapsto c(x, \dots, x)$ is in $\Sup(\fPol(\Delta))$ as well. Since $\Delta$ is a finite core, $c_0$ is bijective and therefore $c_0^{-1}$ (which is just a finite power of $c_0$) and the idempotent cyclic operation $c_0^{-1} \circ c$ lie in $\Sup(\fPol(\Delta))$.
By  Proposition~17 in~\cite{KozikOchremiak15}, $\Delta$ has a cyclic fractional polymorphism and by Lemma \ref{lem:h1-transfer}, $\Gamma$ also has one.
\end{proof}

The following outstanding result classifies the computational complexity of VCSPs for valued structures over finite domains; it does not appear in this form in the literature, but
we explain
how to derive it from results in~\cite{KozikOchremiak15,KolmogorovKR17,BulatovFVConjecture,ZhukFVConjecture,Zhuk20}. 
In the proof, if $\bC$ is a finite relational structure (understood also as a valued structure), we denote by $\Pol(\bC)$ the set $\Sup(\fPol(\bC))$; this notation is consistent with the literature since the set $\Sup(\fPol(\bC))$ coincides with the set of polymorphisms of $\bC$.

\begin{theorem}\label{thm:fin-vcsp-tract}
Let $\Gamma$ be a valued structure with a finite signature and a finite domain. If $(\{0,1\};\OIT)$ does not have a pp-construction in $\Gamma$, then $\Gamma$ has a fractional cyclic polymorphism and $\VCSP(\Gamma)$ is in P, and it is NP-hard otherwise. 
\end{theorem}
\begin{proof}
If $(\{0,1\};\OIT)$ has a pp-construction in $\Gamma$, then the NP-hardness of $\VCSP(\Gamma)$ follows from 
Corollary~\ref{cor:pp-constr-red}.
So assume that $(\{0,1\};\OIT)$ does not have a pp-construction in $\Gamma$. 

Let $\bC$ be a classical relational structure
on the same domain as $\Gamma$ such that 
$\Pol(\bC)= \Sup(\fPol(\Gamma))$; it exists since $\Sup(\fPol(\Gamma))$ contains projections by Remark \ref{expl:id} and is closed under composition by Proposition~\ref{prop:frac-comp} and Remark~\ref{rem:frac-pol-frac-hom}. 
Note that therefore $\fPol(\Gamma) \subseteq \fPol(\bC)$ and since $\Gamma$ has a finite domain, \cite[Theorem 3.3]{FullaZivny} implies that every relation 
of $\bC$ lies in $\langle \Gamma \rangle$. Since $\Gamma$ does not pp-construct $(\{0,1\};\OIT)$, neither does $\bC$, and in particular, $\bC$ does not pp-construct $(\{0,1\};\OIT)$ in the classical relational setting (see \cite[Definition 3.4, Corollary 3.10]{wonderland}). Combining Theorems 1.4 and 1.8 from \cite{wonderland}, $\Pol(\bC)$ contains a cyclic operation. 

Since $\Sup(\fPol(\Gamma))$ contains a cyclic operation, by Lemma~\ref{lem:cyclic-supp}, 
$\Gamma$ has a cyclic fractional polymorphism.
Then 
Kolmogorov, Rol\'inek, and Krokhin \cite{KolmogorovKR17} prove that in this case $\CSP(\Gamma)$ can be reduced to the CSP of a finite relational structure with a cyclic polymorphism, or, equivalently, with a weak near-unanimity polymorphism \cite[Theorem 6.9.2]{Book}; such CSPs were shown to be in P by Bulatov~\cite{BulatovFVConjecture} and, independently, by Zhuk~\cite{ZhukFVConjecture}. 
\end{proof}

\begin{remark}\label{rem:bin-enc}
The polynomial-time tractability result in Theorem~\ref{thm:fin-vcsp-tract} also holds if the
multiplicities of the summands in the instances of $\VCSP(\Gamma)$ are given in binary, essentially because it builds on linear programming, which can cope with binary representations of coefficients.
\end{remark}

The problem of deciding for a given valued structure $\Gamma$ with finite domain and finite signature whether $\Gamma$ satisfies the condition given in the previous theorem can be solved in exponential time~\cite{Kolmogorov-Meta}. 
We also give a reformulation of the theorem above in terms of the underlying crisp structure of $\Gamma$. This is considered to be known, but it is the first time this fact appears in this form in the literature. 

\begin{corollary}\label{cor:underl-crisp}
Let $\Gamma$ be a valued structure with a finite signature and a finite domain. Let $\bC$ be the underlying crisp structure of $\Gamma$. If $(\{0,1\};\OIT)$ does not have a pp-construction in $\bC$, then $\Pol(\bC)$ contains a cyclic operation
and $\VCSP(\Gamma)$ is in P, and it is NP-hard otherwise. 
\end{corollary}

\begin{proof}
By Proposition~\ref{prop:pp-constr-rel}, $\bC$ pp-constructs $(\{0,1\};\OIT)$ if and only if $\Gamma$ does. Hence, if $\bC$ pp-constructs $(\{0,1\};\OIT)$, then $\VCSP(\Gamma)$ is NP-hard by Theorem~\ref{thm:fin-vcsp-tract}. Otherwise, $\Gamma$ does not pp-construct $(\{0,1\};\OIT)$ and again by Theorem~\ref{thm:fin-vcsp-tract}, $\Gamma$ has a cyclic fractional polymorphism and $\VCSP(\Gamma)$ is in P. In this case, $\bC$ has a cyclic fractional polymorphism by Lemma~\ref{lem:easy} and, by Lemma~\ref{lem:cyclic-supp}, $\Sup(\fPol(\bC)) = \Pol(\bC)$ contains a cyclic operation.
\end{proof}

The importance of Corollary~\ref{cor:underl-crisp} is that it shows that for a valued structure $\Gamma$ on a finite domain, the complexity of $\VCSP(\Gamma)$ is completely captured in the CSPs that can be encoded in $\VCSP(\Gamma)$ via pp-constructions (up to polynomial-time equivalence). 
We now state consequences of Theorem~\ref{thm:fin-vcsp-tract} for certain valued structures with an infinite domain.

\begin{proposition}\label{prop:black-box}
Let $\bB$ be a finitely bounded homogeneous structure 
and let $\Gamma$ be a valued structure with finite relational signature such that
$\Aut(\bB) \subseteq \Aut(\Gamma)$. Let $m$ be as in Theorem~\ref{thm:red-fin}.
Then the following are equivalent. 
\begin{enumerate}
\item $\fPol(\Gamma)$ contains a fractional operation which is canonical and pseudo cyclic with respect to $\Aut(\bB)$;  
\item $\fPol(\Gamma^*_{\bB,m})$ contains a cyclic fractional operation; 
\item $\Sup(\fPol(\Gamma^*_{\bB,m}))$ contains a cyclic operation.
\item $\Sup(\fPol(\Gamma^*_{\bB,m}))$ contains a Siggers operation.  
\end{enumerate}
\end{proposition}
\begin{proof}
First, we prove that (1) implies (2). 
If $\omega$ is a fractional polymorphism of 
$\Gamma$, then
$\omega^*$ is a fractional polymorphism of 
$\Gamma^*_{\bB,m}$: 
the fractional operation $\omega^*$ improves $R^*$ because $\omega$ improves $R$,
and
$\omega^*$ improves $C_{i,j}$ for all $i,j$ because $\omega$
is canonical with respect to $\Aut(\bB)$. 
Finally, if $\omega$ is pseudo cyclic with respect to $\Aut(\bB)$, then $\omega^*$ is cyclic.

The implication from (2) to (1) is a consequence of Lemma~\ref{lem:lift} and Lemma~\ref{lem:pc}. 
The equivalence of (2) and (3) follows from Lemma~\ref{lem:cyclic-supp}. 
The equivalence of (3) and (4) is proved in \cite[Theorem 6.9.2]{Book}; the proof is based on \cite[Theorem 4.1]{Cyclic}.
\end{proof} 

Note that item (4) in the previous proposition can be decided algorithmically for a given valued structure $\Gamma^*_{\bB,m}$ (which has a finite domain
and finite signature) by testing all $4$-ary operations on $\Gamma^*_{\bB,m}$. 

\begin{theorem}\label{thm:tract}
If the conditions from Proposition~\ref{prop:black-box} 
hold, then $\VCSP(\Gamma)$ is in P. 
\end{theorem} 
\begin{proof}
If $\Gamma^*_{\bB,m}$ has a cyclic fractional polymorphism 
of arity $\ell \geq 2$, then the polynomial-time tractability of $\VCSP(\Gamma^*_{\bB,m})$ follows from Theorem~\ref{thm:fin-vcsp-tract}.
For $m$ large enough, we may apply Theorem~\ref{thm:red-fin}
and obtain a polynomial-time reduction from $\VCSP(\Gamma)$ 
to $\VCSP(\Gamma^*_{\bB,m})$, which concludes the proof. 
\end{proof}

\begin{remark}\label{rem:bin-enc-2}
    The polynomial-time tractability result in Theorem~\ref{thm:tract} also holds if the multiplicities of the summands in the input instance of 
    $\VCSP(\Gamma)$ are given in binary; see Remark~\ref{rem:bin-enc}.
\end{remark}

\section{Application: Resilience}
\label{sect:resilience}

We introduce the resilience problem for
conjunctive queries and, more generally,
unions of conjunctive queries.
We generally work with Boolean queries, i.e., queries without free variables. 
A \emph{conjunctive query} is a primitive positive $\tau$-sentence and 
a \emph{union of conjunctive queries} is a (finite) disjunction of conjunctive
queries. Note that every existential positive sentence can be written as a union of conjunctive queries.

Let $\tau$ be a finite relational signature and $\mu$ a conjunctive query  over $\tau$. The input to the \emph{resilience problem for} $\mu$ consists of 
a finite $\tau$-structure $\bA$, called a ($\tau$-)\emph{database}\footnote{To be precise, a finite relational structure is not exactly the same as a database, because the latter may not contain elements that are not contained in any relation. This difference, however, is inessential for the problems studied in this article.}, and the 
task is to compute the number of tuples that have to be removed from relations of $\bA$ so that $\bA$ does \emph{not} satisfy $\mu$. This number is called the \emph{resilience} of $\bA$ (with respect to $\mu$). 
As usual, this can be turned into a decision problem where the input also contains a natural number $u \in {\mathbb N}$ and the question is whether the resilience is at most $u$. 
Clearly, $\bA$ does not satisfy $\mu$ if and only if its resilience is $0$. 

 A natural variation of the problem is that the input database is a \emph{bag database}, meaning that it may contain tuples with \emph{multiplicities}. 
Formally,
 a \emph{multiset relation} on a set $A$ of arity $k$ is a multiset with elements from $A^k$ and a \emph{bag database} $\bA$ over a relational signature $\tau$ consists of a finite domain $A$ and for every $R \in \tau$ of arity $k$, a multiset relation $R^\bA$ of arity $k$. A bag database $\bA$ satisfies a union of conjunctive queries $\mu$ if the relational structure obtained from $\bA$ by forgetting the multiplicities of tuples in its relations satisfies $\mu$.
In this article, we focus on bag databases whose relevance
has already been discussed in the introduction.
The reader might wonder why, unlike in~\cite{LatestResilience}, we restrict to study of resilience problems in bag semantics and do not consider set semantics. This will become apparent from the proof of Proposition~\ref{prop:connection}, see Remark~\ref{rem:set}.

The basic resilience problem defined above can
be generalized  by admitting the decoration of databases with a subsignature $\sigma \subseteq \tau$, in this way declaring all tuples in $R^{\bA}$, $R \in \sigma$, to be \emph{exogenous}. This means that we are not allowed to remove such tuples
from  $\bA$ to make $\mu$ false; the tuples in the other relations are then called \emph{endogenous}. For brevity, we also refer to the relations in $\sigma$ as being exogenous and those in $\tau \setminus \sigma$ as being endogenous. If not specified, then $\sigma=\emptyset$, i.e., all tuples are endogenous. As an alternative, one may also
declare individual tuples as being endogenous
or exogenous. Under bag semantics, however,
this case can be reduced to the one studied
here (see Remark~\ref{rem:exo}).
The resilience problem that we study is summarized
in Figure~\ref{fig:res-problem}. Note that 
for every $\mu$ 
this problem is in NP.
We remark that in Section~\ref{sect:rpqs}, we briefly deviate from our focus on unions of conjunctive queries and, by adapting our results from Section~\ref{sect:fin-dual}, we establish a complexity dichotomy for resilience problems for two-way regular path queries; all the necessary definitions will be given in Section~\ref{sect:rpqs}.

\begin{figure} 
\flushleft 
\begin{framed}
Fixed: a relational signature $\tau$, a subset $\sigma \subseteq \tau$, and a union $\mu$ of conjuctive queries over $\tau$. 

\textbf{Input:} A bag database $\bA$ in signature $\tau$ and $u \in \N$.

$m := $ minimal number of tuples to be removed from the relations in $\{R^{\bA} \mid R \in \tau \setminus \sigma \}$
so that $\bA \not \models \mu$.

\textbf{Output:} Is $m \leq u$?
\end{framed}
 \Description{A relational signature, its subset of exogenous relations and a union of conjunctive queries are fixed. Input is a bag database in the signature and a threshold. Let $m$ be the minimal number of tuples to be removed from the relations of the database that are not exogenous so that the database does not satisfy the union of queries. Output is the answer to ``Is $m$ at most the threshold?''.}
\caption{The resilience problem considered in this article.} 
\label{fig:res-problem} 
\end{figure}

The \emph{canonical database}
of a conjunctive query $\mu$ with relational signature $\tau$ is the $\tau$-structure $\bA$ whose domain are the variables of $\mu$ and where $a \in R^{\bA}$ for
$R \in \tau$ if and only if $\mu$ contains the conjunct $R(a)$. 
Conversely, the \emph{canonical query} of a relational $\tau$-structure
$\bA$ is the conjunctive query whose variable set is the domain $A$ of $\bA$, and which contains for every $R \in \tau$ and $a \in R^{\bA}$ the conjunct $R(a)$.

\begin{remark}
For every conjunctive query $\mu$, the resilience problem for $\mu$ parameterized by the threshold $u$ is fixed-parameter tractable (FPT, which we  refrain from defining here, see~\cite{ParametrizedAlgorithmsBook}). Indeed, there is
a parameterized reduction to $k$-Hitting Set parametrized by the threshold $u$, which is known to be FPT~\cite{ParametrizedAlgorithmsBook}.\footnote{We thank Peter Jonsson and George Osipov for suggesting this remark to us and letting us include it in this article.} The reduction is as follows. Let $\tau$ be the signature of $\mu$, let $k$ be the number of conjuncts of $\mu$, and let $\bA$ be an input database. Construct a $k$-uniform hypergraph $H$ where 
\begin{itemize}

\item
the vertices take the form $(R(a), i)$ with $R \in \tau$, $a \in R^{\bA}$, and $i \geq 1$ bounded by the multiplicity of $a$ in $R^{\bA}$,  and
\item every homomorphism $h$ from the canonical database of $\mu$ to the input database $\bA$ gives rise to
all hyperedges of the form $\{ (R_1(h(x^1)),i_1),\dots,(R_\ell(h(x^\ell)),i_\ell)\}$ where $R_1(x^1),\dots,R_\ell(x^\ell)$ are the atoms in $\mu$ and 
each $i_j$ is bounded by the multiplicity of $h(x^j)$ in $R^\bA$.
\end{itemize}
 Then $\mu$ has resilience at most $u$ with respect to $\bA$ if and only if there is a set of vertices of $H$ of size at most $u$ that intersects every hyperedge. 
Note that since $\mu$ is fixed, the reduction runs in polynomial time in the size of the database. 
\end{remark}

We next explain how to represent resilience problems as VCSPs using appropriately chosen valued structures with oligomorphic automorphism groups.

\begin{expl}\label{expl:fin-dual}
The following query is taken from Meliou, Gatterbauer, Moore, and Suciu~\cite{Meliou2010}; they show how to solve its resilience problem without multiplicities in polynomial time by a reduction to a max-flow problem.  
Let $\mu$ be the query $$\exists x,y,z \big ( R(x,y) \wedge S(y,z) \big ).$$
 Observe that a finite $\tau$-structure
 satisfies $\mu$ if and only if it does not have a homomorphism to the $\tau$-structure $\bB$
with domain $B = \{0,1\}$ and the relations
$R^{\bB} = \{(0,1),(1,1)\}$ and $S^{\bB} = \{(0,0),(0,1)\}$ (see Figure~\ref{fig:fin-dual}). We turn $\bB$ into the valued structure $\Gamma$ with domain $\{0,1\}$ where $R^{\Gamma}(0,1) = R^{\Gamma}(1,1) = 0 = S^{\Gamma}(0,0) = S^{\Gamma}(0,1)$ and $R^{\Gamma}$ and $S^{\Gamma}$ take value 1 otherwise.
Then $\VCSP(\Gamma)$ is precisely the resilience problem for $\mu$ (with multiplicities).
In Example~\ref{expl:fin-dual-rev} we reprove the result from \cite{LatestResilience} 
that even with multiplicities, the problem can be solved in polynomial time.
\end{expl}

\begin{figure}
    \centering
    \includegraphics[]{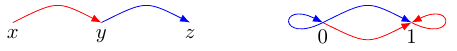}
    \caption{The query $\mu$ from Example~\ref{expl:fin-dual} (on the left) and the corresponding structure $\bB$ (on the right).}
    \label{fig:fin-dual}
     \Description{Arrow depiction of the query from the example and the dual structure.}
\end{figure}

\begin{expl}\label{expl:simple-inf}
Let $\mu$ be the conjunctive query 
$$\exists x,y,z (R(x,y) \wedge S(x,y,z)).$$
This query is \emph{linear} in the sense of Freire, Gatterbauer, Immerman, and Meliou and thus its resilience problem without multiplicities can be solved in polynomial time (Theorem 4.5 in~\cite{Meliou2010}; also see Fact 3.18 in~\cite{Resilience-Arxive}).
In Example~\ref{expl:simple-inf-rev} we
reprove the result from \cite{LatestResilience}
that this problem remains polynomial-time solvable with multiplicities.
\end{expl}

\begin{remark}
Consider the computational problem of \emph{finding} a set of tuples to be removed from the input database $\bA$ so that $\bA \not \models \mu$. We observe that if the resilience problem (with or without multiplicities) for a union $\mu$ of conjunctive queries is in P, then this problem also is in P. To see this, let $u \in \N$ be threshold. 
If $u=0$, then no tuple needs to be found and we are done. Otherwise, 
for every tuple $t$ in a relation $R^{\bA}$, we remove all copies of $t$ from $R^{\bA}$ and test the resulting database with the threshold $u-m$, where $m$ is the multiplicity of $t$. If the modified instance is accepted, then $t$ is a correct tuple to be removed and we may proceed to find a solution of this modified instance, otherwise we go one step back and try to remove a different tuple.
\end{remark}

\subsection{Connectivity}
\label{sect:conn}
We show that when classifying the resilience problem for conjunctive queries, it suffices to consider queries that are connected. This was observed already in~\cite{Resilience} for conjunctive queries in the set semantics setting; we provide a proof in our setting for completeness.
A $\tau$-structure is  \emph{connected} if 
it cannot be written as the disjoint union of two $\tau$-structures with non-empty domains.

\begin{remark}\label{rem:db}
All terminology introduced for $\tau$-structures also applies to conjunctive queries (i.e., primitive positive formulas) with signature $\tau$: by definition, the query has the property if the canonical database has the property. 
\end{remark}

\begin{lemma}\label{lem:con-aux}
Let $\nu_1,\dots,\nu_k$ be conjunctive queries such that $\nu_i$ does not imply $\nu_j$ if $i \neq j$. 
Then the resilience problem for $\nu := \nu_1 \wedge \cdots \wedge \nu_k$ is NP-hard if the resilience problem for one of the $\nu_i$ is NP-hard.
Conversely, if the resilience problem is in P
for each $\nu_i$, then the resilience problem for $\nu$ is in P
as well. The same is true in the setting without multiplicities and/or exogeneous relations.
\end{lemma}

\begin{proof}
We first present a polynomial-time reduction from the resilience problem of $\nu_i$, for some $i \in \{1,\dots,k\}$, to the resilience problem of $\nu$. 
Given an instance $\bA$ of the resilience problem for $\nu_i$, let $m$ be the total multiplicity of tuples in relations of $\bA$. 
Let $\bA'$ be the disjoint union of $\bA$ with $m$ copies of the canonical database of $\nu_j$ for every $j \in \{1,\dots,k\} \setminus \{i\}$. 
Observe that $\bA'$ can be computed in polynomial time in the size of $\bA$ and that the resilience of $\bA$
with respect to $\nu_i$ equals the resilience of 
$\bA'$ with respect to $\nu$. 

Conversely, if the resilience problem is in P for each $\nu_i$,
then also the resilience problem for $\nu$ is in P: 
given an instance $\bA$ of the resilience problem for $\nu$, 
we compute the resilience of $\bA_j$ with respect to $\nu_i$ for every $i \in \{1,\dots,k\}$, and the minimum of all the resulting values.
 
The same proof works in the setting without multiplicities. 
\end{proof}

\begin{corollary}\label{cor:con}
Let $\nu_1,\dots,\nu_k$ be conjunctive queries such that $\nu_i$ does not imply $\nu_j$ if $i \neq j$.
Let $\nu = \nu_1 \wedge \dots \wedge \nu_k$ and suppose that $\nu$ occurs in a union $\mu$ of conjunctive queries. For $i \in \{1, \dots, k\}$, let $\mu_i$ be the union of queries obtained by replacing $\nu$ by $\nu_i$ in $\mu$. Then the resilience problem for $\mu$ is NP-hard if the resilience problem for one of the $\mu_i$ is NP-hard.
Conversely, if the resilience problem is in P
for each $\mu_i$, then the resilience problem for $\mu$ is in P as well.
The same is true in the setting without 
 multiplicities and/or  exogeneous relations.
\end{corollary}

\begin{proof}
Follows immediately from Lemma~\ref{lem:con-aux}.
\end{proof}

By applying Corollary~\ref{cor:con} finitely many times, we obtain that, when classifying the complexity of the resilience problem for unions of conjunctive
queries, we may restrict our attention to unions of connected conjunctive queries.

\subsection{Finite Duals} \label{sect:fin-dual}
If $\mu$ is a union of conjunctive queries with signature $\tau$, then a \emph{dual} of $\mu$ is a $\tau$-structure $\bB$ with the property that a finite $\tau$-structure $\bA$ has a homomorphism to $\bB$ if and only if $\bA$ does not satisfy $\mu$. 
The conjunctive query in Example~\ref{expl:fin-dual}, for instance, even has a \emph{finite} dual. 
There is an elegant characterization of the (unions of) conjunctive queries that have a finite dual.
To state it, we need some basic terminology from database theory. 

\begin{definition}
The \emph{incidence graph} of a relational $\tau$-structure $\bA$ is the bipartite undirected multigraph whose first color class is $A$, and whose second color class consists of expressions of the form $R(b)$ where $R \in \tau$ has arity $k$, $b \in A^k$, 
and $\bA \models R(b)$.
An edge $e_{a,i,R(b)}$ joins $a \in A$ with $R(b)$ 
if $b_i = a$. 
A structure is called \emph{incidence-acyclic} (also known as Berge-acyclic) if its incidence graph is acyclic, i.e., it contains no cycles (if two vertices are linked by two different edges, then they establish a cycle).
A structure is called a \emph{tree} if it is incidence-acyclic and connected in the sense defined in Section~\ref{sect:conn}.
\end{definition}

Note that every relational $\tau$-structure that is homomorphically equivalent to a tree is incidence-acyclic, but the reverse implication does not hold. However, the property of being homomorphically equivalent to a tree is weaker than being a tree. 
What follows is due to Ne\v{s}et\v{r}il and Tardif~\cite{NesetrilTardif};  see also~\cite{Foniok-Thesis,LLT}. 

\begin{theorem}\label{thm:NT}
A conjunctive query $\mu$ has a finite dual if and only if 
the canonical database of 
$\mu$ is homomorphically equivalent to a tree.
A union of conjunctive queries has a finite dual if and only if the canonical database for each of the conjunctive queries is homomorphically equivalent to a tree. 
\end{theorem}

The theorem shows that in particular the query $\mu$ from
Example~\ref{expl:simple-inf} does not have a finite dual, since the query given there is not incidence-acyclic.
To construct valued structures from duals, we introduce the following notation.

\begin{definition}\label{def:valued-dual}
Let $\bB$ be a $\tau$-structure and $\sigma \subseteq \tau$. Define $\Gamma(\bB, \sigma)$ to be the valued $\tau$-structure on the same domain as $\bB$ such that
\begin{itemize}
    \item for each
    $R \in \tau \setminus \sigma$, $R^{\Gamma(\bB, \sigma)}(a) := 0$ if $a \in R^{\bB}$ and $R^{\Gamma(\bB, \sigma)}(a) := 1$ otherwise, and
    \item for each $R \in \sigma$, $R^{\Gamma(\bB, \sigma)}(a) := 0$ if $a \in R^{\bB}$ and $R^{\Gamma(\bB, \sigma)}(a) := \infty$ otherwise.
\end{itemize}
\end{definition}
Note that $\Aut(\bB) = \Aut(\Gamma(\bB, \sigma))$ for any $\tau$-structure $\bB$ and any $\sigma$. In the following result we use the correspondence between resilience problems for incidence-acyclic conjunctive queries and 
 valued CSPs introduced in the present article. 
 The result then follows from the P versus NP-complete dichotomy theorem
for valued CSPs over finite domains
stated in Theorem~\ref{thm:fin-vcsp-tract}. 

\begin{theorem}\label{thm:fin-dual} 
Let $\mu$ be
a union of incidence-acyclic conjunctive queries
with relational signature $\tau$ and let $\sigma \subseteq \tau$. Then the resilience problem for $\mu$ with exogenous relations from $\sigma$ is in P or NP-complete. Moreover, it is decidable whether the resilience problem for given $\mu$ and $\sigma$ 
is in P. 
If $\mu$ is
a union of queries each of which is
homomorphically equivalent to a tree
and $\bB$ is a finite dual of $\mu$ (which exists by Theorem~\ref{thm:NT}), then $\VCSP(\Gamma(\bB,\sigma))$ is polynomial-time equivalent to the resilience problem for $\mu$ with exogenous relations from $\sigma$.
\end{theorem}
\begin{proof}
By virtue of Corollary~\ref{cor:con}, we may assume for the P versus NP-complete dichotomy that each of the
conjunctive queries in $\mu$ is connected and thus a tree. The same
is true also for the polynomial-time equivalence to a VCSP  since replacing a conjunctive query in a union with a homomorphically equivalent one does not affect the complexity of resilience. 
Define $\Gamma := \Gamma(\bB, \sigma)$.
We show that $\VCSP(\Gamma)$ is polynomial-time equivalent to the resilience problem for $\mu$ with exogenous relations from $\sigma$. 

Given a finite-domain bag database $\bA$ with signature $\tau$ and
with exogenous tuples from relations in $\sigma$,
let $\phi$ be the $\tau$-expression which contains for every $R\in \tau$ and for every tuple $a \in R^{\bA}$ the summand $R(a)$ with the same number of occurrences as is the multiplicity of $a$ in $R^{\bA}$.
Conversely, for every $\tau$-expression $\phi$ 
we can create a bag database $\bA$ with signature $\tau$ and exogenous relations from $\sigma$. The domain of $\bA$ is the set of variables of $\phi$ and for every $R \in \tau$ and 
$a \in R^{\bA}$ with multiplicity equal to the number of occurrences of the summand $R(a)$ in $\phi$. 
In both situations, the resilience of $\bA$ with respect to $\mu$ equals the cost of $\phi$ with respect to $\Gamma$.
This shows the final statement of the theorem. The first statement now follows from
Theorem~\ref{thm:fin-vcsp-tract}. 

Concerning the decidability of the tractability condition, first note that the finite dual of $\mu$, and hence also $\Gamma$, can be effectively computed from $\mu$ (e.g., the construction of the dual in~\cite{NesetrilTardif} is effective).
The existence of a fractional cyclic polymorphism for a given valued structure $\Gamma$ with finite domain and finite signature can be decided (in exponential time in the size of $\Gamma$; see~\cite{Kolmogorov-Meta}).
\end{proof}

Theorem~\ref{thm:fin-dual} can be combined with the tractability results for VCSPs from Section~\ref{sect:tract} that use fractional polymorphisms to obtain tractability results for concrete resilience problems. To illustrate fractional polymorphisms and how to find them, we revisit a known tractable resilience problem from~\cite{Meliou2010,Resilience-Arxive,Resilience,NewResilience}
and show that the corresponding valued structure has a fractional canonical pseudo cyclic polymorphism. 

\begin{expl}\label{expl:fin-dual-rev}
We  revisit Example~\ref{expl:fin-dual}. Consider again
 the conjunctive query
$$ \exists x,y,z (R(x,y) \wedge S(y,z)).$$
There is a finite dual $\bB$ of $\mu$ with domain $\{0,1\}$ which is finitely bounded homogeneous, as described in Example~\ref{expl:fin-dual}. That example also describes
a valued structure $\Gamma$ which is actually  $\Gamma(\bB,\emptyset)$.
Let $\omega$ be the fractional cyclic operation given by $\omega(\min) = \omega(\max) = \frac{1}{2}$.
The fractional operation $\omega$ improves both valued relations $R$ and $S$ (they are \emph{submodular}; see, e.g.,~\cite{KrokhinZivny17}) and hence is a cyclic fractional polymorphism of $\Gamma$.
Combining 
Theorems~\ref{thm:fin-vcsp-tract} and \ref{thm:fin-dual}, Example~\ref{expl:fin-dual-rev} reproves the results from \cite{Resilience} (without multiplicities) and \cite{LatestResilience} (with multiplicities) that the resilience problem for this query is in P.
\end{expl}
In the following example we generalize the approach from Example~\ref{expl:fin-dual-rev}; tractability for this class of resilience problems in set semantics was proved in~\cite{Resilience} and in bag semantics in~\cite{LatestResilience}.

\begin{expl} \label{expl:path}
Consider for $n \geq 1$ the conjunctive query 
 \[\mu := \exists x_0,\dots,x_n (R_1(x_0,x_1) \wedge \cdots \wedge R_n(x_{n-1},x_n)).\]
 Let $\tau=\{R_1, \dots, R_n\}$ be the signature of $\mu$.
 We describe a finite dual $\bB$ of $\mu$. The domain $B$ consists of $2^{n-1}$ elements $v_S$, which we index by a subset $S$ of $\{1, \dots, n-1\}$. 
Then the dual satisfies $R_i(v_S,v_T)$ for $S,T \subseteq \{1, \dots, n-1\}$ and $i \in \{1,\dots,n\}$ if and only if 
\begin{itemize}
    \item $i < n$ and $i \notin T$, or
    \item $i > 1$ and $(i-1) \in S$. 
\end{itemize}

To see that $\bB$ is indeed a dual of $\mu$, let $\bA$ be a finite relational $\tau$-structure that does not satisfy $\mu$, and let $a \in A$. Let $S_a$ be the set of all elements
$i \in \{1, \dots, n-1\}$ such that there exist elements $a_{i+1},\dots,a_n \in A$ such that 
$$\bA \models R_{i+1}(a,a_{i+1}) \wedge R_{i+2}(a_{i+1},a_{i+2}) \wedge \cdots \wedge R_n(a_{n-1},a_n).$$

We claim that $a \mapsto v_{S_a}$, $a \in A$, defines a homomorphism $h$ from $\bA$ to $\bB$.

Suppose that $a,b \in A$ are such that $R_j(h(a),h(b))$ does not hold in $\bB$ for some $j \in \{1, \dots, n\}$. 
This implies that
\begin{itemize}
    \item $j = n$ and $(n-1) \notin S_a$, or
    \item $j = 1$ and $1 \in S_b$, or
    \item $j\in \{2, \dots, n-1\}$, $(j-1) \notin S_a$, and $j \in S_b$. 
\end{itemize}
In all three cases,
the definition of $S_a$ and $S_b$ and $\bA \not\models \mu$ imply that $R_j(a,b)$ does not hold in $\bA$.
This shows that $h$ is a homomorphism.

Now suppose that $\bA$ has a homomorphism to $\bB$. We have to show that $\bA$ does not satisfy $\mu$. It suffices to show that $\bB$ does not satisfy $\mu$. Suppose for contradiction that there are elements $v_{S_0},\dots,v_{S_n}$ such that $\bB$ satisfies
$R_1(v_{S_0},v_{S_1}) \wedge \cdots \wedge R_n(v_{S_{n-1}},v_{S_n})$. Then $R_n(v_{S_{n-1}},v_{S_n})$ implies that  
$(n-1) \in S_{n-1}$. 
By induction, we obtain from $R_i(v_{S_{i-1}},v_{S_i})$ that $(i-1) \in S_{i-1}$ for every $i \in \{2, \dots, n\}$. In particular, $1 \in S_1$.
However, note that $R_1(v_{S_0},v_{S_1})$ by definition implies that $1 \notin S_1$.
We reached a contradiction, hence, $\bB$ does not satisfy $\mu$. 

Let $f$ be the symmetric binary operation that maps $(v_S,v_T)$ to $v_{S \cup T}$, and let $g$ be the symmetric binary operation that maps
$(v_S,v_T)$ to $v_{S \cap T}$. 
Define the binary symmetric fractional operation $\omega$ by setting 
$\omega(f)=\omega(g)=\frac{1}{2}$.
We claim that $\omega$ is a fractional polymorphism of $\Gamma(\bB, \emptyset)$. 

Let $i \in \{1,\dots,n\}$. 
We have to show that for all $v_S,v_T,v_P,v_Q \in B$, it holds in $\Gamma$ that
\begin{align} R_i(v_S,v_T)+R_i(v_P,v_Q) & \geq R_i(f(v_S,v_P),f(v_T,v_Q)) + R_i(g(v_S,v_P),g(v_T,v_Q)) \label{eq:path}  \\
& = R_i(v_{S \cup P},v_{T \cup Q}) + R_i(v_{S \cap P},v_{T \cap Q}) \nonumber. 
\end{align}
First suppose that $1<i<n$.
If the left-hand side in \eqref{eq:path} is equal to $2$, then the disequality holds trivially.
Whenever the left-hand side in \eqref{eq:path} is equal to $1$, we must have $(i-1) \in S$, $(i-1) \in P$, $i \notin T$, or $i \notin Q$. Hence $(i-1) \in S \cup P$  or $i \notin T \cap Q$. In the first case, we have $R_i(v_{S \cup P},v_{T \cup Q}) = 0$, and in the second case we have $R_i(v_{S \cap P},v_{T \cap Q}) = 0$, and in both cases \eqref{eq:path} holds.
Finally, if the left-hand side in \eqref{eq:path} is $0$, then $R_i(v_S, v_T)=R_i(v_P, v_Q)=0$. Thus we have $(i-1) \in S$ or $i \notin T$ and at the same time $(i-1) \in P$ or $i \notin Q$. Therefore, we have $(i-1) \in S \cap P$, $i \notin T \cup Q$, or both $(i-1) \in S \cup P$ and $i \notin T \cap Q$. It follows that $R_i(v_{S \cup P},v_{T \cup Q})=R_i(v_{S \cap P},v_{T \cap Q})=0$.
The case that $ i = 1$ and the case that $i = n$ can be treated similarly.
Therefore, $\omega$ is a cyclic fractional polymorphism of $\Gamma(\bB, \emptyset)$, which by Theorem~\ref{thm:fin-dual} and Theorem~\ref{thm:fin-vcsp-tract} implies that the resilience problem for $\mu$ is in P.
\end{expl}

\subsection{Regular Path Queries}
\label{sect:rpqs}

We introduce two-way regular path queries
and show that they have finite duals. A \emph{two-way regular path query (2RPQ)} over a finite relational signature $\tau$ is a
regular expression over the alphabet $\Sigma_{\tau}$ that consists of
the
symbols $R$ and $R^-$ for every binary relation symbol
$R \in \tau$.  We assume that regular expressions are built up
from alphabet symbols, $\emptyset$, $\varepsilon$, $+$ (union), $;$
(concatenation), and $\cdot^*$ (Kleene star). As usual, we may omit
the concatenation symbol for brevity.

Every 2RPQ $\mu$ defines a regular language $L(\mu)$ over the alphabet~$\Sigma_{\tau}$. It also defines a database query of arity~2, as
follows.  A \emph{path} in a $\tau$-database
~$\mathfrak{A}$ is a sequence
$a_1P_1a_2P_2 \cdots P_{n-1}a_n$ with $a_1,\dots,a_n \in A$ and
$P_1,\dots,P_{n-1} \in \Sigma_{\tau}$ such that for
$1 \leq i < n$, one of the following holds: (i)~$P_i \in \tau$ and
$\mathfrak{A} \models P_i(a_i,a_{i+1})$ or~(ii)~$P_i=P^-$ with $P \in \tau$ and $\mathfrak{A} \models P(a_{i+1},a_i)$.
Note that we do not require $a_1,\dots,a_n$ to be distinct.
We call $a_1$ the \emph{source} of the path, $a_n$ the \emph{target},
and $P_1 \cdots P_{n-1}$ the \emph{label}.
A pair
$(a,b) \in A\times A$ is an \emph{answer} to a 2RPQ
$\mu$ if there is a path in $\mathfrak{A}$ that has source $a$, target $b$, and
some label $w \in L(\mu)$. 
We use $\mu(\mathfrak{A})$ to denote the set of all
answers to the 2RPQ $\mu$ on the database~$\mathfrak{A}$.

We consider a Boolean version of the resilience problem for 2RPQs.
For a 2RPQ $\mu$ over a signature $\tau$ and a set $\sigma \subseteq \tau$ of
exogenous relations, this problem is as follows: given a finite (bag) database $\bA$ and a $k \geq 1$, decide whether it suffices to remove at most $k$
tuples from $\bA$ to achieve that $\mu(\bA)=\emptyset$. It is folklore that given a
2RPQ $\mu$ and a database $\bA$, one can decide in polynomial time whether $\mu(\bA)=\emptyset$, e.g.\ by forming the product of $\bA$ and a finite two-way automaton
$A$ with $L(A)=L(\mu)$ and then checking whether some element with a final state in the
second component is reachable from any element with the initial state in the second
component. Consequently, for any 2RPQ $\mu$ with exogenous relations from any set $\sigma$ the resilience
problem is in NP by
a straightforward guess-and-check algorithm.

Duals for 2RPQs are defined in exactly the same way as for conjunctive queries, that is, if $\mu$ is a 2RPQ with signature~$\tau$, then a
    \emph{dual} of $\mu$ is a $\tau$-structure $\bB$ such
     that any finite $\tau$-structure $\bA$ has a homomorphism to $\bB$ if
     and only if $\bA \not \models \mu$. We now argue that every 2RPQ has a finite dual.
Note that the construction from~\cite{CalvaneseDegiocomoLenzeriniVardi}
is related, but does not produce duals in our
sense.
\begin{lemma}
  \label{lem:finitedual-2RPQ}
  Every  2RPQ has a finite dual.
\end{lemma}
\begin{proof}
Theorem 4.6 in \cite{OBDA} states that every simple Boolean monadic Datalog program has a finite dual.
It thus suffices to show that every 2RPQ can
 be expressed as such a program.

A Boolean monadic Datalog (MDLog)
program over a finite relational signature $\tau$ is a finite collection of rules of the forms
$$
\begin{array}{rcl}
R_1(\bar{x}_1) \wedge \cdots \wedge R_n(\bar{x}_n)  &\rightarrow& P(x)\\[1mm]
 \textsf{true}(x) &\rightarrow& P(x)\\[1mm]
R_1(\bar{x}_1) \wedge \cdots \wedge
R_n(\bar{x}_n)  &\rightarrow& \textsf{goal}()
\end{array}
$$
where each $R_i$ is a relation symbol from $\tau$ or a unary relation symbol that is not in $\tau$,  $P$ is a unary relation symbol that is not in $\tau$, and \textsf{goal} is a designated nullary relation symbol that is also not in $\tau$. Informally,
an MDLog program $\Pi$ is true on a database 
$\mathfrak{A}$ if, starting from $\mathfrak{A}$,
we can derive $\textsf{goal}()$ by applying the rules in $\Pi$. For details, see \cite{DBLP:books/aw/AbiteboulHV95}.

An MDLog program $\Pi$ is \emph{simple} if every rule body (which is a conjunctive query) is connected and
contains at most one atom that uses a relation symbol from $\tau$, and in that atom there are no
repeated occurrences of variables. 
  
Now let $\mu$ be a 2RPQ over a relational signature $\tau$.  For every subexpression $\nu$ of
$\mu$, we define an associated simple MDLog program $\Pi_\nu$ inductively as
follows:
  \begin{itemize}
   \item if $\nu=R$, then $\Pi_\nu$ contains the rule $S_\nu(x) \wedge
     R(x,y) \rightarrow E_\nu(y)$;

   \item if $\nu=R^-$, then $\Pi_\nu$ contains the rule $S_\nu(y) \wedge
     R(x,y) \rightarrow E_\nu(x)$;

  \item if $\nu=\nu_1;\nu_2$, then $\Pi_\nu$ contains the rules
    $S_\nu(x) \rightarrow S_{\nu_1}(x)$,
    $E_{\nu_1}(x) \rightarrow S_{\nu_2}(x)$, and $E_{\nu_2}(x) \rightarrow
    E_\nu(x)$;

  \item if $\nu=\nu_1 + \nu_2$, then $\Pi_\nu$ contains the rules
    $S_\nu(x) \rightarrow S_{\nu_i}(x)$ and
    $E_{\nu_i}(x) \rightarrow E_\nu(x)$ for all $i \in \{1,2\}$;

  \item if $\nu=\nu_1^*$, then $\pi_\nu$ contains the rules
    $S_\nu(x) \rightarrow E_\nu(x)$, $S_\nu(x) \rightarrow S_{\nu_1}(x)$,
    and $E_{\nu_1}(x) \rightarrow S_\nu(x)$.
    
  \end{itemize}
  We obtain from $\Pi_\mu$ a simple MDLog program that is equivalent to $\mu$ by adding the rules $\textsf{true}(x) \rightarrow S_\mu(x)$
  and $E_\mu(x) \rightarrow \textsf{goal}()$.
\end{proof}

If $\mu$ is
a 2RPQ and $\bB$ is the finite dual of $\mu$, then $\VCSP(\Gamma(\bB,\sigma))$ is polynomial-time equivalent to the Boolean resilience problem for $\mu$ with exogenous relations from $\sigma$. This can be proved in analogy
with the proof of Theorem~\ref{thm:fin-dual}. 
The following result thus follows from the P versus NP-complete dichotomy theorem
for valued VCSPs over finite domains
stated in Theorem~\ref{thm:fin-vcsp-tract}. 
\begin{theorem}
\label{thm:rpqbool}
Let $\mu$ be a 2RPQ over a relational signature $\tau$ and let $\sigma \subseteq \tau$. Then the resilience problem for $\mu$ with exogenous relations from $\sigma$ is in P or NP-complete. Moreover, it is decidable whether the resilience problem for a given 2RPQ $\mu$ and set of exogenous relations $\sigma$
is in P.
\end{theorem}

\begin{example}
    Consider the 2RPQ $\mu$ over the signature $\{R\}$
    given by $R(R R^{-})^* RR$. It is well-known (see Propositions 1.6 and 1.13 in~\cite{HNBook}) that the relational $\{R\}$-structure $\bB$  with domain 
    $\{0,1,2\}$ and the relation
    $R^{\bB} := \{(0,1),(1,2)\}$ is a dual of $\mu$. Let $\Gamma := \Gamma(\bB; \emptyset)$. Note that the crisp unary relation $\{0,1\}$ is pp-expressible in $\Gamma$ by
    $$\phi(x) := \Opt (\inf_{y} R(x,y)).$$
    Let $R'(x,y) := R(x,y) + \phi(x) + \phi(y)$ and note that $R' \in \langle \Gamma \rangle$. It is easy to see that the valued structure $(\{0,1,2\}; R')$ is fractionally homomorphically equivalent to the valued structure $\Gamma_{<}$ from Example~\ref{expl:vs-mc}. In other words, $\Gamma$ pp-constructs $\Gamma_{<}$. Since $\VCSP(\Gamma_{<})$ is NP-complete (Example~\ref{expl:mc-hard}), $\VCSP(\Gamma)$ is NP-complete by Lemma~\ref{cor:pp-constr-red} (every finite-domain VCSP is in NP). Therefore, the resilience problem for the 2RPQ $\mu$ is NP-complete as well.
\end{example}

\begin{example} \label{expl:rpq-tract}
    Consider the 2RPQ $\mu$ over the signature $\tau=\{R,S,T\}$
    given by $RS^*T$.
    By Lemma~\ref{lem:finitedual-2RPQ},
    it has a finite dual. Consider the relational structure $\bB$ with  domain $\{0,1\}$ and relations
    \begin{itemize}
        \item $R^{\bB} = \{(0,0),(1,0)\}$, 
       \item $S^{\bB} = \{(0,0),(1,0),(1,1)\}$, 
       \item $T^{\bB} = \{(1,0),(1,1)\}$
    \end{itemize}
    (see Figure~\ref{fig:rpq}).
    We argue that $\bB$ is a dual of $\mu$. Clearly, $\bB \not\models \mu$. Let $\bA$ be a finite $\tau$-structure
    with the domain $A$ such that $\bA \not \models \mu$. Let $\bA'$ be the directed graph with the vertex set $A$ and the edge relation $E$ where $E = R^\bA \cup S^{\bA} \cup T^{\bA}$. We define $h \colon A \to \{0,1\}$ by setting $h(a)=0$ if there exists $b,c \in A$ such that $(b,c) \in R^{\bA}$ and there is a directed path from $c$ to $a$ in $\bA'$,
    and $h(a)=1$ otherwise. It is straightforward to verify that $h$ is a homomorphism from $\bA$ to $\bB$.

Let $\omega$ be a binary fractional operation on $\{0,1\}$ defined by $\omega(\min)=\omega(\max) = 1/2$, where $\min$ and $\max$ are the binary minimum and maximum operations on $\{0,1\}$. It is easy to check that $\omega \in \fPol(\Gamma(\bB, \emptyset))$.
    Therefore, $\omega$ is a cyclic fractional polymorphism of $\Gamma(\bB, \emptyset)$, which by Theorem~\ref{thm:fin-vcsp-tract} implies that $\VCSP(\Gamma(\bB, \emptyset))$ and therefore
    the resilience problem for $\mu$ is in P.
\end{example}

\begin{figure}
    \centering
    \includegraphics[]{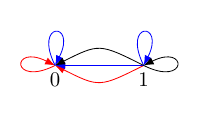}
    \caption{The dual structure $\bB$ from Example~\ref{expl:rpq-tract}.}
    \label{fig:rpq}
    \Description{Arrow depiction of the dual structure.}
\end{figure}

\subsection{Infinite Duals}
\label{sect:inf-dual}
Conjunctive queries might not have a finite dual (see Example~\ref{expl:simple-inf}), but unions of connected conjunctive queries always have a countably infinite dual. Cherlin, Shelah and Shi~\cite{CherlinShelahShi} showed that 
in this case we may even find a dual with an oligomorphic automorphism group (see Theorem~\ref{thm:css} below). This is the key insight to phrase resilience problems as VCSPs for valued structures with oligomorphic automorphism groups. The not necessarily connected case again 
reduces to the connected case by Corollary~\ref{cor:con}. 

In Theorem~\ref{thm:css} below we state a variant of a theorem of Cherlin, Shelah, and Shi~\cite{CherlinShelahShi} (also see~\cite{Hubicka-Nesetril-ForbiddenHomomorphisms,Hubicka-Nesetril,Book}). 
If $\bB$ is a structure, we write $\bB_{\text{pp}(m)}$ for the expansion of $\bB$ by all relations that can be defined with a connected primitive positive formula (see Remark~\ref{rem:db})
with at most $m$ variables, at least one free variable, and without equality. 
For a union of conjunctive queries $\mu$ over the signature $\tau$, we write $|\mu|$ for the maximum of the number of variables of each conjunctive query in $\mu$, the maximal arity of $\tau$, and 2. 

\begin{theorem}[{see, e.g.,~\cite[Proposition 4.3.8]{Book}}]
\label{thm:css}
For every union $\mu$ of connected conjunctive queries over a finite relational signature $\tau$ there exists a $\tau$-structure $\bB_{\mu}$ such that the following statements hold:
\begin{enumerate}
\item $(\bB_\mu)_{\text{pp}(|\mu|)}$ is homogeneous.
\item $\Age(\bB_{\text{pp}(|\mu|)})$ is the class of all substructures of structures of the form $\bA_{\text{pp}(|\mu|)}$ for a finite structure $\bA$ that satisfies $\neg \mu$.
\item A countable
$\tau$-structure $\bA$ satisfies $\neg \mu$ if and only if it embeds into $\bB_{\mu}$. 
\item $\bB_\mu$ is finitely bounded.
\item $\Aut(\bB_{\mu})$ is oligomorphic.  
\item $(\bB_\mu)_{\text{pp}(|\mu|)}$ is finitely bounded.
\end{enumerate}
\end{theorem}
\begin{proof}
The construction of a structure $\bB_{\mu}$ 
with the given properties 
follows from a proof of Hubi\v{c}ka and Ne\v{s}et\v{r}il~\cite{Hubicka-Nesetril,Hubicka-Nesetril-ForbiddenHomomorphisms} of the theorem of Cherlin, Shelah, and Shi~\cite{CherlinShelahShi}, and can be found in~\cite[Theorem 4.3.8]{Book}. 
Properties (1), (2) and property (3) restricted to finite structures $\bA$ are explicitly stated in~\cite[Theorem 4.3.8]{Book}. 
Property (3) restricted to finite structures clearly implies property (4). Property (5) holds because homogeneous structures with a finite relational signature have an oligomorphic automorphism group. 
Property~(3) for countable structures now follows from~\cite[Lemma 4.1.7]{Book}.

Since we are not aware of a reference for (6) in the literature, we present a proof here. Let $\sigma$ be the signature of $(\bB_\mu)_{\text{pp}(|\mu|)}$.
We claim that the following universal $\sigma$-sentence $\psi$ describes the structures in the age of
$(\bB_\mu)_{\text{pp}(|\mu|)}$. 
If $\phi$ is a $\sigma$-sentence, then $\phi'$ denotes the $\tau$-sentence obtained from $\phi$ by replacing every occurrence of $R(\bar x)$, for $R \in \sigma \setminus \tau$, 
by the primitive positive $\tau$-formula $\eta(\bar x)$ for which $R$ was introduced in $(\bB_\mu)_{\text{pp}(|\mu|)}$.
Then $\psi$
is a conjunction
of all $\sigma$-sentences $\neg \phi$ such that $\phi$ is primitive positive, 
$\phi'$ has at most $|\mu|$ variables, and $\phi'$ implies $\mu$. 
Clearly, there are finitely many conjuncts of this form.

Suppose that $\bA \in \Age(\bB_\mu)_{\text{pp}(|\mu|)}$.
Then $\bA$ satisfies each conjunct $\neg \phi$ of $\psi$, because otherwise $\bB_{\mu}$ satisfies $\phi'$, and thus satisfies $\mu$, contrary to our assumptions.

The interesting direction is that if a finite $\sigma$-structure $\bA$ satisfies $\psi$, then $\bA$ embeds into $(\bB_\mu)_{\text{pp}(|\mu|)}$. 
Let $\phi$ be the canonical query of $\bA$.
Let $\bA'$ be the canonical database of the $\tau$-formula $\phi'$. 
Suppose for contradiction that $\bA' \models \mu$. 
Let $\theta$ be a minimal subformula of $\phi'$ such that
the canonical database of $\theta$ models $\mu$. 
Then $\theta$ has at most $|\mu|$ variables and implies $\mu$, and hence $\neg \theta$ is a conjunct of of $\psi$ which is not satisfied by $\bA$, a contradiction to our assumptions.
Therefore, $\bA' \models \neg \mu$ and by Property (2), we have that $\bA'_{\text{pp}(|\mu|)}$
has an embedding $f$ into $(\bB_\mu)_{\text{pp}(|\mu|)}$.

We claim that the restriction of $f$ to the elements of $\bA$ is an embedding of $\bA$ into $(\bB_\mu)_{\text{pp}(|\mu|)}$. Clearly, if $\bA \models R(\bar x)$ for some relation $R$ that has been introduced for a primitive positive formula $\eta$, then $\bA'$ satisfies $\eta(\bar x)$, 
and hence $\bB_\mu \models \eta(f(\bar x))$, which in turn implies that
$(\bB_\mu)_{\text{pp}(|\mu|)} \models R(f(\bar x))$ as desired. 
Conversely, if $(\bB_\mu)_{\text{pp}(|\mu|)} \models R(f(\bar x))$, then
$\bA'_{\text{pp}(|\mu|)} \models R(\bar x)$, and hence $\bA' \models \eta(\bar x)$. 
This in turn implies that $\bA \models R(\bar x)$.
Since the restriction of $f$ and its inverse preserve the relations from $\tau$ trivially, we conclude that $\bA$ embeds into $(\bB_\mu)_{\text{pp}(|\mu|)}$.
\end{proof}

By Properties (1) and (6) of Theorem~\ref{thm:css}, $\bB_\mu$ is always a reduct of a finitely bounded homogeneous structure. For short, we write $\Gamma_\mu$ for $\Gamma(\bB_\mu, \emptyset)$ and $\Gamma_{\mu,\sigma}$ for $\Gamma(\bB_\mu, \sigma)$, see Definition~\ref{def:valued-dual}.
For some queries $\mu$, the structure $\bB_\mu$ can be replaced by a simpler structure $\bH_\mu$. This will be 
convenient for some examples that we consider later, because the structure $\bH_\mu$ is finitely bounded and homogeneous itself.
To define the respective class of queries, we need the following definition. The \emph{Gaifman graph} of a 
relational structure $\bA$ 
is the undirected graph with vertex set $A$ where $a,b \in A$ are adjacent if and only if $a \neq b$ and there exists a tuple in a relation of $\bA$ that contains both $a$ and $b$. The Gaifman graph of a conjunctive query is the Gaifman graph of the canonical database of that query.

\begin{theorem}\label{thm:freeAP}
    For every union $\mu$ of 
    conjunctive queries over a finite relational signature $\tau$ such that the Gaifman graph of each of the conjunctive queries in $\mu$  is complete, there exists a countable $\tau$-structure $\bH_\mu$ such that the following statements hold:
    \begin{enumerate}
        \item $\bH_\mu$ is homogeneous.
        \item $\Age(\bH_\mu)$ is the class of all finite structures $\bA$ that satisfy $\neg \mu$.
    \end{enumerate}
    Moreover, $\bH_\mu$ is finitely bounded, $\Aut(\bH_\mu)$ is oligomorphic, and a countable $\tau$-structure satisfies $\neg \mu$ if and only if it embeds into $\bH_\mu$.
\end{theorem}
\begin{proof}
    Let $\bA_1$ and $\bA_2$ be finite $\tau$-structures that satisfy $\neg \mu$ such that the substructure induced by $A_1 \cap A_2$ in $\bA_1$ and $\bA_2$ is the same. Since the Gaifman graph of each of the conjunctive queries in $\mu$ is complete, the union of the structures $\bA_1$ and $\bA_2$ satisfies $\neg \mu$ as well, because there are no relations between elements from $A_1 \setminus A_2$ and $A_2 \setminus A_1$ in the union. By Fra\"iss\'e's Theorem (see, e.g., \cite{Hodges})
    there is a countable homogeneous $\tau$-structure $\bH_\mu$ such that $\Age(\bH_\mu)$ is the class of all finite structures that satisfy $\neg \mu$; this shows that $\bH_\mu$ is finitely bounded. 
    Homogeneous structures with finite relational signature clearly have an oligomorphic automorphism group. 
    For the final statement, see~\cite[Lemma 4.1.7]{Book}.
\end{proof}

Note that $\bH_\mu$ is homomorphically equivalent to $\bB_\mu$ by \cite[Lemma 4.1.7]{Book}. Therefore, $\Gamma(\bH_\mu, \sigma)$ is homomorphically equivalent to $\Gamma_{\mu,\sigma}$ for any $\sigma \subseteq \tau$.

We now continue with general propositions that apply to all unions of conjunctive queries $\mu$. The following proposition follows straightforwardly from the definitions and provides a valued constraint satisfaction problem that is polynomial-time equivalent to the resilience problem for $\mu$, similar to Theorem~\ref{thm:fin-dual}. 

\begin{proposition}
\label{prop:connection}
The resilience problem for a 
union of connected conjunctive queries $\mu$
where the relations
from $\sigma \subseteq \tau$
are exogenous
is polynomial-time equivalent to 
$\VCSP(\Gamma(\bB,\sigma))$, for any dual $\bB$ of $\mu$; in particular, to  $\VCSP(\Gamma_{\mu,\sigma})$. 
\end{proposition}
\begin{proof}
Let $\bB$ be a dual of $\mu$.
For every  bag database $\bA$ over  signature $\tau$ and with exogenous relations from $\sigma$,
let $\phi$ 
be the 
$\tau$-expression 
obtained by adding atomic $\tau$-expressions $S(x_1,\dots,x_n)$ according to the multiplicity of the tuples $(x_1,\dots,x_n)$ in $S^{\bA}$ for all $S \in \tau$. 
Note that $\phi$ can be computed in polynomial time.
Then the resilience of $\bA$ with respect to $\mu$
is at most $u$ if and only if $(\phi,u)$ has a solution over 
$\Gamma(\bB, \sigma)$.

To prove a polynomial-time reduction in the other direction,
let $\phi$ be a $\tau$-expression. We construct a  bag database $\bA$ with signature $\tau$.
 The domain of $\bA$ are the variables that appear in $\phi$ and for every $S\in \tau$, we put a tuple $(x_1, \dots, x_n)$ in $S^{\bA}$ with a multiplicity
equal to the number of occurrences of $S(x_1, \dots x_n)$ 
as a summand  of $\phi$.
The relations $S^{\bA}$ with $S \in \sigma$ are exogenous in $\bA$, the remaining ones are endogenous. 
Again, $\bA$ can be computed in polynomial time and the resilience of $\bA$ with respect to $\mu$ is at most $u$ if and only if $(\phi,u)$ has a solution over $\Gamma(\bB, \sigma)$.
\end{proof}

\begin{remark} \label{rem:set}
The reason why we restrict to study of resilience problems in bag semantics and do not consider resilience problems in set semantics comes from the proof of Proposition~\ref{prop:connection}. To model resilience problems in set semantics, we would have to restrict the instances of the corresponding VCSP to instances without repeating summands. Little is known about the complexity of this variant of VCSPs and there is no algebraic theory describing the complexity of such problems~\cite{GeorgePersCom}. Consequently, the theory of constraint satisfaction problems does not provide tools to classify the complexity of resilience problems in set semantics; nevertheless, tractability results in bag semantics imply tractability results for the same queries in set semantics. 
\end{remark}

In~\cite{LatestResilience} one may find a seemingly more general notion of exogenous tuples, where in a single relation there might be both endogenous and exogenous tuples. However, one can show that classifying the complexity of resilience problems according to our original definition also entails a classification of this variant, using the operator $\Opt$ and a similar reduction as in Proposition~\ref{prop:connection}, as explained in the following remark.

\begin{remark} \label{rem:exo}
Consider a union $\mu$ of conjunctive queries with the signature $\tau$, let $\sigma \subseteq \tau$,  and let $\rho \subseteq \tau \setminus \sigma$. Suppose we would like to model the resilience problem for $\mu$ where the relations in $\sigma$ are exogenous and the relations in $\rho$ might contain both endogenous and exogenous tuples. Let $\bB$ be a dual of $\mu$ and $\Gamma$ be the expansion of $\Gamma(\bB, \sigma)$ where for every relational symbol $R \in \rho$, there is also a relation $(R^{x})^{\Gamma}=R^{\bB}$, i.e., a classical relation that takes values $0$ and $\infty$. The resilience problem for $\mu$ with exogenous tuples specified as above is polynomial-time equivalent to $\VCSP(\Gamma)$ by analogous reductions as in Proposition~\ref{prop:connection}.
Note that  $(R^{x})^{\Gamma}=\Opt\big (R^{\Gamma(\bB, \sigma)} \big)$ for every $R \in \rho$, and therefore by Lemma~\ref{lem:expr-reduce}, $\VCSP(\Gamma)$ is polynomial-time equivalent to $\VCSP(\Gamma(\bB,\sigma))$ and thus to the resilience problem for $\mu$ where the relations in $\sigma$ are exogenous and the relations in $\tau \setminus \sigma$ are purely endogenous. This justifies the restriction to our setting for exogenous tuples. Moreover, the same argument shows that if resilience of $\mu$ with all tuples endogenous is in P, then all variants of resilience of $\mu$ with exogenous tuples are in P as well.
\end{remark}

Similarly as in Example~\ref{expl:fin-dual-rev}, Proposition~\ref{prop:connection} can be combined with the tractability results for VCSPs from Section~\ref{sect:tract} that use fractional polymorphisms to prove tractability of resilience problems.
\begin{remark}\label{rem:bin-enc-3}
Note that if the multiplicities in the given database are stored in binary, then this corresponds to binary representations of multiplicities of summands in instances of the VCSP, and as we have mentioned earlier (Remark~\ref{rem:bin-enc} and Remark~\ref{rem:bin-enc-2}) many of the algorithmic results about VCSPs apply even in this setting.
\end{remark}

\begin{expl}\label{expl:simple-inf-rev}
We revisit Example~\ref{expl:simple-inf}. Consider the conjunctive query
$\exists x,y,z \, (R(x,y) \wedge S(x,y,z))$ over the signature $\tau = \{R,S\}$.
We observe that its resilience problem is in P. If $\bA$ is the input database, let $m$ be the multiplicity of $(a,b)$ in $R^\bA$ and let $m'$ be the sum of multiplicities of tuples of the form $(a,b,c)$ for some $c \in A$ in $S^\bA$. Then, if $m \leq m'$, we remove all copies of $(a,b)$ from $R^\bA$ and otherwise we remove all tuples of the form $(a,b,c)$ from $S^\bA$.

Note that the Gaifman graph of $\mu$ is complete; let $\bH_\mu$ be the structure from Theorem~\ref{thm:freeAP} and recall that it is finitely bounded and homogeneous. We construct a binary fractional polymorphism of $\Gamma(\bH_\mu, \emptyset)$ which is canonical and pseudo cyclic with respect to $\Aut(\Gamma(\bH_\mu, \emptyset)) = \Aut(\bH_\mu)$. 
Let $\bM$ be the $\tau$-structure with
domain $(H_\mu)^2$ 
and where 
\begin{itemize} 
\item $((b^1_1,b^2_1),(b^1_2,b^2_2)) \in R^{\bM}$ if 
$(b^1_1,b^1_2) \in R^{\bH_\mu}$ and 
$(b^2_1,b^2_2) \in R^{\bH_\mu}$ 
\item $((b^1_1,b^2_1),(b^1_2,b^2_2),(b^1_3,b^2_3)) \in S^{\bM}$ if
$(b^1_1,b^1_2,b^1_3) \in S^{\bH_\mu}$ or 
$(b^2_1,b^2_2,b^2_3) \in S^{\bH_\mu}$. 
\end{itemize}
Similarly, let $\bN$ be the $\tau$-structure with
domain $(H_\mu)^2$ 
and where 
\begin{itemize} 
\item $((b^1_1,b^2_1),(b^1_2,b^2_2)) \in R^{\bN}$ if 
$(b^1_1,b^1_2) \in R^{\bH_\mu}$ or 
$(b^2_1,b^2_2) \in R^{\bH_\mu}$,
\item $((b^1_1,b^2_1),(b^1_2,b^2_2),(b^1_3,b^2_3)) \in S^{\bN}$ if 
$(b^1_1,b^1_2,b^1_3) \in S^{\bH_\mu}$ and 
$(b^2_1,b^2_2,b^2_3) \in S^{\bH_\mu}$. 
\end{itemize}
Note that $\bM \not \models \mu$ and $\bN \not \models \mu$ and hence 
there are embeddings $f \colon \bM \to \bH_\mu$
and $g \colon \bN \to \bH_\mu$.
Clearly, both $f$ and $g$ regarded as operations on the set $H_\mu$ are pseudo cyclic (but in general not cyclic) and canonical with respect to $\Aut(\bH_\mu)$ (see Claim~6 in Proposition~\ref{prop:mu1} for a detailed argument of this type). Let $\omega$ be the fractional operation given by $\omega(f) = \frac{1}{2}$ and $\omega(g) = \frac{1}{2}$. 
Then $\omega$ is a binary fractional polymorphism of $\Gamma := \Gamma(\bH_\mu,\emptyset)$:
for $b^1,b^2 \in (H_\mu)^2$ we have
\begin{align}
\sum_{h \in {\mathscr O}^{(2)}} \omega(h) R^{\Gamma}(h(b^1,b^{2})) & = \frac{1}{2} R^{\Gamma}(f(b^1,b^2)) + \frac{1}{2} R^{\Gamma}(g(b^1,b^2)) \nonumber \\ 
& =
\frac{1}{2} \sum_{j = 1}^2 R^{\Gamma}(b^j). 
\label{eq:submod}
\end{align}
so $\omega$ improves $R$, and similarly we see that $\omega$ improves $S$. 

We proved that the corresponding valued structure has a binary canonical pseudo cyclic fractional polymorphism. By Theorem~\ref{thm:tract} and \ref{prop:connection}, this reproves the results from \cite{Resilience} (without multiplicities) and \cite{LatestResilience} (with multiplicities) that the resilience problem for this query is in P.
\end{expl}

We give another example of a conjunctive query with an infinite dual, which requires a ternary cyclic fractional polymorphism.

\begin{expl} \label{expl:double-edge}
Consider the conjunctive query     
\[\mu := \exists x,y \; (R(x,y) \wedge R(y,x))\]
and observe that its resilience problem is in P: if $\bA$ is the input database, and both $(a,b)$ and $(b,a)$ lie in $R^{\bA}$ with multiplicities $m$ and $m'$, respectively, we remove all copies of $(a,b)$ if $m \leq m'$ and otherwise we remove all copies of $(b,a)$.

Let $\bH_\mu$ be the homogeneous dual of $\mu$ that embeds every countable structure $\bA$ that does not satisfy $\mu$. Let let $\Gamma := \Gamma(\bH_\mu, \emptyset)$ and let $C$ be the domain of $\Gamma$.
We show that $\Gamma$ has a ternary canonical pseudo cyclic fractional polymorphism (with respect to $\Aut(\bH_\mu)$), which implies tractability of $\VCSP(\Gamma)$ (Theorem~\ref{thm:tract}) and hence of the resilience problem for $\mu$ (Proposition~\ref{prop:connection}).
To increase readability, we write
$R$ for $R^{\bH_\mu}$ and write
$\breve{R}$ for $\{(a,b) \mid (b,a) \in R^{\bH_\mu}\}$.
Note that since $\bH_\mu \models \neg \mu$, for every $a,b \in C$, we have precisely one of the following: $(a,b) \in R$, $(a,b) \in \breve{R}$, or $(a,b) \notin R \cup \breve{R}$.
Let $\bM$ be an $\{R\}$-structure with the domain $C^3$ such that 
\[\bM \models R((x,y,z), (u,v,w))\]
if and only if at least one of the following is true: 
\begin{enumerate}[(1)]
\item at least two of $(x,u), (y,v), (z,w)$ lie in $R$;
\item  $(x,u) \in R$, $(y,v) \in \breve{R}$, and $(z,w) \notin R \cup \breve{R}$;
\item 
$(x,u) \notin R \cup \breve{R}$,
$(y,v) \in R$, and $(z,w) \in \breve{R}$;
\item  $(x,u) \in \breve{R}$, $(y,v) \notin R \cup \breve{R}$, and $(z,w) \in R$.
\end{enumerate}
Note that items (2)-(4) are just cyclic shifts of the same condition.
It is straightforward to verify that $\bM \models \neg \mu$. For example, if $\bM \models R((x,y,z), (u,v,w))$ because of item (1), then at least two of $(u,x), (v,y), (w,z)$ lie in $\breve{R}$ and hence $\bM \models \neg R((u,v,w), (x,y,z))$ by definition. Therefore, there is an embedding $f$ of $\bM$ into $\bH_\mu$. By the definition of $\bM$, the operation $f$ is pseudo cyclic and canonical with respect to $\Aut(\bH_\mu)$
(this is easy to see, because $\bH_\mu$ is homogeneous; a detailed argument of this type can be found in Claim 6 in Proposition~\ref{prop:mu1}).
The idea is that $f$ has the behavior of a majority operation on orbits of pairs of distinct elements. 

Let $\bN$ be an $\{R\}$-structure with the domain $C^3$ such that 
\[\bN \models R((x,y,z), (u,v,w))\]
if and only if at least one of the following is true: 
\begin{enumerate}[(1)]
\setcounter{enumi}{4}
\item $(x,u), (y,v), (z,w) \in R$;
\item one of $(x,u), (y,v), (z,w)$ lies in $R$, and the remaining two lie in 
$\breve{R}$; 
\item one of $(x,u), (y,v), (z,w)$ lies in $R$, and the remaining two do not lie in $R \cup \breve{R}$; 
\item  $(x,u) \in \tilde{R}$, $(y,v) \in R$, and $(z,w) \notin R \cup \breve{R}$;
\item $(x,u) \notin R \cup \breve{R}$, $(y,v) \in \breve{R}$ and $(z,w) \in R$; 
\item  $(x,u) \in R$, $(y,v) \notin R \cup \breve{R}$ and $(z,w) \in \breve{R}$.
\end{enumerate}

Note that items (8)-(10) are just cyclic shifts of the same condition. Again one can verify that $\bN \models \neg \mu$, because $\bH_\mu \models \neg \mu$. Therefore, there is an embedding $g$ of $\bN$ into $\bH_\mu$. By the definition of $\bN$, the operation $g$ is pseudo cyclic and canonical with respect to $\Aut(\bH_\mu)$; it has the behavior of a minority operation on orbits of pairs of distinct elements. 

Let $\omega$ be the ternary fractional operation defined by $\omega(f)=2/3$ and $\omega(g)=1/3$. Note that $\omega$ is a pseudo cyclic and canonical ternary fractional operation on $C$. We show that $\omega \in \fPol(\Gamma)$.
Let $(x,u), (y,v), (z,w) \in C^2$. We want to verify that
\[E_\omega \left[h \mapsto R^\Gamma\left(h \left(\begin{pmatrix}x\\u\end{pmatrix},
\begin{pmatrix}y\\v\end{pmatrix}, \begin{pmatrix}z\\w\end{pmatrix} \right) \right) \right] \leq \frac{1}{3}(R^\Gamma(x,u)+R^\Gamma(y,v)+R^\Gamma(z,w)),\]
equivalently, 
\begin{align} \label{eq:maj_min}
2 R^\Gamma\left( f\left( \begin{pmatrix}x\\u\end{pmatrix}, \begin{pmatrix}y\\v\end{pmatrix}, \begin{pmatrix}z\\w\end{pmatrix} \right) \right) + 
R^\Gamma\left(g \left( \begin{pmatrix}x\\u\end{pmatrix}, \begin{pmatrix}y\\v\end{pmatrix}, \begin{pmatrix}z\\w\end{pmatrix} \right)\right)
\leq R^\Gamma(x,u)+R^\Gamma(y,v)+R^\Gamma(z,w).
\end{align}
We break into cases:
\begin{itemize}
    \item If $(x,u), (y,v), (z,w) \in R$, then by item (1) and (5) the left-hand side of \eqref{eq:maj_min} evaluates to 0 and hence \eqref{eq:maj_min} holds.
    \item If exactly two of $(x,y), (y,v), (z,w)$ lie in $R$, then by item (1) the left-hand side of \eqref{eq:maj_min} is at most 1 and hence \eqref{eq:maj_min} holds.
    \item If exactly one of $(x,y), (y,v), (z,w)$ lies in $R$, then precisely one of the conditions (2)-(4), (6)-(10) applies and therefore the left-hand side of \eqref{eq:maj_min} is at most 2. Therefore, \eqref{eq:maj_min} holds.
    \item If $(x,u), (y,v), (z,w) \notin R$, then \eqref{eq:maj_min} holds trivially since the left-hand side is always at most 3.
\end{itemize}
We conclude that $\omega \in \fPol(\Gamma)$.

We remark that there is no binary fractional polymorphism of $\Gamma$ that is canonical and pseudo cyclic with respect to $\Aut(\bH_\mu)$;  this is shown in detail in~\cite[Example 5.23]{ThesisZaneta}. 
\end{expl}

\subsection{The Resilience Tractability Conjecture}

In this section we present a conjecture which implies,  together with Corollary~\ref{cor:OIT} and Corollary~\ref{cor:con}, a P versus NP-complete dichotomy for resilience problems for finite unions of conjunctive queries. 

\begin{conjecture}\label{conj:tract}
Let $\mu$ be a union of connected conjunctive queries over the signature $\tau$, and let $\sigma \subseteq \tau$. If the structure $(\{0,1\};\OIT)$ 
has no pp-construction in $\Gamma_{\mu,\sigma}$,
then there exists a dual $\bB$ of $\mu$ such that $\Aut(\bB)$ contains an automorphism group of a finitely bounded homogeneous structure $\bA$ and 
$\Gamma := \Gamma(\bB, \sigma)$
has a fractional polymorphism of arity $\ell \geq 2$ which is canonical and pseudo cyclic with respect to $\Aut(\bA)$ 
(and in this case, $\VCSP(\Gamma)$ is in P by Theorem~\ref{thm:tract}). 
\end{conjecture}
Note that all duals $\bB$ of $\mu$ are homomorphically equivalent to $\bB_\mu$ (see Theorem~\ref{thm:css} for the definition of $\bB_\mu$), hence all valued structures $\Gamma(\bB, \sigma)$ are fractionally homomorphically equivalent to $\Gamma_{\mu,\sigma}$. Therefore, by the transitivity of pp-constructability,  $\Gamma_{\mu,\sigma}$ pp-constructs $(\{0,1\};\OIT)$ if and only if $\Gamma(\bB, \sigma)$ does. If P$\neq$NP, then there cannot be a union of queries $\mu$ such that $\Gamma_{\mu,\sigma}$ that pp-constructs $(\{0,1\};\OIT)$ and there is a dual $\bB$ of $\mu$ satisfying the assumptions from Conjecture~\ref{conj:tract}.
For all examples of resilience problems so far, we provided a dual $\bB$ that satisfies the condition from Conjecture~\ref{conj:tract} and thus proves the tractability of the resilience problem. In Example~\ref{expl:triad} below, we give an example of a resilience problem for which we prove hardness by pp-constructing $(\{0,1\};\OIT)$. 

Note that if $\bB$ is a relational structure on a finite domain $B$ and $\Gamma(\bB, \sigma)$ has a cyclic fractional polymorphism, then this polymorphism is trivially pseudo cyclic with respect to any group and it is canonical with respect to the trivial permutation group on $B$.  Let $\tau = \{R_b \mid b \in B\}$ be a relational signature where all symbols are unary. Note that the relational structure $\bA=(B; (R_b^{\bA})_{b \in B})$ where $R_b^{\bA} = \{b\}$ for every $b \in B$ has a trivial automorphism group and is homogeneous. It is also finitely bounded: a finite relational $\tau$-structure embeds into $\bA$ if and only if it satisfies
\[\forall x,y \left( \left( \bigvee_{b \in B} R_b(x) \right) \wedge \bigwedge_{b, c \in B, b \neq c} \big(\neg R_b(x) \vee \neg R_c (x) \big)  \wedge \bigwedge_{b \in B} (R_b(x) \wedge R_b(y) \Rightarrow x=y) \right).\]
Hence, every union of conjunctive queries each of which is homomorphically equivalent to a tree satisfies the condition from Conjecture~\ref{conj:tract} by Theorem~\ref{thm:fin-dual} and Theorem~\ref{thm:fin-vcsp-tract}.

Conjecture~\ref{conj:tract} is intentionally formulated only  for VCSPs that stem from resilience problems, because it is known to be false for the more general situation of VCSPs for valued structures $\Gamma$ that have the same automorphisms as a reduct of a finitely bounded homogeneous structure~\cite{Book} (Section 12.9.1; the counterexample is even a CSP). However, the structures $\bB_\mu$ from Theorem~\ref{thm:css} that allow to formulate
resilience problems as VCSPs are particularly well-behaved for the universal-algebraic approach and more specifically, for canonical operations (see, e.g.,~\cite{MMSNP,BodirskyBodorUIPJournal,MottetPinskerSmooth}), which is why we believe in the strong formulation of Conjecture~\ref{conj:tract}. 
See Conjecture~\ref{conj:VCSP} for a conjecture that could hold for VCSPs in the more general setting of reducts of finitely bounded homogeneous structures.

For the following conjunctive query $\mu$, the NP-hardness of the resilience problem without multiplicities was shown in~\cite{Resilience}; to illustrate our condition, we verify that the structure $(\{0,1\}; \OIT)$ has a pp-construction in $\Gamma_\mu$
and thus prove in a different way that the resilience problem (with multiplicities) for $\mu$ is NP-hard. 

\begin{expl}[Triangle query] \label{expl:triad}
 Let $\tau$ be the signature that consists of three binary relation symbols $R$, $S$, and $T$, and let $\mu$ be the conjunctive query 
$$\exists x,y,z \big (R(x,y) \wedge S(y,z) \wedge T(z,x) \big ).$$
The resilience problem without multiplicities for $\mu$  
is NP-complete~\cite{Resilience}, and hence $\VCSP(\Gamma_{\mu})$ is NP-hard (Proposition~\ref{prop:connection}). Since the Gaifman graph of $\mu$ is NP-complete,  the structure $\bH_\mu$ from Theorem~\ref{thm:freeAP} exists. Let $\Gamma:=\Gamma(\bH_\mu, \emptyset)$. 
We provide a pp-construction of $(\{0,1\};\OIT)$ in $\Gamma$, which also proves NP-hardness of $\VCSP(\Gamma)$ and hence the resilience problem of $\mu$ with multiplicities by Corollary~\ref{cor:OIT}. Since $\Gamma$ is homomorphically equivalent to $\Gamma_\mu$, this also provides a pp-construction of $(\{0,1\};\OIT)$ in $\Gamma_\mu$ (see Lemma~\ref{lem:pp-trans}).

Let $C$ be the domain of $\Gamma$. Denote the relations $\Opt(R)$, $\Opt(S)$, $\Opt(T)$ by $R^*$, $S^*$, $T^*$, respectively. In the following, for $U \in \{R,S,T\}$ and variables $x,y$ we write $2 U(x,y)$ for short instead of $U(x,y)+U(x,y)$.
Let $\phi(a,b,c,d,e,f,g,h,i)$ be the pp-expression
\begin{align}
& R(a,b) + 2 S(b,c) + 2 T(c,d) + 2 R(d,e)    
+ 2 S(e,f) + 2 T(f,g) + 2 R(g,h) + S(h,i) \label{eq:firstgroup}\\
+ \; & T^*(i,g) + S^*(h,f) + R^*(g,e) + T^*(f,d)  
+ S^*(e,c) + R^*(d,b) + T^*(c,a). \label{eq:secondgroup}
\end{align}
For an illustration of $\mu$ and $\phi$, see Figure~\ref{fig:triad}.
Note that $\phi$ can be viewed as $7$ non-overlapping copies of $\mu$ (if we consider the doubled constraints as two separate constraints) with some constraints forbidden to violate.

\begin{figure*}
    \centering 
    \includegraphics[width=\textwidth]{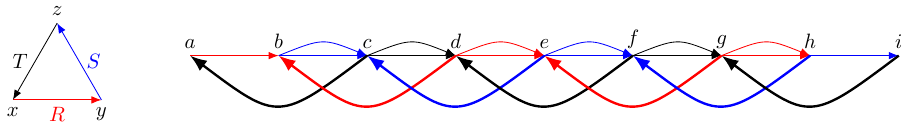}
    \caption{Example~\ref{expl:triad}, visualization of $\mu$ and $\phi$. The thick edges cannot be removed.}
    \label{fig:triad}
    \Description{Arrow depiction of the query and the pp-expression from the example, with thick arrows symbolizing the crisp constraints.}
\end{figure*}

In what follows, we say that an atomic $\tau$-expression 
holds if it evaluates to $0$ and an atomic $\tau$-expression is violated if it does not hold.
Since there are $7$ non-overlapping copies of $\mu$ in $\phi$, the cost of $\phi$ is at least $7$.
Every assignment where
\begin{itemize}
\item all atoms in~\eqref{eq:secondgroup} hold, 
and 
\item either 
every atom at even position or every atom at odd position in~\eqref{eq:firstgroup} holds,
\end{itemize}
evaluates $\phi$ to $7$ and hence is a solution to $\phi$.

Let $RT \in \langle \Gamma \rangle$ be given by
\[RT(a,b,f,g) := \Opt \inf_{c,d,e,h,i \in C} \phi.\]
Note that $RT(a,b,f,g)$ holds if and only if
\begin{itemize}
\item $R(a,b)$ holds and $T(f,g)$ does not hold, or
\item $T(f,g)$ holds and $R(a,b)$ does not hold,
\end{itemize}
where the reverse implication uses that $\bH_\mu$ is homogeneous and embeds all finite structures that do not satisfy $\mu$.
Define $RS\in \langle \Gamma \rangle$ by  
\[RS(a,b,h,i) := \Opt \inf_{c,d,e,f,g \in C} \phi.\]
Note that $RS(a,b,h,i)$
holds if and only if

\begin{itemize}
\item $R(a,b)$ holds and $S(h,i)$ does not hold, or
\item $S(h,i)$ holds and $R(a,b)$ does not hold. 
\end{itemize}
Next, we define the auxiliary relation $\widetilde{RS}(a,b,e,f)$ to be
\[ \Opt \inf_{c,d,g,h,i \in C} \phi. \]
Note that $\widetilde{RS}(a,b,e,f)$ holds if and only if
\begin{itemize}
\item both $R(a,b)$ and $S(e,f)$ hold, or
\item neither $R(a,b)$ and nor $S(e,f)$ holds. 
\end{itemize}
This allows us to define the relation
\[RR(u,v,x,y):= \inf_{w,z \in C} RS(u,v,w,z) + \widetilde{RS}(x,y,w,z)\] 
which holds if and only if 
\begin{itemize}
\item $R(u,v)$ holds and $R(x,y)$ does not hold, or
\item $R(x,y)$ holds and $R(u,v)$ does not hold. 
\end{itemize}
Define $M \in \langle \Gamma \rangle$ as
\begin{align*}
M(u,v,u',v',u'',v''):= \Opt \inf_{x,y,z \in C} \big (R(x,y) + S(y,z) + T(z,x) + \; RR(u,v,x,y) + RS(u',v',y,z) + RT(u'',v'',z,x) \big ).
\end{align*}
Note that $R(x,y)$, $S(y,z)$ and $T(z,x)$ cannot hold at the same time and therefore $(u,v,u',v',u'',v'') \in M$ if and only if exactly one of of 
$R(u,v)$, $R(u',v')$, and $R(u'',v'')$ holds. 
Let $\Delta$ be the pp-power of $(C;M)$
of dimension two with signature $\{\OIT\}$ such that $$\OIT^{\Delta}((u,v),(u',v'),(u'',v'')) := M(u,v,u',v',u'',v'').$$ 
Then $\Delta$ is homomorphically equivalent to 
$(\{0,1\};\OIT)$, witnessed by the homomorphism
from $\Delta$ to $(\{0,1\};\OIT)$ that maps
$(u,v)$ to $1$ if $R(u,v)$ and to $0$ otherwise,
and the homomorphism 
$(\{0,1\};\OIT) \to \Delta$ that maps $1$ to any pair of vertices $(u,v) \in R$
and $0$ to any pair of vertices $(u,v) \notin R$. Therefore, $\Gamma$ pp-constructs $(\{0,1\}; \OIT)$.
\end{expl}

We mention that another conjecture concerning a P vs.\ NP-complete complexity dichotomy for resilience problems appears in \cite[Conjecture 7.7]{LatestResilience}. The conjecture has a similar form as Conjecture~\ref{conj:tract} in the sense that it states that a sufficient hardness condition for resilience is also necessary. The relationship between our hardness condition from Corollary~\ref{cor:OIT} and the condition from \cite{LatestResilience} remains open.

\subsection{An example of formerly open complexity}\label{sect:expl}

We use our approach to settle the complexity of the resilience
problem for a conjunctive query that was mentioned as an open problem in~\cite{NewResilience} (Section 8.5):
\begin{align}
    \mu & := \exists x,y (S(x) \wedge R(x,y) \wedge R(y,x) \wedge R(y,y)) \label{eq:mu1}
\end{align}
Let $\tau = \{R,S\}$ be the signature of $\mu$. 
To study the complexity of resilience of $\mu$, it will be convenient to work with 
a dual which has different 
model-theoretic properties than the duals $\bB_\mu$ from Theorem~\ref{thm:css} and $\bH_\mu$ from Theorem~\ref{thm:freeAP}, namely a dual that is a model-complete core.

\begin{definition}\label{def:mc-core}
A structure
$\bB$ with an oligomorphic automorphism group is \emph{model-complete} if every embedding of $\bB$ into $\bB$ preserves all first-order formulas. It is a \emph{core} if every endomorphism is an embedding.
\end{definition}

Note that the definition of cores of valued structures with finite domain (Definition~\ref{def:core}) and the definition above specialize to the same concept for relational structures over finite domains. 
A structure with an oligomorphic automorphism group is a model-complete core if and only if for every $n \in {\mathbb N}$ every orbit of $n$-tuples 
can be defined with an existential positive formula~\cite{Book}.
Every countable structure $\bB$ is homomorphically equivalent to a model-complete core, which is unique up to isomorphism \cite{Cores-journal, Book}; we refer to this structure as the model-complete core of $\bB$.
The advantage of working with model-complete cores is that the structure is in a sense `minimal' and therefore easier to work with in concrete examples.\footnote{The model-complete core of $\bB_\mu$ would be a natural choice for the canonical dual of $\mu$ to work with instead of~$\bB_\mu$. However, proving that the model-complete core has a finitely bounded homogeneous expansion (so that, for example, Theorem~\ref{thm:fb-NP} applies) requires introducing further model-theoretical notions~\cite{MottetPinskerCores} which we want to avoid in this article.} 

\begin{figure}
    \centering
    \includegraphics[]{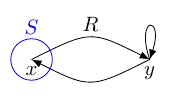}
    \caption{Visualization of the query $\mu$ from \eqref{eq:mu1}.}
    \label{fig:open}
    \Description{Arrow depiction of the query.}
\end{figure}

\begin{proposition}\label{prop:mu1}
There is a finitely bounded homogeneous dual $\bB$ of $\mu$ such that the valued $\tau$-structure 
$\Gamma := \Gamma(\bB, \emptyset)$ has a binary fractional polymorphism which is canonical and pseudo cyclic with respect to $\Aut(\bB)= \Aut(\Gamma)$. Hence, $\VCSP(\Gamma)$ and the resilience problem for $\mu$ are in P. The polynomial-time tractability result even holds for
resilience of $\mu$ with exogenous relations from any $\sigma \subseteq \tau$.
\end{proposition}

The rest of this section is devoted to the proof of Proposition~\ref{prop:mu1}.
\begin{figure}
    \centering
    \includegraphics[]{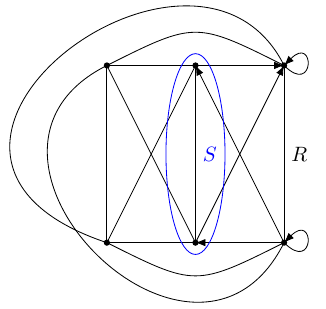}
    \caption{An illustration of a finite substructure of $\bB$ from the proof of Proposition~\ref{prop:mu1} that contains representatives for all orbits of pairs of $\Aut(\bB)$. Arrows are not drawn on undirected edges.}
    \label{fig:dual-mu1}
    \Description{Arrow depiction of all possible orbits of pairs of the automorphism group. Contains all possibilities of how relations S and R may be imposed on a pair of vertices.}
\end{figure}
Since the Gaifman graph of $\mu$ is a complete graph, there exists the structure $\bH_\mu$ as in  Theorem~\ref{thm:freeAP}.
Let $\bB$ be the model-complete core of $\bH_\mu$. Note that $\bB$ has the property that a countable structure $\bA$ maps homomorphically to $\bB$ if and only if $\bA \models \neg \mu$; in particular, $\bB$ is a dual of $\mu$ and $\bB \models \neg \mu$.
The structure $\bH_\mu$ is homogeneous, and it is known that the model-complete core of a homogeneous structure is again homogeneous (see Proposition 4.7.7 in~\cite{Book}), so $\bB$ is homogeneous. Let $\Gamma := \Gamma(\bB, \emptyset)$.

Note that 
\begin{align}
\bB & \models \forall x \big ( \neg S(x) \vee \neg R(x,x) \big ) \label{eq:parti} \\
  \text{ and } \bB & \models \forall x,y \big (x = y \vee R(x,y) \vee R(y,x) \big ).
    \label{eq:total}
\end{align}
    To see~\eqref{eq:total}, suppose for contradiction that $\bB$ contains distinct elements $x,y$ such that neither $(x,y)$ nor $(y,x)$ is in $R^{\bB}$. 
Let $\bB'$ be the structure obtained from $\bB$ by adding $(x,y)$ to $R^{\bB}$. Then $\bB' \models \neg \mu$ as well, and hence there is a homomorphism from $\bB'$ to $\bB$ by the properties of $\bB$. This homomorphism is also an endomorphism of $\bB$ which is not an embedding,
a contradiction to the assumption that $\bB$ is a model-complete core. 

Also observe that 
\begin{equation} \bB \models \forall x,y  \big(x= y \vee (R(x,y) \wedge R(y,x)) \vee  (S(x) \wedge R(y,y)) \vee (R(x,x) \wedge S(y)) \big ). \label{eq:double}
\end{equation}
Suppose for contradiction that \eqref{eq:double} does not hold and let $x$ and $y$ be witnesses, in particular, $x \neq y$.
Then $\neg S(x) \vee \neg R(y,y)$ 
and $\neg R(x,x) \vee \neg S(y)$, i.e.,
$\neg S(x) \wedge \neg R(x,x)$, 
or $\neg S(x) \wedge \neg S(y)$, 
or $\neg R(y,y) \wedge \neg R(x,x)$, 
or $\neg R(y,y) \wedge \neg S(y)$.
In each of these cases 
we may add both $R$-edges between the distinct elements $x$ and $y$ to $\bB$ and obtain a structure not satisfying $\mu$, which leads to a contradiction as above.

For an illustration of a finite substructure of $\bB$ which contains a representative for every orbit of pairs in $\Aut(\bB)$; see Figure~\ref{fig:dual-mu1}.
A homomorphism is called \emph{strong} 
if it also preserves the complements of all relations (note that an injective strong homomorphism is an embedding).

\begin{claim}\label{cl:1}
For every finite $\tau$-structure $\bA$ that satisfies $\neg \mu$ and the sentences in~\eqref{eq:total} and \eqref{eq:double}, there exists a strong homomorphism to $\bB$. 
\end{claim}

\begin{proof}
We first observe that $\bB$ embeds 
the relational $\tau$-structure $\bA'$ on the domain $A' = \N \cup \{a\} \cup \{b\}$, where $a,b \notin \N$ are distinct, $S^{\bA'} = \{a\}$ and 
\[R^{\bA'} = \{(b,b)\} \cup \{(x,y) \mid x, y \in A', x \neq y\} \setminus \{(a,b)\}  .\]
To see this, first note that
there exists a homomorphism $h \colon \bA' \to \bB$, because 
$\bA' \models \neg \mu$.
If $h$ is not an embedding, then $h$ is not a strong homomorphism or it is not injective. Note that if $\bB \models S(h(x))$ for $x \in \N \cup \{b\}$, then $\bB \models \mu$, because $h$ is a homomorphism. Similarly we get $\bB \models \mu$ if $\bB \models R(h(a), h(b))$ or $\bB \models R(h(x),h(x))$ for $x \in \N \cup \{a\}$. Finally, if $h(x)=h(y)$ for $x,y \in A'$, we can again verify that the substructure induced by $h(A')$ in $\bB$ satisfies $\mu$. In each of these cases, we get a contradiction with $\bB$ being a dual of $\mu$. 
Therefore, $\bA'$ embeds into $\bB$. 
In particular, there are infinitely many $x \in B$ such that $\bB \models \neg S(x) \wedge \neg R(x,x)$ and by \eqref{eq:double}, for every $y \in B$, $x \neq y$, we have $\bB \models R(x,y) \wedge R(y,x)$.

To prove the claim, let $\bA$ be a finite structure that satisfies $\neg \mu$ and the sentences in~\eqref{eq:total} and \eqref{eq:double}. For a homomorphism $h$ from $\bA$ to $\bB$, let 
\[s(h) := |\{x \in A \mid \bA \models \neg S(x) \wedge \bB \models S(h(x)) \} | \]
and 
\[r(h) := |\{(x,y) \in A^2 \mid \bA \models \neg R(x,y) \wedge \bB \models R(h(x), h(y)) \}|.\]
Let $h$ be a homomorphism from $\bA$ to $\bB$, which exists since $\bA \models \neg \mu$. 
If $s(h)+r(h)=0$, then $h$ is a strong homomorphism and there is nothing to prove. Suppose therefore $s(h)+r(h)>0$, equivalently, $s(h)>0$ or $r(h)>0$. We construct a homomorphism $h'$ such that $r(h')+s(h')<r(h)+s(h)$. Since $r(h)+s(h)$ is finite, by applying this construction finitely many times, we obtain a strong homomorphism from $\bA$ to $\bB$.

First suppose that $s(h)>0$. Then there exists $a \in A \setminus S^{\bA}$ such that $h(a) \in S^{\bB}$. By \eqref{eq:parti}, $\bB \not \models R(h(a),h(a))$ and hence $\bA \not \models R(a,a)$. Pick $b \in B \setminus h(A)$ such that $\bB \models \neg S(b) \wedge \neg R(b,b)$ (recall that there are infinitely many such $b \in B$) and define
\[
h'(x):=
\begin{cases}
b \text{ if }x=a,\\
h(x) \text{ otherwise}.
\end{cases}
\]
Observe that $h'$ is a homomorphism, $s(h')<s(h)$ and $r(h')=r(h)$.

Second suppose that $r(h)>0$. Then there exists $(x,y) \in A^2 \setminus R^{\bA}$ such that $(h(x),h(y)) \in R^{\bB}$. If $x=y$, the argument is similar as in the case $s(h)>0$. Finally, if $x \neq y$, then $\bA \models (S(x) \wedge R(y,y)) \vee (R(x,x) \wedge S(y))$, because $\bA$ satisfies the sentence in \eqref{eq:double}. Since $\bA$ satisfies the sentence in \eqref{eq:total}, $\bA \models R(y,x)$. Since $h$ is a homomorphism, we have
\begin{align*}
\bB \models R(h(x), h(y)) \wedge R(h(y), h(x)) 
\wedge ((S(h(x)) \wedge R(h(y),h(y))) \vee (R(h(x),h(x)) \wedge S(h(y)))),
\end{align*}
which contradicts $\bB \not \models \mu$. 
\end{proof}

\begin{claim}\label{cl:2}
Every finite $\tau$-structure $\bA$ that satisfies $\neg \mu$ and the universal sentences in~\eqref{eq:total} and \eqref{eq:double} embeds into $\bB$. In particular, $\bB$ is finitely bounded.
\end{claim}

\begin{proof}
Let $\bA$ be such a structure. By Theorem~\ref{thm:freeAP}, there is an embedding $e$ of $\bA$ into $\bH_\mu$. 
Since $\bH_\mu$ is homogeneous and embeds every finite $\tau$-structure that satisfies $\neg \mu$, there exists a finite substructure $\bA'$ of $\bH_\mu$ satisfying the sentences in \eqref{eq:total} and \eqref{eq:double} such that $e(\bA)$ is a substructure of $\bA'$ and for all distinct $a, b \in A$ there exists $s \in S^{\bA'}$ such that $\bH_\mu \models R(e(a),s) \wedge R(s,e(b))$. By Claim~\ref{cl:1}, there is a strong homomorphism $h$ from $\bA'$ to $\bB$.

We claim that $h \circ e$ is injective and therefore an embedding of $\bA$ into $\bB$. Suppose there exist distinct $a,b \in A$ such that $h(e(a))=h(e(b))$. Since $e(\bA)$ satisfies the sentence in~\eqref{eq:total},
we have $\bB \models R(h(e(a)), h(e(a)))$.
Let $s \in S^{\bA'}$ be such that $\bH_\mu \models R(e(a),s) \wedge R(s,e(b))$. Hence,
\begin{align*}
\bB \models S(h(s)) \wedge R(h(e(a)), h(s))
\wedge R(h(s), h(e(a))) \wedge R(h(e(a)), h(e(a))),
\end{align*}
a contradiction to $\bB \not \models \mu$. It follows that $h \circ e$ is an embedding of $\bA$ into $\bB$.
\end{proof}

We define two $\{R,S\}$-structures $\bM,\bN$ with domain $B^2$ as follows. 
 For all $x_1,x_2,y_1,y_2,x,y \in B$ define 
 \begin{align}
 \bM,\bN & \models R \big ((x_1,y_1),(x_2,y_2) \big) \label{eq:pres-R}
 & \text{ if } \bB & \models R(x_1,x_2) \wedge R(y_1,y_2), \\
 \bM,\bN & \models S \big ((x,y) \big) & \text{ if } \bB & \models S(x) \wedge S(y) \label{eq:pres-S} 
 \\
 \bM & \models S \big ((x,y) \big ) & \text{ if } \bB & \models S(x) \vee S(y) 
 \label{eq:M-max} \\
    \bN & \models R \big ((x,y),(x,y) \big ) & \text{ if } 
 \bB & \models R(x,x) \vee R(y,y) .
 \label{eq:N-max}
\end{align}
Add pairs of distinct elements to $R^{\bM}$ and $R^{\bN}$ such that both $\bM$ and $\bN$ satisfy the sentence in~\eqref{eq:double} (note that no addition of elements to $S^{\bM}$ and $S^{\bN}$ is needed).
Finally, add $((x_1,y_1),(x_2,y_2))$ to $R^{\bM}$
and $((x_2,y_2),(x_1,y_1))$ to $R^{\bN}$ 
 \label{eq:M-key}
 if at least one 
 of the following cases holds:
\begin{enumerate}
\item [(A)] 
$\bB \models S(x_1) \wedge R(x_1,x_2) \wedge R(x_2,x_2) \wedge R(y_2, y_2) \wedge R(y_2,y_1) \wedge S(y_1)$, 
\item [(B)]
$\bB \models R(x_1,x_1) \wedge R(x_1,x_2) \wedge S(x_2) \wedge y_1 = y_2 \wedge R(y_1,y_2)$, 
\item [(C)] 
$\bB \models S(y_1) \wedge R(y_1,y_2) \wedge R(y_2,y_2) \wedge R(x_2,x_2) \wedge R(x_2,x_1) \wedge S(x_1)$,
\item [(D)] $\bB \models R(y_1,y_1) \wedge R(y_1,y_2) \wedge S(y_2) \wedge x_1 = x_2 \wedge R(x_1,x_2)$. 
\end{enumerate}
Conditions (A) and (B) are illustrated in Figure~\ref{fig:mu1-ABCD}; conditions (C) and (D) are obtained from (A) and (B) by replacing $x$ by $y$. Note that for $(x_1, y_1)=(x_2,y_2)$, none of the conditions (A)-(D) is ever satisfied. No other atomic formulas hold on $\bM$ and $\bN$. Note that both $\bM$ and $\bN$ satisfy the property stated for $\bB$ in~\eqref{eq:parti}.

\begin{figure*}
    \centering
    \includegraphics[]{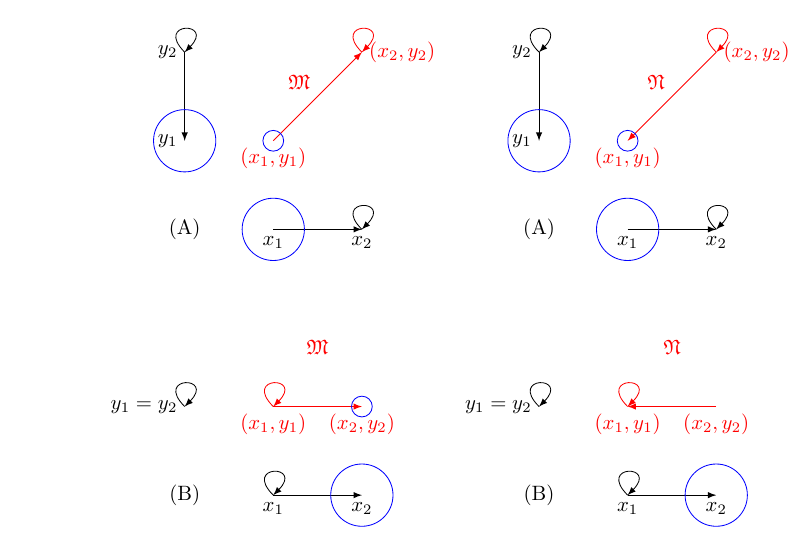}
    \caption{An illustration of the conditions (A) and (B) in $\bM$ and $\bN$.}
    \label{fig:mu1-ABCD}
    \Description{Arrow depiction of conditions (A)-(D), always showing the conditions for (x1, x2) on the horizontal axis, the conditions for (y1, y2) on the vertical axis and the resulting situation in M or N. }
\end{figure*}

\begin{claim} \label{cl:3}
$\bM$ and $\bN$ satisfy the sentence in \eqref{eq:total}.
\end{claim}

\begin{proof}
We prove the statement for $\bM$; the proof for $\bN$ is similar. Let $(x_1,y_1), (x_2, y_2) \in B$ be such that $(x_1, y_1) \neq (x_2, y_2)$ and $\bM \models 
\neg R((x_2, y_2), (x_1, y_1))$. Since $\bM$ satisfies the sentence in \eqref{eq:double}, we must have either that $\bM \models S(x_1, y_1) \wedge R((x_2, y_2), (x_2, y_2))$ or that $\bM \models S(x_2, y_2) \wedge R((x_1, y_1), (x_1, y_1))$. Suppose the former is true; the other case is treated analogously. Then $\bB \models R(x_2, x_2) \wedge R(y_2, y_2)$ and $\bB \models S(x_1) \vee S(y_1)$. If $\bB \models S(x_1)$, then $x_1 \neq x_2$ and by \eqref{eq:total} we have $\bB \models R(x_1,x_2) \vee R(x_2, x_1)$. By \eqref{eq:total} and \eqref{eq:double} for $(y_1,y_2)$, we obtain that $\bM \models R((x_1, y_1), (x_2, y_2))$ by \eqref{eq:pres-R} or one of the conditions (A)-(D). The argument if $\bB \models S(y_1)$ is similar with $x$ and $y$ switched.
\end{proof}

\begin{claim} \label{cl:4}
$\bM$ and $\bN$ satisfy $\neg \mu$. Let
$x_1,x_2,y_1,y_2 \in B$.
\end{claim}

\begin{proof}
Suppose for contradiction that
\begin{align*}
\bM \models S(x_1,y_1) \wedge R((x_1,y_1),(x_2,y_2))
\wedge R((x_2,y_2),(x_1,y_1)) \wedge R((x_2,y_2),(x_2,y_2)).  
\end{align*}
By the definition of $\bM$, we have $\bB \models R(x_2,x_2) \wedge R(y_2,y_2)$ and $\bB \models S(x_1) \vee S(y_1)$.
Assume that $\bB \models S(x_1)$; the case $\bB \models S(y_1)$ is analogous.

By the assumption, $\bM \models R((x_1, y_1),(x_2,y_2))$. Then, by the definition of $\bM$, one of the conditions \eqref{eq:pres-R}, (A)-(D) holds, or
\[\bM \models \neg \big (S(x_1,y_1) \wedge R((x_2,y_2),(x_2,y_2)) \big) \]
(recall that $((x_1,y_1),(x_2,y_2))$ might have been added to $R^{\bM}$ so that $\bM$ satisfies the sentence in \eqref{eq:double}). The last option is false by the assumption and by \eqref{eq:parti}, $\bB \models \neg S(x_2) \wedge \neg S(y_2)$, and hence neither (B) nor (D) holds.  Therefore, one of the conditions \eqref{eq:pres-R}, (A), or (C) holds for $((x_1,y_1), (x_2,y_2))$. Similarly, we obtain that one of the conditions \eqref{eq:pres-R} or (B) holds for $((x_2,y_2),(x_1,y_1))$, since $\bM \models R((x_2, y_2),(x_1,y_1))$ (to exclude (D) we use the assumption that $\bB \models S(x_1)$ and hence $x_1 \neq x_2$). This yields six cases and in each of them
we must have that $\bB \models R(x_1,x_2) \wedge R(x_2,x_1)$ or $\bB \models S(y_1) \wedge R(y_1, y_2) \wedge R(y_2,y_1)$. Since $\bB \models S(x_1) \wedge R(x_2,x_2) \wedge R(y_2,y_2)$, this contradicts $\bB \models \neg \mu$. 
Since $(x_1,y_1),(x_2,y_2) \in M$ were chosen arbitrarily, this shows that $\bM \models \neg \mu$.
The argument for $\bN$ is similar. 
\end{proof}

\begin{claim} \label{cl:5}
There is an embedding $f$ of $\bM$ into $\bB$ and an embedding $g$ of $\bN$ into $\bB$.
\end{claim}

\begin{proof}
We show the claim for $\bM$; the proof for $\bN$ is analogous. By \cite[Lemma 4.1.7]{Book}, it is enough to show that every finite substructure of $\bM$ embeds into $\bB$. By the definition of $\bM$ and Claims ~\ref{cl:3} and~\ref{cl:4}, every finite substructure $\bM$ satisfies \eqref{eq:total}, \eqref{eq:double} and $\neg \mu$ and hence, by Claim~\ref{cl:2}, it embeds into $\bB$.
\end{proof}

Let $\omega$ be the fractional operation over $B$ defined by $\omega(f) = \frac{1}{2}$ and $\omega(g) = \frac{1}{2}$. 

\begin{claim} \label{cl:6}
$\omega$ is pseudo cyclic and canonical with respect to the group $\Aut(\bB)=\Aut(\Gamma)$.
\end{claim}

\begin{proof}
Note that since $\bB$ is homogeneous in a finite relational signature, two $k$-tuples of elements of $\bB$ lie in the same orbit if and only if they satisfy the same atomic formulas. 
Therefore, the canonicity of $f$ and $g$ with respect to $\Aut(\bB)$ follows from the definition of $\bM$ and $\bN$:
for $(a,b) \in B^2$, whether $\bB \models S(f(a,b))$ only depends on whether $\bM \models S(a,b)$ by Claim~\ref{cl:5}, which depends only on the atomic formulas that hold on $a$ and on $b$ in $\bB$. An analogous statement is true for atomic formulas of the form $R(x,y)$ and $x=y$.
Therefore, $f$ is canonical.
The argument for the canonicity of $g$ is analogous.

To see that $f$ and $g$ are pseudo cyclic, we show that $f^*$ and $g^*$ defined on $2$-orbits (using the terminology of Remark~\ref{rem:can-act}) are cyclic. By the definition of $f^*$, we need to show that for any $a_1, a_2, b_1, b_2 \in B$, the two pairs $(f(a_1, b_1), f(a_2,b_2))$ and $(f(b_1, a_1), f(b_2,a_2))$ satisfy the same atomic formulas. For the formulas of the form $S(x)$ and $R(x,y)$, 
this can be seen from Claim~\ref{cl:5} and the definition of $\bM$ and $\bN$, 
since each of the conditions~\eqref{eq:pres-R},\eqref{eq:pres-S},\eqref{eq:M-max},\eqref{eq:N-max},\eqref{eq:double} and the union of (A), (B), (C), (D) is symmetric with respect to exchanging $x$ and $y$. For the atomic formulas of the form $x=y$, this follows from the injectivity of $f$. This shows that $f^*$ is cyclic; the argument for $g^*$ is the same. Hence, the pseudo-cyclicity of $f$ and $g$ is a consequence of Lemma~\ref{lem:pc} for $m = 2$.
\end{proof}

\begin{claim} \label{cl:7}
$\omega$ improves $S$.
\end{claim}

\begin{proof}
By the definition of $\bM$ and $\bN$ and Claim~\ref{cl:5}, we have for all $x,y \in B$ 
$$\omega(f)S^{\Gamma}(f(x,y)) + \omega(g)S^{\Gamma}(g(x,y)) = \frac{1}{2} (S^{\Gamma}(x) + S^{\Gamma}(y)).$$
\end{proof}

\begin{claim} \label{cl:8}
$\omega$ improves $R$.
\end{claim}

\begin{proof}
Let $x_1,y_1,x_2,y_2 \in B$. 
We have to verify that 
\begin{align}
\omega(f) R^{\Gamma}(f(x_1,y_1),f(x_2,y_2)) + \omega(g) R^{\Gamma}(g(x_1,y_1),g(x_2,y_2)) \leq \frac{1}{2} (R^{\Gamma}(x_1,x_2) + R^{\Gamma}(y_1,y_2)).
    \label{eq:pres-expl}
\end{align}

We distinguish four cases. 
\begin{itemize}
    \item $\bM,\bN \models R((x_1,y_1),(x_2,y_2))$. 
Then Inequality~\eqref{eq:pres-expl} holds since the left-hand side is zero, and the right-hand side is non-negative (each valued relation in $\Gamma$ is non-negative). 
\item $\bM,\bN \models \neg R((x_1,y_1),(x_2,y_2))$.
We need to show that $\bB \models \neg R(x_1,x_2) \wedge \neg R(y_1,y_2)$. This is clear if $(x_1,y_1)=(x_2,y_2)$ by the definition of $\bN$. Suppose therefore that $(x_1,y_1)\neq(x_2,y_2)$.
Since $\bM$ satisfies the sentence in~\eqref{eq:double}, we have $\bM 
\models S(x_1, y_1) \wedge R((x_2, y_2), (x_2,y_2))$ or $\bM 
\models S(x_2, y_2) \wedge R((x_1, y_1), (x_1,y_1))$. Suppose that $\bM \models S(x_1, y_1) \wedge R((x_2, y_2), (x_2,y_2))$; the other case is analogous.
Since $\bN$ satisfies the sentence in~\eqref{eq:double} as well, this implies that $\bN \models S(x_1, y_1) \wedge R((x_2, y_2), (x_2,y_2))$; note that if $\bN \models S(x_2, y_2) \wedge R((x_1, y_1), (x_1,y_1))$, we would get a contradiction with~\eqref{eq:parti}.
By the definition of $\bM$ and $\bN$, we have $\bB \models S(x_1) \wedge S(y_1) \wedge R(x_2, x_2) \wedge R(y_2,y_2)$, in particular, by~\eqref{eq:parti}, $x_1 \neq x_2$ and $y_1 \neq y_2$. By~\eqref{eq:total}, there is an $R$-edge in $\bB$ between $x_1$ and $x_2$ and between $y_1$ and $y_2$. By the condition (A) and (C) for $\bM$ and for $\bN$ with $1$ and $2$ switched, we see that $\bM,\bN \models \neg R((x_1,y_1),(x_2,y_2))$ implies $\bB \models \neg R(x_1,x_2) \wedge \neg R(y_1,y_2)$.
Therefore, both sides of the inequality evaluate to $1$.
\item $\bM \models \neg R((x_1,y_1),(x_2,y_2))$ and $\bN \models R((x_1,y_1),(x_2,y_2))$. By Claim~\ref{cl:5}, the left-hand side evaluates to $\frac{1}{2}$. By \eqref{eq:pres-R}, we have $\bB \models \neg R(x_1, x_2)$ or $\bB \models \neg R(y_1, y_2)$. Therefore, the right-hand side of \eqref{eq:pres-expl} is at least $\frac{1}{2}$ and the inequality holds. 
\item $\bM \models R((x_1,y_1),(x_2,y_2))$ and $\bN \models \neg R((x_1,y_1),(x_2,y_2))$. 
Similar to the previous case. 
\end{itemize}

This exhausts all cases and concludes the proof of Claim~\ref{cl:8}. 
\end{proof}

It follows that $\omega$ is a binary fractional polymorphism of $\Gamma$ which is canonical and pseudo cyclic with respect to $\Aut(\bB)= \Aut(\Gamma)$. Polynomial-time tractability of $\VCSP(\Gamma)$ follows 
 by Theorem~\ref{thm:tract}
 and \ref{prop:connection}. 
 The final statement follows from Remark~\ref{rem:exo}. This finishes the proof of Proposition~\ref{prop:mu1}.

\section{Conclusion and Future Work}
\label{sect:concl}

We formulated a general hardness condition for VCSPs of valued structures with an oligomorphic automorphism group and  a new  polynomial-time tractability result. We use the latter to 
resolve the resilience problem for a conjunctive query whose complexity was left open in the literature and conjecture that our conditions exactly capture
the hard and easy resilience problems for conjunctive queries (under bag semantics), respectively. In fact, a full classification of resilience problems for conjunctive queries based on our approach seems feasible, but requires further research, as discussed in the following.

We have proved that if $\Gamma$ is a valued structure with an oligomorphic automorphism group and $R$ is a valued relation in the smallest valued relational clone that contains the valued relations of $\Gamma$, then
$R$ is preserved by all fractional polymorphisms of $\Gamma$ (Lemma~\ref{lem:easy}). We do not know whether the converse is true. It is known to hold for 
the special cases
of finite-domain valued structures~\cite{CohenCooperJeavonsVCSP,FullaZivny}
and for classical relational structures with $0$-$\infty$ valued relations (CSP setting) having an oligomorphic automorphism group~\cite{BodirskyNesetril}.

\begin{question}\label{ques:pres}
    Let $\Gamma$ be a valued structure with an oligomorphic automorphism group. Is it true that $R \in \langle \Gamma \rangle$ if and only if $R \in \Imp(\fPol(\Gamma))$? 
\end{question}

A natural attempt to positively answer Question~\ref{ques:pres} would be to combine the proof strategy for finite-domain valued structures from~\cite{CohenCooperJeavonsVCSP,FullaZivny}
with the one for relational structures with oligomorphic automorphism group from~\cite{BodirskyNesetril}. However, since non-improving of $R$ is not a closed condition, 
the compactness argument from~\cite{BodirskyNesetril} cannot be used to construct an operation from $\fPol(\Gamma)$ that does not improve $R$.
A positive answer to Question~\ref{ques:pres} would imply that the computational complexity of VCSPs for valued structures $\Gamma$ with an oligomorphic automorphism group, and in particular the complexity of resilience problems, is fully determined by the fractional polymorphisms of $\Gamma$.

Note that in all examples that arise from resilience problems that we considered so far, it was sufficient to work with fractional polymorphisms $\omega$ 
that are
\emph{finitary}, i.e., there are finitely many operations $f_1,\dots,f_k \in {\mathscr O}_C$ such that $\sum_{i \in \{1,\dots,k\}} \omega(f_i) = 1$. 
It is therefore possible that all fractional polymorphisms relevant for resilience have discrete probability distributions.
This motivates the following question.

\begin{question}\label{question:int-vs-fin-supp}
\begin{itemize}
\item[]%
\item Does our notion of pp-constructability change if we restrict to finitary fractional homomorphisms $\omega$?
\item Is there a valued structure $\Gamma$ with an oligomorphic automorphism group and a valued relation $R$ such that $R$ is
not improved by all fractional polymorphism of $\Gamma$, but \emph{is} 
improved by all finitary fractional polymorphisms $\omega$? 
\item In particular, 
are these statements 
true if we restrict to valued $\tau$-structures $\Gamma$ that arise from resilience problems as described in Proposition~\ref{prop:connection}?
\end{itemize}
\end{question}

In the following, we formulate a common generalization of the complexity-theoretic implications of Conjecture~\ref{conj:tract} and the infinite-domain tractability conjecture~\cite[Conjecture 1.2]{BPP-projective-homomorphisms} that concerns a full complexity classification of VCSPs for valued structures from reducts of finitely bounded homogeneous structures.
\begin{conjecture}\label{conj:VCSP}
Let $\Gamma$ be a valued structure with finite signature such that $\Aut(\Gamma) = \Aut(\bB)$
for some reduct $\bB$ of a countable finitely bounded homogeneous structure. If $(\{0,1\};\OIT)$ has no pp-construction in $\Gamma$, then $\VCSP(\Gamma)$ is in P (otherwise, we already know that $\VCSP(\Gamma)$ is NP-complete by Theorem~\ref{thm:fb-NP} and Corollary~\ref{cor:OIT}). 
\end{conjecture}

One might hope to prove this conjecture  under the assumption of the infinite-domain tractability conjecture. Recall that also the finite-domain VCSP classification was first proven conditionally on the finite-domain tractability conjecture~\cite{KolmogorovKR17,KozikOchremiak15}, which was only confirmed later~\cite{ZhukFVConjecture,BulatovFVConjecture}. Note that Conjecture~\ref{conj:VCSP} does not generalize Conjecture~\ref{conj:tract}, because it does not provide an algebraic condition for tractability. However, Conjecture~\ref{conj:VCSP} would imply a P vs. NP-complete dichotomy for resilience problems. 

We also believe that the `meta-problem' of deciding whether for a given conjunctive query  the resilience problem is in P is decidable. This might follow from a positive answer to Conjecture~\ref{conj:tract} if an appropriate $\bB$ can be found effectively, because then $\Gamma^*_{\bB,m}$ can be computed effectively and Item 4 of Proposition~\ref{prop:black-box} 
for the finite-domain valued structure $\Gamma^*_{\bB,m}$ can be decided using linear programming~\cite{Kolmogorov-Meta}.

\begin{acks}
The first two authors have been funded by the European Research Council (Project POCOCOP, ERC Synergy Grant 101071674) and by the DFG (Project FinHom, Grant 467967530). Views and opinions expressed are however those of the authors only and do not necessarily reflect those of the European Union or the European Research Council Executive Agency. Neither the European Union nor the granting authority can be held responsible for them. The third author was supported by the DFG project LU 1417/3-1 QTEC. This research was funded in whole or in part by the Austrian Science Fund (FWF) 10.55776/ESP6949724.

The authors thank Antoine Mottet for pointing out a gap in the definition of composition of fractional maps, and Friedrich Martin Schneider for suggesting an elegant formalization of this  composition.
\end{acks}

\bibliographystyle{ACM-Reference-Format}
\bibliography{global}

\appendix
\section{The Lebesgue Integral} \label{sect:lebesgue}
It will be convenient to use an additional value $-\infty$ that has the usual properties:
\begin{itemize}
    \item $-\infty < a$ for every $a \in \mathbb{R} \cup \{\infty\}$,
    \item $a +(-\infty) = (-\infty) + a= - \infty$ for every $a \in \mathbb{R}$,
    \item $a \cdot \infty = \infty \cdot a = -\infty$ for $a<0$,
    \item $0\cdot (-\infty) = (-\infty)\cdot 0=0$.
    \item $a\cdot (-\infty)=(-\infty) \cdot a = -\infty$ for $a > 0$ and $a \cdot (-\infty)= (-\infty)\cdot a = \infty$ for $a < 0$.
\end{itemize}
The sum of $\infty$ and $-\infty$ is undefined.

Let $C$ and $D$ be sets. We define the Lebesgue integration over the space $C^D$ of all functions from $D$ to $C$. We usually (but not always) work with the special case $D = C^{\ell}$, i.e. the space is $\mathscr O_C^{(\ell)}$ for some set $C$ and $\ell\in \N$.

To define the Lebesgue integral, 
we need the definition of a \emph{simple function}:
this is a function $Y \colon C^D \to {\mathbb R}$ given by
$$\sum_{k=1}^n a_k 1_{S_k}$$
where $n \in \N$, $S_1,S_2,\dots$ are disjoint elements of $B(C^D)$, $a_k \in {\mathbb R}$, and $1_S \colon C^D \to \{0,1\}$ denotes the indicator function for $S \subseteq C^D$.
If $Y$ is a such a simple function, 
then the Lebesgue integral is defined as follows:
$$\int_{C^D} Y d \omega := \sum_{k =1}^n a_k \omega(S_k).$$
If $X$ and $Y$ are two random variables, then we write $X \leq Y$ if $X(f) \leq Y(f)$ for every $f \in C^D$. We say that $X$ is \emph{non-negative} if $0 \leq X$. 
If $X$ is a non-negative measurable function, then the Lebesgue integral is defined as 
$$\int_{C^D} X d \omega := \sup \left\{ \int_{C^D} Y d \omega \mid 0 \leq Y \leq X, Y \text{ simple} \right\}.$$
For an arbitrary measurable function $X$, we write $X = X^+ - X^-$, 
where $$X^+(x) := \begin{cases} X(x) & \text{if } X(x) > 0 \\
0 & \text{otherwise} \end{cases}$$
and 
$$X^-(x) := \begin{cases} -X(x) & \text{if } X(x) < 0 \\
0 & \text{otherwise.} \end{cases}$$
Then both $X^+$ and $X^-$ are measurable, and both 
$\int_{C^D} X^- d \omega$ 
and 
$\int_{C^D} X^+ d \omega$ take values in ${\mathbb R}_{\geq 0} \cup \{\infty\}$. 
If both take value $\infty$, then the integral is undefined (see Remark~\ref{rem:undef}). 
Otherwise, define 
$$\int_{C^D} X d \omega := \int_{C^D} X^+ d \omega - \int_{C^D} X^- d \omega.$$
In particular, note that for $X \geq 0$ the integral is always defined.

Let $\omega$ be a fractional map from $D$ to $C$, let $R \in {\mathscr R}^{(k)}_{C}$ be a valued relation, and let $s \in D^k$. Then $X \colon C^D \rightarrow {\mathbb R} \cup \{\infty\}$ given by 
$$ f \mapsto R(f(s))$$
is a random variable: 
if $(a,b)$ is a basic open subset of ${\mathbb R} \cup \{\infty\}$, then
\begin{align*}
    X^{-1}((a,b)) 
    & = \{f \in C^D \mid R(f(s)) \in (a,b) \} \\
    & = \bigcup_{t \in C^k,R(t) \in (a,b)}{\mathscr S}_{s,t}
\end{align*}
is a union of basic open sets in $C^D$, hence open. The argument for the other basic open sets in ${\mathbb R} \cup \{\infty\}$ is similar. 

If the set $C$ is countable and $X$ is as above, we may express $E_{\omega}[X]$ as a sum, which is useful in proofs in Sections~\ref{sect:gen-hard} and~\ref{sect:fpol}.
If $E_{\omega}[X]$ exists,
then
\begin{align}
E_{\omega}[X] & = \int_{C^D} X^+ d \omega - \int_{C^D} X^- d \omega \nonumber \\
& = \sup \left \{ \int_{C^D} Y d \omega \mid 0 \leq Y \leq X^+, Y \text{ simple} \right \} - \sup \left \{ \int_{C^D} Y d \omega \mid 0 \leq Y \leq X^-, Y \text{ simple} \right \} \nonumber\\
& = \sum_{t \in C^k,R(t) \geq 0} R(t) \omega(
{\mathscr S}_{s,t})   
+ \sum_{t \in C^k,R(t) < 0} R(t) \omega(
{\mathscr S}_{s,t}) \nonumber \\
& = \sum_{t \in C^k} R(t) \omega(
{\mathscr S}_{s,t}).
\label{eq:exp-sum}
\end{align}

\begin{remark}\label{rem:undef}
The Lebesgue integral 
$$\int_{{\mathscr O}^{(\ell)}_C} X d \omega = \int_{{\mathscr O}^{(\ell)}_C} X^+ d \omega - \int_{{\mathscr O}^{(\ell)}_C} X^- d \omega $$
need not exist: 
e.g., consider $C = {\mathbb N}$, $k = \ell = 1$, and $R(x) = -2^x$ if $x \in {\mathbb N}$ is even and $R(x) = 2^x$ otherwise. Let $s \in C$ and define 
$X \colon {\mathscr O}^{(1)}_C \to {\mathbb R} \cup \{\infty\}$ by 
$$ f \mapsto R(f(s)).$$

Let $\omega$ be a unary fractional operation such that for every $t \in C$
we have 
$\omega({\mathscr S}_{s,t}) = \frac{1}{2^{t+1}}$. Then 
\begin{align*}
\int_{{\mathscr O}^{(\ell)}_C} X^+ d \omega & = \sup \{ \int_{{\mathscr O}^{(\ell)}_C} Y d \omega \mid 0 \leq Y \leq X^+, Y \text{ simple} \} \\
& = \sum_{t \in C,R(t) \geq 0} R(t) \omega(
{\mathscr S}_{s,t}) \\
& = \sum_{t \in C, R(t) \geq 0} \frac{1}{2} = \infty
\end{align*}
and, similarly, $\int_{{\mathscr O}^{(\ell)}_C} X^- d \omega
= \infty$. 
\end{remark}

\end{document}